\documentclass[a4paper,11pt,reqno,twoside]{amsart}
\usepackage[pdftex]{graphicx}
\usepackage{subfigure}
\usepackage{color}
\usepackage{cite}
\usepackage{ifpdf}
\usepackage{array}
\usepackage{slashbox}
\usepackage[T1]{fontenc}
\usepackage[latin1]{inputenc}
\usepackage{mathtools}
\usepackage{amsthm,amssymb}
\usepackage{url}

\newcommand*{\K}{\mathbb{K}}
\newcommand*{\A}{\mathbb{A}}
\newcommand*{\La}{\mathsf{L}}
\newcommand*{\DLa}{\mathsf{L}_\pm}
\newcommand*{\LaCa}{\mathsf{LC}}
\newcommand*{\Me}{\mathsf{M}}
\newcommand*{\MeCa}{\mathsf{MC}}
\DeclarePairedDelimiter\abs{\lvert}{\rvert}
\newcommand*{\C}{\mathbb{C}}
\newcommand*{\dd}{%
  \mathop{\mathrm{d}\null}\mskip-\thinmuskip\mathord{\null}} 
\renewcommand*{\epsilon}{\varepsilon}
\newcommand*{\fGamma}{\mathop{\Gamma}\nolimits}
\renewcommand*{\geq}{\geqslant}
\newcommand*{\gGamma}{\varGamma}
\let\oldIm=\Im
\renewcommand*{\Im}{\mathop{\oldIm\mkern-2mu 
m}\nolimits}
\renewcommand*{\leq}{\leqslant}
\newcommand*{\NiceE}{\mathrm{Nice-Ell}\null}
\newcommand*{\NiceWE}{\mathrm{Nice^{\flat}-Ell}\null}
\newcommand*{\NiceHW}{\mathrm{Nice-HW}\null}
\newcommand*{\NiceH}{\mathrm{Nice-H}\null}
\newcommand*{\MPHW}{\mathrm{Mp-HW}\null}
\newcommand*{\MPH}{\mathrm{Mp-H}\null}
\newcommand*{\Q}{\mathbb{Q}}
\newcommand*{\R}{\mathbb{R}}
\let\oldRe=\Re
\renewcommand*{\Re}{\mathop{\oldRe\mkern-2mu e}\nolimits}

\newcommand*{\Schwartz}{\mathcal{S}}
\newcommand*{\Z}{\mathbb{Z}}
\makeatletter
 \newlength{\h@uteurnumerateur}
 \newlength{\h@uteurdenominateur}
\newcommand*{\quotientdroite}[2]{
  \mathchoice%
  {
    \settoheight{\h@uteurnumerateur}{\ensuremath{\displaystyle{#1#2}}}%
    \settoheight{\h@uteurdenominateur}{\ensuremath{\displaystyle{#1#2}}}%
    \raisebox{0.5\h@uteurnumerateur}{\ensuremath{\displaystyle{#1}}}%
    \mkern-5mu\diagup\mkern-4mu%
    \raisebox{-0.5\h@uteurdenominateur}{\ensuremath{\displaystyle{#2}}}%
  }
  {
    \settoheight{\h@uteurnumerateur}{\ensuremath{\textstyle{#1#2}}}%
    \settoheight{\h@uteurdenominateur}{\ensuremath{\textstyle{#1#2}}}%
  \raisebox{0.2\h@uteurnumerateur}{\ensuremath{\textstyle{#1}}}%
  /%
  \raisebox{-0.2\h@uteurdenominateur}{\ensuremath{\textstyle{#2}}}%
  }
  {
    \settoheight{\h@uteurnumerateur}{\ensuremath{\scriptstyle{#1#2}}}%
    \settoheight{\h@uteurdenominateur}{\ensuremath{\scriptstyle{#1#2}}}%
  \raisebox{0.2\h@uteurnumerateur}{\ensuremath{\scriptstyle{#1}}}%
  /%
  \raisebox{-0.2\h@uteurdenominateur}{\ensuremath{\scriptstyle{#2}}}%
  }
  {
    
\settoheight{\h@uteurnumerateur}{\ensuremath{\scriptscriptstyle{#1#2}}}%
    
\settoheight{\h@uteurdenominateur}{\ensuremath{\scriptscriptstyle{#1#2}}}%
  \raisebox{0.2\h@uteurnumerateur}{\ensuremath{\scriptscriptstyle{#1}}}%
  /%
  \raisebox{-0.2\h@uteurdenominateur}{\ensuremath{\scriptscriptstyle{#2}}}%
  }
}
\newcommand*{\quotientgauche}[2]{
  \mathchoice%
  {
    \settoheight{\h@uteurnumerateur}{\ensuremath{\displaystyle{#1#2}}}%
    \settoheight{\h@uteurdenominateur}{\ensuremath{\displaystyle{#1#2}}}%
    \raisebox{-0.5\h@uteurnumerateur}{\ensuremath{\displaystyle{#1}}}%
    \mkern-3mu\diagdown\mkern-5mu%
    \raisebox{0.5\h@uteurdenominateur}{\ensuremath{\displaystyle{#2}}}%
  }
  {
    \settoheight{\h@uteurnumerateur}{\ensuremath{\textstyle{#1#2}}}%
    \settoheight{\h@uteurdenominateur}{\ensuremath{\textstyle{#1#2}}}%
  \raisebox{-0.2\h@uteurnumerateur}{\ensuremath{\textstyle{#1}}}%
  \backslash%
  \raisebox{0.2\h@uteurdenominateur}{\ensuremath{\textstyle{#2}}}%
  }
  {
    \settoheight{\h@uteurnumerateur}{\ensuremath{\scriptstyle{#1#2}}}%
    \settoheight{\h@uteurdenominateur}{\ensuremath{\scriptstyle{#1#2}}}%
  \raisebox{-0.2\h@uteurnumerateur}{\ensuremath{\scriptstyle{#1}}}%
  \backslash%
  \raisebox{0.2\h@uteurdenominateur}{\ensuremath{\scriptstyle{#2}}}%
  }
  {
    
\settoheight{\h@uteurnumerateur}{\ensuremath{\scriptscriptstyle{#1#2}}}%
    
\settoheight{\h@uteurdenominateur}{\ensuremath{\scriptscriptstyle{#1#2}}}%
  \raisebox{-0.2\h@uteurnumerateur}{\ensuremath{\scriptscriptstyle{#1}}}%
  \backslash%
  \raisebox{0.2\h@uteurdenominateur}{\ensuremath{\scriptscriptstyle{#2}}}%
  }
}
\makeatother
\newtheoremstyle{erdfn}
  {}
  {}
  {\itshape}
  {}
  {\sffamily\itshape}
  {--}
  { }
  {}
\newtheoremstyle{erthm}
  {}
  {}
  {\itshape}
  {}
  {\sffamily\bfseries\itshape}
  {--}
  { }
  {}
\newtheoremstyle{errem}
  {}
  {}
  {}
  {}
  {\sffamily\itshape}
  {--}
  { }
  {}
\theoremstyle{erthm}
\newtheorem{theorem}{Theorem}[section] 
\newtheorem{corollary}[theorem]{Corollary} 
\newtheorem{proposition}[theorem]{Proposition}
\newtheorem{lemma}[theorem]{Lemma}

\newtheorem*{hypothesisE}{Hypothesis $\NiceE(\K)$}

\newtheorem*{hypothesisHW}{Hypothesis $\NiceHW(\K)$}
\newtheorem*{hypothesisH}{Hypothesis $\NiceH(\K)$}
\newtheorem*{hypothesisMPHW}{Hypothesis $\MPHW(\K)$} 
\newtheorem*{hypothesisMPH}{Hypothesis $\MPH(\K)$} 
\theoremstyle{erdfn}
\newtheorem{definition}[theorem]{Definition}
\newtheorem*{notations}{Notation}
\newtheorem*{merci}{Acknowledgements}
\theoremstyle{errem}
\newtheorem{remark}[theorem]{Remark}

\numberwithin{equation}{section}

\title[Mean-periodicity and zeta functions]%
      {Mean-periodicity and zeta functions} 
\author[M. Suzuki]{Masatoshi Suzuki}
\author[G. Ricotta]{Guillaume Ricotta}
\author[I. Fesenko]{Ivan Fesenko}

\AtBeginDocument{%
\mathtoolsset{showonlyrefs,mathic = true}
\begin{abstract}
This paper establishes new bridges between number theory and modern harmonic analysis, namely between the class of complex functions, which contains zeta functions of arithmetic schemes and closed with respect to product and quotient, and the class of mean-periodic functions in several spaces of functions on the real line. In particular, the meromorphic continuation and  functional equation of the  zeta function of arithmetic scheme with its expected analytic shape is shown to  correspond to   mean-periodicity of a certain explicitly defined function associated to the zeta function. 
This correspondence can be viewed as an extension of the Hecke--Weil correspondence. 
 The case of elliptic curves over number fields and their regular models is treated in more details, and many other examples are included as well.
\end{abstract}
\maketitle
\tableofcontents
}

\begin{document}
\begin{merci}
The authors would like to thank A. Borichev for his enlightening lectures on mean-periodic functions in Nottingham.
We are grateful to him and to  D.~Goldfeld, N. Kuro\-kawa, 
J. Lagarias, K. Matsumoto for their remarks and discussions. 
This work was partially supported by ESPRC grant EP/E049109. The first author would like to thank the JSPS for its support. The second author is partially financed by the ANR project "Aspects Arithm\'etiques des Matrices Al\'eatoires et du Chaos Quantique".
\end{merci}
\section{Introduction}%
%

The Hecke--Weil correspondence between modular forms and Dirichlet series plays a pivotal role in number theory. 
In this paper we propose a new correspondence between mean-periodic functions in several spaces of functions on the real line and a class of  Dirichlet series which admit meromorphic continuation and functional equation. 
This class includes arithmetic zeta functions and their products, quotients, derivatives. 
In the first approximation the new  correspondence  can be stated as this: 
the rescaled completed  zeta-function of arithmetic scheme $S$  
has meromorphic continuation of expected analytic shape and satisfies 
the  functional equation $s \to  1-s$  with sign $\epsilon$
if and only if  the function 
$f(\exp(-t))-\varepsilon  \exp(t)f(\exp(t))$ is a mean-periodic function in 
the space of smooth functions on the real line of exponential growth,
where  $f(x)$ is the inverse Mellin transform of the product of an appropriate
sufficiently large positive power of the completed Riemann zeta function and 
the rescaled completed zeta function. For precise statements see  Subsection \ref{zetamp} below and other subsections of this Introduction. 

Whereas modular forms is a classical object investigated for  more than 150 years, 
the theory of mean-periodic functions  is a relatively recent part of functional analysis 
whose relations to  arithmetic zeta functions is studied and described in this paper. 
The text  presents the analytic aspects of the new correspondence  in accessible
and relatively self-contained  form. It is expected that the modern harmonic analysis and mean-periodic functions will have many applications to the study of zeta functions
of arithmetic schemes.

\subsection{Boundary terms in the classical one-dimensional case}\label{subsec_classical}%
The completed zeta function of a number field $\K$ is defined on $\Re{(s)}>1$ by
\begin{equation*}
\widehat{\zeta}_\K(s)\coloneqq \zeta_{\infty,\K}(s)\zeta_K(s)
\end{equation*}
where $\zeta_\K(s)$ is the classical Dedekind zeta function of $\K$ 
defined using the Euler product over all maximal ideal of the ring of integers of $\K$ 
and $\zeta_{\K,\infty}(s)$ is a finite product of $\fGamma$-factors defined in \eqref{eq_zeta_infinity}. It satisfies the integral representation (see \cite{MR0217026})
\begin{equation*}
\widehat{\zeta}_\K(s)=\int_{\mathbb{A}_\K^\times}f(x)\abs{x}^s\dd\mu_{\mathbb{A}_\K^\times}(x)
\end{equation*}
where $f$ is an appropriately normalized function in the Schwartz--Bruhat space on $\mathbb{A}_\K$ and $\abs{\;\;}$ stands for the module on the ideles $\mathbb{A}_\K^\times$ of $\K$. An application of analytic duality on $\K\subset \mathbb{A}_\K$ leads to the decomposition
\begin{equation*}
\widehat{\zeta}_\K(s)=\xi\left(f,s\right)+\xi\left(\widehat{f},1-s\right)+\omega_f(s)
\end{equation*}
where $\widehat{f}$ is the Fourier transform of $f$, $\xi(f,s)$ is an entire function and
\begin{equation*}
\omega_f(s)=\int_{0}^1h_f(x)x^s\frac{\dd x}{x}
\end{equation*}
with
\begin{equation*}
h_f(x)\coloneqq -\int_{\gamma\in\mathbb{A}_\K^1\left/\K^\times\right.}\int_{\beta\in\partial\K^\times}\left(f(x\gamma\beta)-x^{-1}\widehat{f}\left(x^{-1}\gamma\beta\right)\right)\dd\mu(\beta)\dd\mu(\gamma).
\end{equation*}
Every continuous function on $\mathbb{A}_\K$ which vanishes on $\K^\times$ vanishes on $\K$, and the boundary $\partial\K^\times$ of $\K^\times$ is $\K \setminus \K^\times=\{0\}$. The meromorphic continuation and the functional equation for $\widehat{\zeta}_\K(s)$ are equivalent to the meromorphic continuation and the functional equation for $\omega_f(s)$. Let us remark that $\omega_f(s)$ is the Laplace transform of $H_f(t)\coloneqq h_f\left(e^{-t}\right)$ thanks to the change of variable $x=e^{-t}$. The properties of the functions $h_f(x)$ and $H_f(t)$, which are called the \emph{boundary terms} for obvious reason,
are crucial in order to have a better understanding of $\omega_f(s)$. We have 
\begin{eqnarray*}
h_f(x) & = & -\mu\left(\mathbb{A}_\K^1\left/\K^\times\right.\right)\left(f(0)-x^{-1}\widehat{f}(0)\right), \\
H_f(t) & = & -\mu\left(\mathbb{A}_\K^1\left/\K^\times\right.\right)\left(f(0)-e^{t}\widehat{f}(0)\right)
\end{eqnarray*}
since  $\partial\K^\times$ is just the single point $0$, with  the appropriately normalized measure on the idele class group. 
As a consequence, $\omega_f(s)$ is a rational function of $s$ invariant with respect to $f\mapsto\widehat{f}$ and $s\mapsto(1-s)$. Thus, $\widehat{\zeta}_\K(s)$ admits a meromorphic continuation to $\C$ and satisfies a functional equation with respect to $s\mapsto(1-s)$.%
\subsection{Mean-periodicity and analytic properties of Laplace transforms}%
The previous discussion in the one-dimensional classical case naturally leads to the analytic study of Laplace transforms of specific functions. In particular, the meromorphic continuation and functional equation for Mellin transforms of real-valued functions $f$ on $\R_+^\times$ of rapid decay at $+\infty$ and polynomial order at $0^+$ is equivalent to the meromorphic continuation and functional equation of
\begin{equation*}
\omega_{f}(s)\coloneqq \int_0^1h_{f}(x)x^s\frac{\dd x}{x}=\int_0^{+\infty}H_f(t)e^{-st}\dd t
\end{equation*}
with
\begin{equation*}
h_{f}(x)=f(x)-\epsilon x^{-1}f\left(x^{-1}\right)
\end{equation*}
and $H_f(t)=h_{f}\left(e^{-t}\right)$, where  $\epsilon=\pm 1$ is the sign of the expected functional equation (see Section \ref{sec_mbp0}). The functions $h_{f}(x)$ and $H_f(t)$ are called the \emph{boundary terms} by analogy. 

The main question is the following one. \textit{What property of $h_f(x)$ or $H_f(t)$ implies the meromorphic continuation and functional equation for $\omega_{f}(s)$?} One sufficient answer is \textit{mean-periodicity}\footnote{The general theory of mean-periodicity is recalled in Section \ref{sec_mp}.}. Mean-periodicity is an easy generalization of periodicity; a function $g$ of a functional space $X$ is $X$-mean-periodic if the space spanned by its translates is not dense in $X$. When the Hahn-Banach theorem is available in $X$, $g$ is $X$-mean-periodic if and only if $g$ satisfies a \emph{convolution equation} $g\ast\varphi=0$ for some non-trivial element $\varphi$ of the dual space $X^\ast$ and a suitable convolution $\ast$. Such a convolution equation may be thought as a generalized differential equation. Often\footnote{In this case one says that the \emph{spectral synthesis} holds in $X$ (see Section \ref{spectral_synthesis}).} mean-periodic functions are  limits of \emph{exponential polynomials} satisfying the same convolution equation. Exponential polynomials were used already by Euler in his method of solving ordinary differential equations with constant coefficients. 

The functional spaces $X$ which will be useful  for number theory purposes are
\begin{equation*}
\mathfrak{X}_\times=\begin{cases}
\mathcal{C}({\R_+^\times}) & \text{the continuous functions on $\R_+^\times$ (Section \ref{ssec_C})}, \\
\mathcal{C}^\infty({\R_+^\times}) & \text{the smooth functions on $\R_+^\times$ (Section \ref{ssec_C})}, \\
\mathcal{C}_{\rm poly}^\infty({\R_+^\times}) & \text{the smooth functions on $\R_+^\times$ of at most polynomial growth (see \eqref{eq_poly-growth})}
\end{cases}
\end{equation*}
in the multiplicative setting, which is related to $h_f(x)$ and
\begin{equation*}
\mathfrak{X}_+=\begin{cases}
\mathcal{C}({\R}) & \text{the continuous functions on $\R$ (Section \ref{ssec_C})}, \\
\mathcal{C}^\infty({\R}) & \text{the smooth functions on $\R$ (Section \ref{ssec_C})}, \\
\mathcal{C}_{\rm exp}^\infty({\R}) & \text{the smooth functions on $\R$ of at most exponential growth (see \eqref{eq_exp-growth})}
\end{cases}
\end{equation*}
in the additive setting, which is related to $H_f(t)$. 
A nice feature is that the spectral synthesis holds in both $\mathfrak{X}_+$ and $\mathfrak{X}_\times$. 
The general theory of $\mathfrak{X}_\times$-mean-periodic functions shows that if $h_f(x)$ is $\mathfrak{X}_\times$-mean-periodic then $\omega_{f}(s)$ has a meromorphic continuation given by the \emph{Mellin--Carleman transform} of $h_f(x)$ and satisfies a functional equation. Similarly, if $H_f(t)$ is $\mathfrak{X}_+$-mean-periodic then $\omega_{f}(s)$ has a meromorphic continuation given by the \emph{Laplace--Carleman transform} of $H_f(t)$ and satisfies a functional equation. See Theorem \ref{thm_rmp} for an accurate statement. Note that in general $\omega_{f}(s)$ can have a meromorphic continuation to $\C$ and can satisfy a functional equation without the functions $h_f(x)$ and $H_f(t)$ being mean-periodic (see Remark \ref{rem_nmp} for explicit examples).
%
\subsection{Arithmetic  zeta functions and higher dimensional adelic analysis}\label{arzeta}%
For a scheme $S$  of dimension $n$ its arithmetic (Hasse)  \emph{zeta function} 
\begin{equation*}
\zeta_S(s)\coloneqq\prod_{x\in S_0} (1-\abs{k(x)}^{-s})^{-1}
\end{equation*}
whose Euler factors correspond to all closed points $x$ of $S$, say $x\in S_0$, with finite residue field of cardinality $\abs{k(x)}$, is the most fundamental object in number theory. Not much  is known about it when $n>1$.

Higher dimensional adelic analysis  aims to study the  zeta functions $\zeta_S(s)$ using integral representations on higher adelic spaces and analytic duality. It employs geometric structures of regular model of elliptic curve
which are difficult to see directly at purely analytic level. 
It is expected that the $n$-th power of the completed versions of the zeta functions $\zeta_S(s)$ times a product of appropriately completed and rescaled lower dimensional zeta functions can be written as an adelic integral over an appropriate higher dimensional adelic space against an appropriate translation invariant measure. Then a procedure similar to the one-dimensional procedure given above leads to the decomposition of the completed zeta functions into the sum of two entire functions and another term, which in characteristic zero is of the type
\begin{equation*}
\omega_{S}(s)\coloneqq\int_{0}^1h_S(x)x^s\frac{\dd x}{x}=\int_0^{+\infty}H_S(t)e^{-st}\dd t
\end{equation*}
where $h_S(x)$ and $H_S(t)\coloneqq h_S(e^{-t})$ are called the \emph{boundary terms} for the following reason. The functions $h_S(x)$ are expected to be an integral over the boundary of some higher dimensional space over some suitably normalised measure. Let us mention that the structure of both the boundary and the measure is quite mysterious. In particular, the boundary is expected to be a very large object, which is totally different from the one-dimensional situation. 
For the case of arithmetic surfaces $\mathcal E$ corresponding to a regular model of elliptic curve over a global field see \cite{Fe3}, \cite{Fe2}.

\subsection{Boundary terms of  zeta functions and mean-periodicity}\label{zetamp}
The papers \cite{Fe2}, \cite{Fe3} suggested to use the theory of property of mean-periodicity of the functions $h_S(x)$ in an appropriate functional space for the study of the meromorphic continuation and the functional equation of the  zeta functions $\zeta_S(s)$. 
This work demonstrates novel important links between the world of certain Dirichlet series which include the  zeta functions $\zeta_S(s)$ coming from number theory and the class  of mean-periodic functions in appropriate functional spaces, in a self-contained manner independently of higher dimensional adelic analysis. 

Let us give a flavour of these links (see Theorem \ref{schemess} for a precise statement). 

Let $S$ be an arithmetic scheme 
proper flat over ${\rm Spec} \, \Z$
with smooth  generic fibre. 
We  prove that if its zeta function $\zeta_S(s)$ extends to a meromorphic function
on the complex plane with special (and typical in number theory)
analytic shape, and satisfies a functional equation with sign $\epsilon$,
 then there exists an integer $m_{\zeta_S}\geq 1$ such that for every integer $m\geq m_{\zeta_S}$  the (ample) boundary term $h_{\zeta_S,m}(x)$, given by
\begin{equation*}
h_{\zeta_S,m}(x)\coloneqq f_{\zeta_S,m}(x)-\epsilon x^{-1}f_{\zeta_S,m}(x^{-1})
\end{equation*}
where $f_{\zeta_S,m}(x)$ is the inverse Mellin transform of the $m$th power of the completed Riemann zeta function  times the completed and rescaled version of the  zeta function $\zeta_S(s)$,  
is $\mathcal{C}_{\rm poly}^\infty({\R_+^\times})$-mean-periodic. The proof uses some of analytic properties of arithmetic zeta functions and does not appeal to higher dimensional adelic analysis. Conversely, if the function  $h_{\zeta_S,m_{\zeta_S}}(x)$ is $\mathcal{C}_{\rm poly}^\infty({\R_+^\times})$-mean-periodic then $\zeta_S(s)$ has a meromorphic continuation to $\C$ and satisfies the expected functional equation, whose sign is $\epsilon$. Note that a similar statement holds for $H_{\zeta_S,m}(t)\coloneqq h_{\zeta_S,m}(e^{-t})$. 
In particular, as a consequence, we get a correspondence
\begin{equation*}
\mathcal{C}:\begin{array}{lcl}
S & \mapsto & h_{\zeta_S,m_{\zeta_S}}
\end{array}
\end{equation*}
from the set of arithmetic  schemes whose  zeta function $\zeta_S(s)$ has the expected analytic properties to the space of $\mathcal{C}_{\rm poly}^\infty({\R_+^\times})$-mean-periodic functions.

It should be mentioned that a study of relations with mean-periodic functions and  Dirichlet series which do not include arithmetic zeta functions was conducted in 
\cite{MR985279}.

\subsection{The case of  zeta functions of models of elliptic curves}%
Let $E$ be an elliptic curve over the number field $\K$ and $\mathfrak{q}_E$ its conductor. We denote $r_1$ the number of real archimedean places of $\K$ and $r_2$ the number of conjugate pairs of complex archimedean places of $\K$. A detailed study of essentially three objects associated to $E$ is done in this work namely
\begin{itemize}
\item
the $L$-function $L(E,s)$, whose conjectural sign of functional equation is $\omega_E=\pm 1$,
\item
the Hasse--Weil zeta function $\zeta_E(s)$. 
It can be defined as the product of factors over all valuations
of $\K$ each of which is the Hasse zeta function of the one-dimensional model corresponding to a local minimal Weierstrass 
equation of $E$ with respect to the valuation.
Taking into account the computation of the zeta functions for curves over finite fields we get
\begin{equation*}
\zeta_E(s)=\frac{\zeta_\K(s)\zeta_\K(s-1)}{L(E,s)},
\end{equation*}
\item
the Hasse zeta function $\zeta_\mathcal{E}(s)$ of a regular proper model $\mathcal{E}$ of $E$.
We get  
\begin{equation*}
\zeta_\mathcal{E}(s)=n_\mathcal{E}(s)\zeta_E(s)
\end{equation*}
where  the factor 
$n_\mathcal{E}(s)$
is the product of finitely many, say $J$,  zeta functions of affine lines over finite extensions
of the residue field of bad reduction primes. The square of $\zeta_\mathcal{E}(s)$
 occurs in the two-dimensional zeta integral defined and studied in
the two dimensional adelic analysis. 
\end{itemize}
The boundary terms associated to these particular Hasse zeta functions are given by
\begin{eqnarray*}
h_E(x) & \coloneqq & f_{Z_E}(x)-(-1)^{r_1+r_2}\omega_Ex^{-1}f_{Z_E}\left(x^{-1}\right), \\
H_E(t) & \coloneqq & h_E\left(e^{-t}\right), \\
h_\mathcal{E}(x) & \coloneqq & f_{Z_\mathcal{E}}(x)-(-1)^{r_1+r_2+J}\omega_Ex^{-1}f_{Z_\mathcal{E}}\left(x^{-1}\right), \\
H_\mathcal{E}(t) & \coloneqq & h_\mathcal{E}\left(e^{-t}\right)
\end{eqnarray*}
where $f_{Z_E}(x)$ is the inverse Mellin transform of $Z_E(s)\coloneqq \Lambda_\K(s)\left(\text{N}_{\K\vert\Q}(\mathfrak{q}_E)^{-1}\right)^{2s/2}\zeta_E(2s)$ and $f_{Z_\mathcal{E}}(x)$ is the inverse Mellin transform of $Z_\mathcal{E}(s)\coloneqq \left(\prod_{1\leq i\leq I}\Lambda_{\K_i}(s)\right)\left(c_\mathcal{E}^{-1}\right)^{2s/2}\zeta_\mathcal{E}(2s)$, $I\geq 1$ being the number of chosen horizontal curves in the two-dimensional zeta integral briefly mentioned before. It is shown in this paper that
\begin{itemize}
\item
if the completed $L$-function $\Lambda(E,s)$ can be extended to a meromorphic function
of expected analytic shape  on $\C$ and satisfies the  functional equation then  $h_E(x)$ and $h_\mathcal{E}(x)$ are $\mathcal{C}_{\rm poly}^\infty({\R_+^\times})$-mean-periodic\footnote{Note that the $\mathcal{C}_{\rm poly}^\infty({\R_+^\times})$-mean-periodicity of $h_E(x)$ (respectively $h_\mathcal{E}(x)$) is equivalent to the $\mathcal{C}_{\rm exp}^\infty({\R})$-mean-periodicity of $H_E(t)$ (respectively $H_\mathcal{E}(t)$).},
\item
if $h_E(x)$ is $\mathcal{C}_{\rm poly}^\infty({\R_+^\times})$-mean-periodic or $h_\mathcal{E}(x)$ is $\mathcal{C}_{\rm poly}^\infty({\R_+^\times})$-mean-periodic then the Hasse--Weil zeta function $\zeta_E(s)$ can be extended to a meromorphic function on $\C$ and satisfies the expected functional equation and the  zeta function $\zeta_\mathcal{E}(s)$ can be extended to a meromorphic function on $\C$ and satisfies the expected functional equation. 
\end{itemize}
See Theorem \ref{prop_HW1} and Theorem \ref{prop_H1} for complete statements. 
Very briefly, the  proofs employ the following three properties:
\begin{enumerate}
\item
polynomial bound in $t$ for $\abs{L(E,\sigma+it)}$ in vertical strips; 
\item
exponential decay of the gamma function in vertical strips;
\item
polynomial bound in $t_n$ for $\abs{L(E,\sigma+it_n)}^{-1}$ in vertical strips
for certain sequences $t_n$ tending to infinity.
\end{enumerate}
The first and second properties are essential in the proof of Theorem \ref{thm_303},
the last property is used in the proof of Theorem \ref{thm_302}. %

\smallskip

The right spaces in which the functions are mean-periodic is one of the results of this paper. 
It is explained in Remark \ref{rem_NOT} that the function $h_E(x)$ cannot be $\mathcal{C}({\R_+^\times})$-mean periodic or $\mathcal{C}^\infty({\R_+^\times})$-mean-periodic and that the function $H_E(t)$ cannot be $\mathcal{C}({\R})$-mean periodic or $\mathcal{C}^\infty({\R})$-mean-periodic. Remark \ref{rem_info_zero} focuses on the fact that the function $h_E(x)$ encodes  in its \emph{Fourier series} some information on the poles of $\zeta_E(s)$, which are mainly the non-trivial zeros of $L(E,s)$. 

\smallskip 

The zeta functions of arithmetic schemes is the fundamental object in number theory.
The pioneering papers of Hecke and Weil initiated a fruitful development in number theory,
which emphasized the use of modular functions in the study of $L$-functions. 
Every zeta function of arithmetic scheme is the product of  first or  minus first powers 
of some $L$-functions which are called the $L$-factors of the zeta function. 
Currently,  modularity of all $L$-factors of the zeta function of an arithmetic scheme  of dimension greater than one is known only in very few  cases. 
Several decades of the study led to the Wiles and others proof of modularity of the $L$-function of elliptic curves over rationals. An extension of that method to elliptic curves over totally real fields would lead to the meromorphic continuation and functional equation of their $L$-function. 
 In particular, in all such cases we get  mean-periodicity
of the function $H_{\mathcal E}$ associated to the zeta function of the curve over those number fields, see \ref{prop_HW1}. 
Unfortunately, the method of the proof of the Wiles theorem does not seem to be extendable to handle the general case of elliptic curves over number fields. 

This paper suggests to study the conjectural meromorphic continuation and functional equation of zeta functions of arithmetic schemes without dealing with the 
$L$-factors of the zeta functions and without proving their modularity.
Instead it relates those properties of the zeta functions  to  mean-periodicity
of associated functions. 
We expect mean-periodicity to be very useful in the study of the zeta function of arithmetic schemes. We hope it is easier to establish  mean-periodicity
of associated functions for the whole zeta function than  to prove automorphic properties
of all $L$-factors of the zeta function. 
In particular,  two-dimensional adelic analysis uses underlying geometric structures
and is expected to help  to prove mean-periodicity of $H_{\mathcal E}$ for all elliptic curve over arbitrary global fields without any restrictions.

\subsection{Other results}%
The location of poles of $\zeta(s)$ and single sign property of $h_{\zeta}(x)$, whose study was initiated in  \cite[Section 4.3]{Fe2}, \cite[Section8]{Fe3} and \cite{Su1}, 
is the subject of  a general Proposition \ref{mono}. 

The poles of the inverse of the Dedekind zeta functions are studied from this point of view 
 in Proposition \ref{prop_502} and Proposition \ref{prop_ss}.

Products and quotients of completed $L$-functions associated to cuspidal automorphic representations are briefly discussed in Section \ref{sec_cusp}.  

From Eisenstein series we get continuous families of mean-periodic functions constructed in Section \ref{sec_Eis}, which leads to  several interesting questions. 

Corollary \ref{cor_201}  contains a new general explicit formula which involves the sum over all integers and thus differs from the standard explicit formulas for $L$-functions.

A variety of open interesting questions naturally arizes. One of them is this:
in the context of automorphic representations, there are many operations (tensor product, symmetric power, exterior power, functoriality), which give rise to different automorphic $L$-functions. Is it possible to translate these operations in the world of mean-periodic functions? Another fundamental question is to study 
 the image of the map $\mathcal{C}$. 

\subsection{Organisation of the paper}%
The general background on mean-periodic functions is given in Section \ref{sec_mp}. In Section \ref{sec_mbp0}, we give some sufficient conditions, which implies that the Mellin transforms of real-valued function on $\R_+^\times$ of rapid decay at $+\infty$ and polynomial order at $0^+$ have a meromorphic continuation to $\C$ and satisfy a functional equation (see Theorem \ref{thm_rmp}). All the general results on mean-periodic functions are proved in Section \ref{sec_results}. In the next two sections, we go in the opposite direction, deducing from analytic properties of zeta functions mean-periodicity of associated functions. Various links between  mean-periodicity and analytic properties of zeta functions of schemes, in particular zeta functions of models of elliptic curves, are shown in Section \ref{sec_zeta}. Section \ref{sec_others} provides many further instances of mean-periodic functions arising from number theory
including those coming from Dedekind zeta functions, standing in the denominator, and families of mean-periodic functions.
Appendix \ref{ap_analytic} contains an analytic estimate for general $L$-functions, which enables us to apply the general results on  mean-periodicity in relevant cases for number theory.
\begin{notations}
$\Z_+$ stands for the non-negative integers and $\R_+^\times$ for the positive real numbers. If $k$ is an integer and $x$ is a positive real number then $\log^k{(x)}\coloneqq \left(\log{x}\right)^k$.
\end{notations}
%
%
\section{Mean-periodic functions}\label{sec_mp}%
In this section, we give some information on mean-periodic functions, which first appeared in Delsarte \cite{Del} but whose theory has been initially developed by Schwartz in \cite{MR0023948}.
The general theory of mean-periodic functions can be found in Kahane \cite{Kah}, especially the theory of continuous mean-periodic functions of the real variable.  \cite[Section 11.3]{Ni} is a complete survey on the subject. \cite[Chapter 6]{MR1344448} is a nice reference for the smooth mean-periodic functions of the real variable whereas \cite[Chapter 4]{MR1326616} deals with smooth mean-periodic functions of the real variable but focuses on convolution equations. We also suggest the reading of \cite{MR570178}, \cite[Pages 169--181]{Me} and \cite{MR1039581}.
\subsection{Generalities}\label{spectral_synthesis}%
Three definitions of mean-periodicity are given in a very general context and links between them are mentioned. 
Let $X$ be a locally convex separated topological ${\C}$-vector space. 
Such space is specified by a suitable family of seminorms. In this paper, it will always be a Fr\'echet space or the inductive or projective limit of Fr\'echet spaces. 
Let  $G$ be a locally compact topological abelian group. Denote by $X^\ast$ the topological dual space of $X$ for some specified topology. We assume that there is a (continuous) \emph{representation}
\begin{eqnarray*}
\tau:G & \rightarrow & {\rm End}(X) \\
g & \mapsto & \tau_g.
\end{eqnarray*}
For $f \in X$, we denote by $\mathcal{T}(f)$ the closure of the $\C$-vector space spanned by $\{\tau_g(f), g \in G\}$ namely
\begin{equation*}
\mathcal{T}(f)\coloneqq \overline{\text{Vect}_\C\left(\left\{\tau_g(f), g\in G\right\}\right)}.
\end{equation*}
\begin{definition}\label{def_1}
$f\in X$ is \textit{$X$-mean-periodic} if $\mathcal{T}(f)\not=X$.
\end{definition}
Let us assume that there exists an \emph{involution} map
\begin{eqnarray*}
\check\;\;\;: X & \rightarrow & X \\
f & \mapsto & \check{f}.
\end{eqnarray*}
For $f \in X$ and $\varphi \in X^\ast$,
we define the \emph{convolution} $f\ast\varphi:G\rightarrow\C$ by
\begin{equation*}
(f\ast\varphi)(g)\coloneqq \langle\tau_g\check{f},\varphi\rangle
\end{equation*}
where $\langle ~,~ \rangle$ is the \emph{pairing} on $X \times X^\ast$.
\begin{definition}\label{def_2}
$f \in X$ is \textit{$X$-mean-periodic} if there exists a non-trivial element $\varphi$ of $X^\ast$ satisfying $f \ast \varphi=0$.
\end{definition}
Finally, let us assume that
\begin{itemize}
\item
$G=\R$ (respectively $\R_+^\times$),
\item
$X$ is a ${\C}$-vector space of functions or measures or distributions on $G$,
\item
there exists an open set $\Omega\subset\C$ such that the \emph{exponential polynomial} $P(t) e^{\lambda t}$ (respectively $x^\lambda P(\log x)$) belongs to X for any polynomial $P$ with complex coefficients and any $\lambda\in\Omega$.
\end{itemize}
\begin{definition}\label{def_3}
$f \in X$ is \textit{$X$-mean-periodic} if $f$ is a limit (with respect to the topology of $X$) of a sum of exponential polynomials belonging to $\mathcal{T}(f)$ .
\end{definition}
The first and second definitions are equivalent in a large class of spaces $X$ where the Hahn-Banach theorem is applicable. The equivalence between the first and third definitions depends on $X$ and is related to the following spectral problems (see \cite[Section 2.3]{Kah}). The \emph{spectral synthesis} holds in $X$ if
\begin{equation*}
\mathcal{T}(f)=\begin{cases}
\overline{\text{Vect}_\C\left(\left\{P(t)e^{\lambda t}\in \mathcal{T}(f), \lambda\in\Omega\right\}\right)} & \text{if $G=\R$,} \\
\overline{\text{Vect}_\C\left(\left\{x^\lambda P(\log{x})\in \mathcal{T}(f), \lambda\in\Omega\right\}\right)} & \text{if $G=\R_+^\times$.}
\end{cases}
\end{equation*}
for any $f$ in $X$ satisfying $\mathcal{T}(f)\neq X$. A representation of a mean-periodic function as a limit of exponential polynomials generalizes the Fourier series representation for continuous periodic functions. If $X=\mathcal{C}({\R})$ (see \cite[Sections 4 and 5]{Kah} and \cite{Me}) or $\mathcal{C}^\infty({\R})$ (see \cite{MR0023948}) then the three definitions are equivalent since the spectral synthesis holds in these spaces. 
The \emph{spectral analysis} holds in $X$ if for any $f$ in $X$ there exists a finite-dimensional translation invariant subspace $X_0$ of $X$ contained in $\mathcal{T}(f)$. For instance, $\mathcal{T}(t^ne^{\lambda t})$ is a finite-dimensional invariant subspace if $G=\R$.

\subsection{Quick review on continuous and smooth mean-periodic functions}\label{ssec_C}%
Let $X=\mathcal{C}(\R)$ be the space of continuous functions on $\R$ with the compact uniform convergence topology. 
It is a Fr\'echet space, hence completed and locally convex, whose dual space $X^\ast=M_0(\R)$ is the space of compactly supported Radon measures. The pairing between $f\in\mathcal{C}(\R)$ and $\mu\in M_0(\R)$ is given by
\begin{equation*}
\langle f,\mu \rangle=\int_{\R} f\dd\mu.
\end{equation*}
Let $G=\R$. The \emph{additive involution} is defined by
\begin{equation}\label{eq_involution}
\forall x\in\R,\forall f\in X, \quad \check{f}(x)=f(-x)
\end{equation}
and the \emph{additive representation} $\tau^+$ by
\begin{equation}\label{eq_representation_+}
\forall(x,y)\in\R^2,\forall f\in X,\quad\tau^+_x(f)(y)\coloneqq f(y-x).
\end{equation}
such that the \emph{additive convolution} $\ast_+$ is
\begin{equation*}
\forall x\in\R,\quad(f\ast_+\mu)(x)\coloneqq \int_{\R}f(x-y)\dd\mu(y).
\end{equation*}
The space $\mathcal{C}(\R)$ has two important properties. Firstly, the definitions \ref{def_1}, \ref{def_2} and \ref{def_3} are equivalent. In other words, the spectral synthesis holds for this space (see \cite[Sections 4 and 5]{Kah}, \cite{Me}). 
Secondly, it is possible to develop the theory of \emph{Laplace--Carleman transforms} of $\mathcal{C}(\R)$-mean-periodic functions. It turns out that Laplace--Carleman transforms in mean-periodicity are as important as Fourier transforms in harmonic analysis. Their theory is developed in Section \ref{ssec_La} in a more general context but let us justify a little bit the analogy. 
If $f$ is any \emph{non-trivial} $\mathcal{C}(\R)$-mean-periodic function then its Laplace--Carleman transform $\mathsf{LC}(f)(s)$ is a meromorphic function on $\C$ having at least one pole, otherwise $f=0$. In addition, its denominator belongs to the Cartwright class $\mathsf{C}$ defined in \eqref{eq_C}, for which the distribution of zeros is quite regular (\cite[Chapter 17]{Le}). The set of all poles of $\mathsf{LC}(f)(s)$ with multiplicity is called the \emph{spectrum} of $f$. 
The exponential polynomials belonging to $\mathcal{T}(f)$ are completely determined by the spectrum of $f$ and $f$ is characterized by the principal parts of the poles of $\mathsf{LC}(f)(s)$ (see \cite[Sections 4-6]{Kah}). 
Let us just mention that if $X=\mathcal{C}^\infty(\R)$, the space of smooth functions on $\R$ with the compact uniform convergence topology, then all what has been said above for $\mathcal{C}(\R)$ holds. In particular, the spectral synthesis holds in $\mathcal{C}^\infty(\R)$ (see \cite{MR0023948}). Finally, we would like to say that the spaces of continuous functions $\mathcal{C}(\R_+^\times)$ and smooth functions $\mathcal{C}^\infty(\R_+^\times)$ on $\R_+^\times$ share the same properties than $\mathcal{C}(\R)$ and $\mathcal{C}^\infty(\R)$. In particular, we can develop the theory of Mellin-Carleman transform via the homeomorphisms
\begin{eqnarray*}
\mathcal{H}_\mathcal{C}:\mathcal{C}(\R_+^\times) & \rightarrow & \mathcal{C}(\R) \\
f(x) & \mapsto & f\left(e^{-t}\right)
\end{eqnarray*}
and
\begin{eqnarray*}
\mathcal{H}_{\mathcal{C}^\infty}:\mathcal{C}^\infty({\R}_+^\times) & \rightarrow & \mathcal{C}^\infty(\R) \\
f(x) & \mapsto & f\left(e^{-t}\right).
\end{eqnarray*}
This will be done in a more general context in Section \ref{sse_Me}.
\subsection{Some relevant spaces with respect to mean-periodicity}%
In this section, we introduce several spaces for which the elements of the dual space are not necessarily compactly supported. 
Let $\mathcal{C}_{\exp}^\infty({\R})$ be the $\C$-vector space of smooth functions on ${\R}$, which have at most \emph{exponential growth} at $\pm\infty$ namely
\begin{equation}\label{eq_exp-growth}
\forall n\in\Z_+,\exists m\in\Z_+,\quad f^{(n)}(x)=O\left(\exp{(m\abs{x})}\right)
\end{equation}
as $x\to\pm\infty$. This space is a $\mathcal{LF}$-space, namely an inductive limit of Fr\'echet spaces $(F_m)_{m\geq 1}$, the topology on each $F_m$ ($m\geq 1)$ being induced from the following family of seminorms
\begin{equation*}
\Vert f \Vert_{m,n}=\sup_{x\in\R}\abs{f^{(n)}(x)\exp{(-m\abs{x})}}
\end{equation*}
for any $n\geq 0$. Let $\mathcal{C}_{\rm poly}^\infty({\R_+^\times})$ be the $\C$-vector space
of smooth functions on $\R_+^\times$, which have at most \emph{polynomial growth}  at $0^+$ and at $+\infty$ namely
\begin{equation}\label{eq_poly-growth}
\forall n\in\Z_+,\exists m\in\Z,\quad f^{(n)}(t)=O\left(t^m\right)
\end{equation}
as $t\to+\infty$ and $t\to0^+$.  The space $\mathcal{C}_{\rm poly}^\infty({\R_+^\times})$ is endowed with a topology such that the bijection
\begin{eqnarray*}
\mathcal{H}_{\mathcal{C}^\infty_\ast}:\mathcal{C}_{\exp}^\infty({\R}) & \rightarrow & \mathcal{C}_{\rm poly}^\infty({\R_+^\times}) \\
f(t) & \mapsto & f(-\log{x})
\end{eqnarray*}
becomes a homeomorphism. Let $\mathcal{C}_{\exp}^\infty({\R})^\ast$ be the dual space of
$\mathcal{C}_{\exp}^\infty({\R})$ equipped with the weak $\ast$-topology. The dual space $\mathcal{C}_{\exp}^\infty({\R})^\ast$ (respectively $\mathcal{C}_{\rm poly}^\infty({\R_+^\times})^\ast$)
is considered as a space of distributions on $\R$ (respectively $\R_+^\times$) having an \emph{over exponential decay} (respectively \emph{over polynomial decay}) in a suitable sense. In fact, every smooth function $g$ on $\R$ satisfying $\abs{g(x)}=O(\exp{(-a\abs{x})})$ for every real number $a>0$ is identified with an element of $\mathcal{C}_{\exp}^\infty({\R})^\ast$. Let $\langle ~, ~\rangle$ be the pairing between $\mathcal{C}_{\exp}^\infty({\R})^\ast$ and $\mathcal{C}_{\exp}^\infty({\R})$ namely $\langle\varphi,f\rangle=\varphi(f)$ for  $\varphi\in\mathcal{C}_{\exp}^\infty({\R})^\ast$ and  $f\in\mathcal{C}_{\exp}^\infty({\R})$. The \emph{additive convolution} $f\ast_+\varphi$ of  $f\in \mathcal{C}_{\exp}^\infty({\R})$ and
 $\varphi\in \mathcal{C}_{\exp}^\infty({\R})^\ast$ is defined as
\begin{equation*}
\forall x\in\R,\quad(f\ast_+\varphi)(x)\coloneqq \langle\varphi,\tau_x^+\check{f}\rangle
\end{equation*}
where the definitions of the involution $\check{}$ and of the representation $\tau^+$ are adapted from \eqref{eq_involution} and \eqref{eq_representation_+}. One can define the \emph{multiplicative convolution} $f\ast_\times\varphi:\R_+^\times\rightarrow\C$ of $f \in \mathcal{C}_{\rm poly}^\infty({\R_+^\times})$ and
 $\varphi \in \mathcal{C}_{\rm poly}^\infty({\R_+^\times})^\ast$ thanks to the homeomorphism $\mathcal{H}_{\mathcal{C}^\infty_\ast}$.\newline\newline
Let $\Schwartz(\R)$ be the \emph{Schwartz space} on $\R$ which consists of smooth functions on $\R$ satisfying
\begin{equation}
\Vert f\Vert_{m,n}=\sup_{x\in\R}\abs{x^mf^{(n)}(x)}<\infty
\end{equation}
for all $m$ and $n$ in $\Z_+$. It is a Fr\'echet space over the complex numbers with the topology induced from the family of seminorms $\Vert ~\Vert_{m,n}$.  
Let us define the Schwartz space $\Schwartz(\R_+^\times)$ on ${\R}_+^\times$ and its topology via the homeomorphism
\begin{eqnarray*}
\mathcal{H}_\Schwartz:\Schwartz(\R) & \rightarrow & \Schwartz(\R_+^\times) \\
f(t) & \mapsto & f(-\log{x}).
\end{eqnarray*}
The \emph{strong Schwartz space} $\mathbf{S}({\R}_+^\times)$ is defined by
\begin{equation}\label{eq_strong}
\mathbf{S}({\R}_+^\times)\coloneqq \bigcap_{\beta\in\R}\left\{f:\R_+^\times\to\C,\left[x \mapsto x^{-\beta}f(x)\right]\in\Schwartz(\R_+^\times)\right\}.
\end{equation}
One of the family of seminorms on $\mathbf{S}({\R}_+^\times)$ defining its topology is given by
\begin{equation} \label{901}
\Vert f \Vert_{m,n}=\sup_{x\in{\R}_+^\times}\abs{x^mf^{(n)}(x)}
\end{equation}
for  $m\in\Z$ and  $n\in\Z_+$. The strong Schwartz space $\mathbf{S}({\R}_+^\times)$ is a Fr\'echet space over the complex numbers where the family of seminorms defining its topology is given in \eqref{901}. In fact, it is a projective limit of Fr\'echet spaces $(F_m)_{m\geq 1}$ since a decreasing intersection of Fr\'echet spaces is still a Fr\'echet space. This space is closed under the multiplication by a complex number and the pointwise addition and multiplication (\cite{Mey1}). 
The strong Schwartz space $\mathbf{S}(\R)$ and its topology are defined via the homeomorphism
\begin{equation}\label{eq_strong_+}
\aligned
\mathcal{H}_\mathbf{S}:\mathbf{S}({\R}_+^\times) & \rightarrow  \mathbf{S}(\R) \\
f(x) & \mapsto  f\left(e^{-t}\right).
\endaligned 
\end{equation}
Let us mention that the Fourier transform is \emph{not} an automorphism of $\mathbf{S}(\R)$ since the Fourier transform of an element $f \in \mathbf{S}({\R})$ does not necessary belong to $\mathbf{S}(\R)$. This feature is different from what happens in $\Schwartz(\R)$. Let $\mathbf{S}(\R_+^\times)^\ast$ be the dual space of $\mathbf{S}(\R_+^\times)$ equipped with the weak $\ast$-topology, whose elements are called \emph{weak-tempered distributions}. The pairing between $\mathbf{S}(\R_+^\times)$ and $\mathbf{S}(\R_+^\times)^\ast$ is denoted $\langle ~, ~\rangle$ namely
\begin{equation*}
\langle f,\varphi\rangle=\varphi(f)
\end{equation*}
for  $f\in\mathbf{S}(\R_+^\times)$ and  $\varphi\in\mathbf{S}(\R_+^\times)^\ast$. A linear functional $\varphi$ on $\mathbf{S}(\R_+^\times)$ is a weak-tempered distribution if and only if
the condition $\lim_{k\to+\infty}\Vert f_k\Vert_{m,n}=0$ for all multi-indices $m$, $n$ implies
$\lim_{k\to+\infty}\langle f_k,\varphi\rangle=\lim_{k\to+\infty}\varphi(f_k)=0$. The \emph{multiplicative representation} $\tau^\times$ of $\R_+^\times$ on $\mathbf{S}({\R_+^\times})$ is defined by
\begin{equation*}
\forall x\in\R_+^\times,\quad\tau_x^\times f(y)\coloneqq f(y/x)
\end{equation*}
and the \emph{multiplicative convolution} $f\ast_\times\varphi$ of  $f\in\mathbf{S}({\R_+^\times})$ and
 $\varphi\in\mathbf{S}(\R_+^\times)^\ast$ by
\begin{equation*}
\forall x\in\R_+^\times,\quad(f\ast_\times\varphi)(x)=\langle\tau_x\check{f},\varphi\rangle
\end{equation*}
where the \emph{multiplicative involution} is given by $\check{f}(x)\coloneqq f(x^{-1})$. One can define the \emph{additive convolution} $f\ast_+\varphi:\R\rightarrow\C$ of  $f \in \mathbf{S}(\R)$ and
 $\varphi\in\mathbf{S}(\R)^\ast$ thanks to the homeomorphism $\mathcal{H}_\mathcal{\mathbf{S}}$. The \emph{multiplicative dual representation} $\tau^{\times,\ast}$ on $\mathbf{S}(\R_+^\times)^\ast$ is defined by
\begin{equation*}
\langle f,\tau_x^{\times,\ast}\varphi\rangle\coloneqq \langle\tau_x^\times f,\varphi\rangle.
\end{equation*}
One can define the \emph{additive dual representation} $\tau^{+,\ast}$ on $\mathbf{S}(\R)^\ast$ thanks to the homeomorphism $\mathcal{H}_\mathcal{\mathbf{S}}$. If $V$ is a $\C$-vector space then the bidual space $V^{\ast\ast}$ (the dual space of $V^\ast$ with respect to the weak $\ast$-topology on $V^\ast$) is identified with $V$ in the following way. For a continuous linear functional $F$ on $V^\ast$ with respect to its weak $\ast$-topology, there exists $v \in V$ such that $F(v^\ast)=v^\ast(v)$ for every $v^\ast \in V^\ast$. Therefore, we do not distinguish the pairing on $V^{\ast\ast}\times V^{\ast}$ from the pairing on $V\times V^\ast$. Under this identification, it turns out that
\begin{equation*}
\mathcal{C}_{\exp}^\infty({\R}) ~\prec~ \mathbf{S}({\R})^\ast, \quad
\mathcal{C}_{\rm poly}^\infty({\R_+^\times}) ~\prec~ \mathbf{S}({\R_+^\times})^\ast
\end{equation*}
and
\begin{equation*}
\mathbf{S}({\R}) ~\prec~ \mathcal{C}_{\exp}^\infty({\R})^\ast, \quad
\mathbf{S}({\R_+^\times}) ~\prec~ \mathcal{C}_{\rm poly}^\infty({\R_+^\times})^\ast,
\end{equation*}
where $A\prec B$ means that $A$ is a subset of $B$ and that the injection map $A\hookrightarrow B$ is continuous.
\subsection{Mean-periodic functions in these relevant spaces}%
In this section, $\mathfrak{X}$ always stands for one of the spaces $\mathcal{C}_{\exp}^\infty({\R})$, $\mathcal{C}_{\rm poly}^\infty({\R_+^\times})$, $\mathbf{S}({\R})^\ast$ and $\mathbf{S}({\R_+^\times})^\ast$.
\begin{definition}[additive]
Let $\mathfrak{X}$ be $\mathcal{C}_{\exp}^\infty({\R})$ or $\mathbf{S}({\R})^\ast$. $x\in\mathfrak{X}$ is said to be $\mathfrak{X}$-mean-periodic if there exists a non-trivial element $x^\ast$ in $\mathfrak{X}^\ast$ satisfying $x\ast_+x^\ast=0$.
\end{definition}
\begin{definition}[multiplicative]
Let $\mathfrak{X}$ be $\mathcal{C}_{\rm poly}^\infty({\R_+^\times})$ or $\mathbf{S}({\R_+^\times})^\ast$.  $x\in\mathfrak{X}$ is said to be $\mathfrak{X}$-mean-periodic if there exists a non-trivial element $x^\ast$ in $\mathfrak{X}^\ast$ satisfying $x\ast_\times x^\ast=0$.
\end{definition}
For $x\in\mathfrak{X}$, we denote by $\mathcal{T}(x)$ the closure of the $\C$-vector space spanned by $\{\tau_g(x), g \in G\}$ where
\begin{equation*}
\tau=
\begin{cases}
\tau^+ & \text{if $G=\R$ and $\mathfrak{X}=\mathcal{C}_{\exp}^\infty({\R})$,} \\
\tau^{+,\ast} & \text{if $G=\R$ and $\mathbf{S}({\R})^\ast$,} \\
\tau^\times & \text{if $G=\R_+^\times$ and $\mathfrak{X}=\mathcal{C}_{\rm poly}^\infty({\R_+^\times})$,} \\
\tau^{\times,*} & \text{if $G=\R_+^\times$ and $\mathfrak{X}=\mathbf{S}({\R_+^\times})^\ast$.}
\end{cases}
\end{equation*}
Hahn-Banach theorem leads to another definition of $\mathfrak{X}$-mean-periodic functions.
\begin{proposition}\label{proposition_td_mp}%
An element $x\in \mathfrak{X}$ is $\mathfrak{X}$-mean-periodic if and only if $\mathcal{T}(x)\not=\mathfrak{X}$.
\end{proposition}
The spectral synthesis holds for $X=\mathcal{C}_{\exp}^\infty({\R})$, this follows from \cite[Theorem A, Page 627]{Gi} which uses \cite {MR0105620}. 
A short self-contained text establishing the spectral synthesis for $\mathcal{C}_{\exp}^\infty({\R})$ is being written by A.~Borichev.

The previous definitions and identifications lead to the following links between the different $\mathfrak{X}$-mean-periodicities. Let $L_{\text{\emph{loc}},\exp}^1({\R})$ be the space of locally integrable functions $H(t)$ on $\R$ satisfying $H(t)=O(\exp(a\abs{t}))$ as $\abs{t}\to+\infty$ for some real number $a\geq 0$.
\begin{proposition}\label{proposition_match}
Let $H(t)\in L_{\text{\emph{loc}},\exp}^1({\R})$.
\begin{itemize} 
\item
If $H(t)\in \mathcal{C}_{\exp}^\infty({\R})$ is $\mathcal{C}_{\exp}^\infty({\R})$-mean-periodic and $F\ast_+H=0$ for some non-trivial $F \in \mathbf{S}({\R})$ then $H(t)$ is $\mathbf{S}({\R})^\ast$-mean-periodic.
\item
If $H(t)$ is $\mathbf{S}({\R})^\ast$-mean-periodic and $F\ast_+ H=0$ for some non-trivial $F\in \mathbf{S}({\R})$, which is continuous and compactly supported on $\R$, then $H(t)$ is $\mathcal{C}({\R})$-mean-periodic.
\end{itemize}
\end{proposition}
\subsection{Mean-periodicity and analytic properties of Laplace transforms I}\label{ssec_La}%
Let $G=\R$ and $X$ be a locally convex separated topological ${\C}$-vector space consisting of functions or distributions on $G$. Developing the theory of Laplace--Carleman transforms on $X$ requires the following additional properties on $X$:
\begin{itemize}
\item
there exists an open set $\Omega\subset\C$ such that every exponential monomial $t^p e^{\lambda t}$ with $p\in{\Z}_+$ and $\lambda\in\Omega$
belongs to $X$,
\item
the two sided Laplace transform
\[
\mathsf{L}_{\pm}(\varphi)(s)\coloneqq \langle e^{-st},\varphi\rangle\left(=\int_{-\infty}^{+\infty}\varphi(t)e^{-st}\dd t\right)
\]
is a holomorphic function on $\Omega$ for all  $\varphi\in X^\ast$.
\item
if $f\ast_+\varphi=0$ for some $\varphi\in X^\ast\setminus\{0\}$ then the two sided Laplace transforms
\[
\mathsf{L}_{\pm}(f^-\ast_+\varphi)(s)=\int_{-\infty}^{+\infty}(f^-\ast_+\varphi)(t)e^{-st}\dd t
\]
and
\[
\mathsf{L}_{\pm}(f^+\ast_+\varphi)(s)=\int_{-\infty}^{+\infty}(f^+\ast_+\varphi(t))e^{-st}\dd t
\]
are  holomorphic functions on $\Omega$,
where
\begin{equation*}
f^+(x)\coloneqq \begin{cases}
f(x) & \text{if $x\geq 0$,} \\
0 & \text{otherwise}
\end{cases}
\end{equation*}
\begin{equation*}
f^-(x)\coloneqq \begin{cases}
0 & \text{if $x\geq 0$,} \\
f(x) & \text{otherwise.}
\end{cases}
\end{equation*}
\end{itemize}
Of course, $X=\mathcal{C}({\R})$ and $X=\mathcal{C}^\infty({\R})$ both satisfy the previous conditions with $\Omega=\C$ since the respective dual spaces consist of compactly supported measures and distributions.
\begin{definition}\label{def_LC}
The \emph{Laplace--Carleman transform} $\mathsf{LC}(f)(s)$ of a $X$-mean-periodic
function (distribution) $f\in X$ is defined by
\[
\mathsf{LC}(f)(s)\coloneqq \frac{\mathsf{L}_{\pm}(f^+\ast_+\varphi)(s)}{\mathsf{L}_{\pm}(\varphi)(s)}=-\frac{\mathsf{L}_{\pm}(f^-\ast_+\varphi)(s)}{\mathsf{L}_{\pm}(\varphi)(s)}.
\]
for a $\varphi\in X^\ast\setminus\{0\}$ satisfying $f\ast_+\varphi=0$.
\end{definition}
It is easy to see  that $\mathsf{LC}(f)(s)$ does not depend on the particular choice of $\varphi\in X^\ast\setminus\{0\}$ satisfying $f\ast_+\varphi=0$. Its main analytic properties are described in the following proposition (see essentially \cite[Section 9.3]{Me}).
\begin{proposition}\label{prop_ana_LaCa}
Let $X$ be a locally convex separated topological $\C$-vector space consisting of functions or distributions on $\R$ satisfying the conditions above.
\begin{itemize}
\item
If $f$ is a $X$-mean-periodic function, then its Laplace--Carleman transform $\LaCa(f)(s)$ is a meromorphic function on $\Omega$.
\item
If $f$ is a $X$-mean-periodic function whose Laplace transform $\mathsf{L}(f)(s)$ exists for some right-half plane $\Re(s)>\sigma_0$ contained in $\Omega$ then $\mathsf{LC}(f)(s)=\mathsf{L}(f)(s)$ in that region and the Laplace--Carleman transform $\mathsf{LC}(f)(s)$ is the meromorphic continuation of $\mathsf{L}(f)(s)$ to $\Omega$.
\end{itemize}
\end{proposition}
Proving some functional equations may sometimes be done thanks to the following proposition.
\begin{proposition}
\label{lemma_mpL}
Let $\mathfrak{X}_+$ be $\mathcal{C}({\R})$ or $\mathcal{C}^\infty({\R})$ or $\mathcal{C}_{\rm exp}^\infty({\R})$ or $\mathbf{S}(\R)^*$. Let $f_1(t)$ and $f_2(t)$ be two $\mathfrak{X}_+$-mean-periodic functions whose Laplace transforms are defined on $\Re(s)>\sigma_0$ for some $\sigma_0$. If
\begin{equation*}
f_1(-t)=\epsilon f_2(t)
\end{equation*}
for some complex number $\epsilon$ of absolute value one, then the Laplace--Carleman transforms $\LaCa(f_1)(s)$ and $\LaCa(f_2)(s)$ of $f_1(t)$ and $f_2(t)$ satisfy the functional equation
\begin{equation*}
\LaCa(f_1)(-s)=-\epsilon\LaCa(f_2)(s).
\end{equation*}
\end{proposition}
\begin{proof}[\proofname{} of Proposition~\ref{lemma_mpL}]%
Let $\varphi_1, \varphi_2\neq 0$ be two elements of $\mathfrak{X}_+^\ast$ satisfying $f_1\ast_+\varphi_1=0$ and $f_2\ast_+\varphi_2=0$. If $\Re{(s)}>\sigma_0$ then
\begin{eqnarray*}
\La_\pm(f_1^+\ast_+\varphi_1)(s) & = & \int_{-\infty}^{+\infty}\left(\int_{-\infty}^{+\infty} f_1^+(t-x)e^{-st}\dd t\right)\dd\varphi_1(x), \\
& = &  \int_{-\infty}^{+\infty}\left(\int_{-\infty}^{+\infty}f_1^+(t)e^{-st}\dd t\right)e^{-sx}\dd\varphi_1(x), \\
& = & \La(f_1)(s)\;\;\La_\pm(\varphi_1)(s)
\end{eqnarray*}
and so $\LaCa(f_1)(s)=\La(f_1)(s)$. On the other hand, if $\Re{(s)}<-\sigma_0$ then
\begin{eqnarray*}
-\La_\pm(f_2^-\ast_+\varphi_2)(s) & = & -\int_{-\infty}^{+\infty}\left(\int_{-\infty}^{+\infty} f_2^-(t-x)e^{-st}\dd t\right)\dd\varphi_2(x), \\
& = &  -\int_{-\infty}^{+\infty}\left(\int_{-\infty}^{+\infty}f_2^-(t)e^{-st}\dd t\right)e^{-sx}\dd\varphi_2(x), \\
& = & -\int_{-\infty}^{0}f_2(t)e^{-st}\dd t\;\;\La_\pm(\varphi_2)(s) \\
& = & -\epsilon^{-1}\La(f_1)(-s)\;\;\La_\pm(\varphi_2)(s)
\end{eqnarray*}
the last line being a consequence of the functional equation satisfied by $f_1$ and $f_2$. 
Thus, if $\Re{(s)}<-\sigma_0$ then $\LaCa(f_2)(s)=-\epsilon^{-1}\La(f_1)(-s)$ or if $\Re{(s)}>\sigma_0$ then $\LaCa(f_2)(-s)=-\epsilon^{-1}\La(f_1)(s)$. We have just proved that
\begin{equation*}
\LaCa(f_1)(s)=-\epsilon\LaCa(f_2)(-s)
\end{equation*}
if $\Re{(s)}>\sigma_0$. Such equality remains valid for all  complex numbers $s$ by analytic continuation.
\end{proof}
\begin{remark}\label{rem_mpi}%
In the previous proof, the formal equality \eqref{107} is implicitly used. For instance, if $f_1(t)=f_2(t)$ and $\epsilon=+1$ then the $\mathcal{C}(\R)$-mean-periodicity of $f_1(t)$ can be formally written as
\[
0=\int_{-\infty}^{+\infty}\int_{-\infty}^{+\infty}f_1(t-x)e^{-st}\dd\varphi_1(t)=\left( \int_{-\infty}^{+\infty}f_1(t)e^{-st}\dd t\right)\times\left(\int_{-\infty}^{+\infty}e^{-sx}\dd\varphi_1(x)\right).
\]
Such formal equality can be compared with the formal Euler equality $\sum_{n\in\Z}z^n=0$ $(z\neq 1)$, which Euler used to calculate values of $\zeta(s)$ at negative integers. This equality was also used in the proof of rationality of zeta functions of curves over finite fields (\cite{MR1846477}). For modern interpretations of the Euler equality, see \cite{MR1861978} and also \cite[Section 8]{Fe1}.
\end{remark}
\subsection{Mean-periodicity and analytic properties of Laplace transforms II}\label{sse_Me}%
The multiplicative setting is sometimes more convenient in analytic number theory
than the additive one. In particular, it is used in the study of boundary terms of two-dimensional zeta integrals in which case $\R_+^\times=\abs{J_S}$ (see \cite[Section 35]{Fe2}). 
This is the reason why we define the multiplicative analogue of the Laplace--Carleman transform. 
Of course, all the arguments provided below hold in the additive setting via the change of variable $t\mapsto x=\exp(-t)$. 
Let $L_{\textrm{loc},{\rm poly}}^1({\R_+^\times})$ be the space of locally integrable functions on ${\R_+^\times}$ satisfying
\begin{equation*}
h(x)=\begin{cases}
O(x^{a}) & \text{as $x\to+\infty$}, \\
O(x^{-a}) & \text{as $x\to0^+$}
\end{cases}
\end{equation*}
for some real number $a\geq 0$. Each $h\in L_{\textrm{loc},{\rm poly}}^1({\R_+^\times})$ gives rise to a distribution $\varphi_h \in \mathbf{S}({\R_+^\times})^\ast$ defined by
\begin{equation*}
\forall f \in \mathbf{S}({\R}_+^\times),\quad\langle f , \varphi_h \rangle=\int_{0}^{+\infty}f(x)h(x)\frac{\dd x}{x}.
\end{equation*}
If there is no confusion, we denote $\varphi_h$ by $h$ itself and use the notations $\langle f,h\rangle=\langle f,\varphi_h\rangle$ and $h(x)\in \mathbf{S}({\R}_+^\times)^\ast$. 
Then 
\begin{equation*}
x^\lambda\log^k{(x)}\in{\mathcal{C}}_{\rm poly}^\infty({\R_+^\times})
\subset L_{\textrm{loc},{\rm poly}}^1({\R_+^\times})\subset \mathbf{S}({\R}_+^\times)^\ast
\end{equation*}
for all $k\in{\Z}_{+}$ and $\lambda\in{\C}$. Moreover, if $h\in L_{\textrm{loc},{\rm poly}}^1({\R_+^\times})$ then the multiplicative convolution $f\ast_\times\varphi_h$ coincides with the ordinary multiplicative convolution on functions on ${\R}_+^\times$ namely
\[
(f\ast_\times h)(x)=\langle\tau^\times_x\check{f},f\rangle=\int_{0}^{+\infty}f(x/y)h(y)\frac{\dd y}{y}=\int_{0}^{+\infty}f(y)h(x/y)\frac{\dd y}{y}.
\]
For a $h \in L_{\textrm{loc},{\rm poly}}^1({\R_+^\times})$ define $h^+$ and $h^-$ by
\begin{equation*}
h^+(x)\coloneqq \begin{cases}
0 & \text{if $x\geq 1$}, \\
h(x) & \text{otherwise}
\end{cases}
\quad 
h^-(x)\coloneqq \begin{cases}
h(x) & \text{if $x\geq 1$}, \\
0 & \text{otherwise}.
\end{cases}
\end{equation*}
Clearly, $h^{\pm}\in L_{\textrm{loc},{\rm poly}}^1({\R_+^\times})$ for all  $h\in L_{\textrm{loc},{\rm poly}}^1({\R_+^\times})$.
\begin{lemma}\label{lem_01}%
Let $h \in L_{\text{\emph{loc}},{\rm poly}}^1({\R_+^\times})$. If $f\ast_\times h=0$ for some non-trivial $f\in\mathbf{S}({\R_+^\times})$ then the Mellin transforms
\begin{equation*}
\mathsf{M}(f\ast_\times h^{\pm})(s)=\int_{0}^{+\infty}(f\ast_\times h^{\pm})(x)x^s\frac{\dd x}{x}
\end{equation*}
are entire functions on ${\C}$.
\end{lemma}
\begin{proof}[\proofname{} of Lemma~\ref{lem_01}]%
On one hand, $f\ast_\times h^-$ is of rapid decay as $x\to0^+$ and on the other hand, $f \ast h^+$ is of rapid decay as $x\to+\infty$. Both are of rapid decay as $x\to0^+$ and as $x\to+\infty$ since $f\ast_\times h^-=-f\ast_\times h^+$. Hence the Mellin transforms of $f \ast h^{\pm}$ are defined on the whole complex plane and entire.
\end{proof}
\begin{definition}\label{def_MC}
Let $h\in L_{\text{\emph{loc}},{\rm poly}}^1({\R_+^\times})$. If $f\ast_\times h=0$ for some non-trivial $f\in\mathbf{S}({\R_+^\times})$ then the \emph{Mellin--Carleman transform} $\mathsf{MC}(h)(s)$ of $h(x)$ is defined by
\begin{equation*}
\mathsf{MC}(h)(s)\coloneqq \frac{\mathsf{M}(f\ast_\times h^+)(s)}{\mathsf{M}(f)(s)}=-\frac{\mathsf{M}(f\ast_\times h^-)(s)}{\mathsf{M}(f)(s)}.
\end{equation*}
\end{definition}
The change of variable $x \mapsto t=-\log{x}$ entails that the Mellin--Carleman transform coincides with the Laplace--Carleman transform namely
\begin{equation*}
\mathsf{MC}(h(x))=\mathsf{LC}(h(e^{-t})).
\end{equation*}
As a consequence, $\mathsf{MC}(h)$ does not depend on the particular choice of non-trivial $f$ satisfying $f\ast_\times h=0$.
\begin{proposition}\label{proposition_MC}
Let $h$ be an element of $L_{\text{\emph{loc}},{\rm poly}}^1({\R_+^\times})$
\begin{itemize}
\item
If $h$ is $\mathbf{S}({\R_+^\times})^\ast$-mean-periodic then the Mellin--Carleman transform $\mathsf{MC}(h)(s)$  of $h$ is a meromorphic function on $\C$.
\item
If $h \in \mathcal{C}_{\rm poly}^\infty({\R_+^\times})$ is $\mathcal{C}_{\rm poly}^\infty({\R_+^\times})$-mean-periodic and $f \ast_\times h=0$ for some non-trivial $f \in\mathbf{S}({\R_+^\times}) \subset \mathcal{C}_{\rm poly}^\infty({\R_+^\times})^\ast$ then $\mathsf{MC}(h)(s)$ is a meromorphic function on $\C$.
\end{itemize}
\end{proposition}
\begin{proof}[\proofname{} of Proposition~\ref{proposition_MC}]%
It follows immediately from the fact that the Mellin transform $\mathsf{M}(f)(s)$ is an entire function since $f\in \mathbf{S}({\R}_+^\times)$ and from Lemma \ref{lem_01}.
\end{proof}
Let us focus on the fact that the Mellin--Carleman transform $\mathsf{MC}(h)(s)$ of $h(x)$ is \emph{not} a generalization of the Mellin transform of $h$ 
but is a generalization of the following integral, half Mellin transform, 
\[
\int_{0}^{1}h(x)x^{s}\frac{\dd x}{x}.
\]
according to the following proposition.
\begin{proposition}\label{prop_MC2}
Let $h \in L_{\text{\emph{loc}},{\rm poly}}^1({\R_+^\times})$. If $f\ast_\times h=0$ for some non-trivial $f\in \mathbf{S}({\R_+^\times})$ and if the integral $\int_{0}^{1}h(x)x^{s}\dd x/x$ converges absolutely for $\Re(s)>\sigma_0$ for some real number $\sigma_0$ then
\[
\mathsf{MC}(h)(s)=\int_{0}^{1}h(x)x^{s}\frac{\dd x}{x}
\]
on $\Re{(s)}>\sigma_0$.
\end{proposition}
The following proposition is the analogue of Proposition \ref{lemma_mpL} and its proof is omitted.
\begin{proposition}
\label{lemma_mpC}
Let $\mathfrak{X}_\times$ be $\mathcal{C}(\R_+^\times)$ or $\mathcal{C}^\infty(\R_+^\times)$ or $\mathcal{C}_{\rm poly}^\infty({\R_+^\times})$ or $\mathbf{S}(\R_+^\times)^*$. Let $f_1(x)$ and $f_2(x)$ be two $\mathfrak{X}_\times$-mean-periodic functions whose transforms
\begin{equation*}
\int_{0}^1f_i(x)x^s\frac{\dd x}{x}
\end{equation*}
are defined on $\Re(s)>\sigma_0$ for some $\sigma_0$. If
\begin{equation*}
f_1\left(x^{-1}\right)=\epsilon f_2(x)
\end{equation*}
for some complex number $\epsilon$ of absolute value one then the Mellin--Carleman transforms $\MeCa(f_1)(s)$ and $\MeCa(f_2)(s)$ of $f_1(x)$ and $f_2(x)$ satisfy the functional equation
\begin{equation*}
\MeCa(f_1)(-s)=-\epsilon\MeCa(f_2)(s).
\end{equation*}
\end{proposition}
Finally, let us mention the following result, which is the analogue of \cite[Theorem Page 23]{Kah}.
\begin{proposition}
Let $h \in L_{\text{\emph{loc}},{\rm poly}}^1({\R_+^\times})$ and let $P(t)$ be a polynomial of degree $n$ with complex coefficients. Let us assume that $h$ is $\mathbf{S}({\R_+^\times})^*$-mean-periodic. The exponential polynomial $x^\lambda P(\log{x})$ belongs to $\mathcal{T}(h)$ if and only if $\lambda$ is a zero of order at least $n$ of the Mellin transform $\mathsf{M}(f)(s)$ of $f$, where $f$ runs through all elements of $\mathbf{S}({\R_+^\times})$ satisfying $f\ast_\times h=0$. Moreover, $x^\lambda P(\log x)$ belongs to $\mathcal{T}(h)$ if and only if $\lambda$ is a pole of order at least $n$ of the Mellin--Carleman transform $\mathsf{MC}(h)(s)$ of $h$.
\end{proposition}
\section{Mean-periodicity and analytic properties of boundary terms}\label{sec_mbp0}%
\subsection{The general issue}\label{subsec_func}%
If $f(x)$ is a real-valued function on $\R_+^\times$ satisfying
\begin{enumerate}
\item\label{eq_1}
$f(x)$ is of \emph{rapid decay} as $x\to+\infty$ namely $f(x)=O\left(x^{-A}\right)$ for all  $A>0$ as $x\to+\infty$,
\item
$f(x)$ is of \emph{polynomial order} as $x\to0^+$ namely $f(x)=O\left(x^{A}\right)$ for some $A>0$ as $x\to0^+$
\end{enumerate}
then its \emph{Mellin transform}
\[
\Me(f)(s)\coloneqq \int_0^{+\infty}f(x)x^s\frac{\dd x}{x}
\]
is a holomorphic function on $\Re{(s)}\gg 0$.
We are interested in necessary and sufficient conditions on $f$, which imply the meromorphic continuation and functional equation of $\Me(f)(s)$. We get immediately
\begin{equation*}
\Me(f)(s)=\varphi_{f,\epsilon}(s)+\omega_{f,\epsilon}(s)
\end{equation*}
where
\begin{equation*}
\varphi_{f,\epsilon}(s)\coloneqq \int_1^{+\infty}f(x)x^s\frac{\dd x}{x}+\epsilon\int_1^{+\infty}f(x)x^{1-s}\frac{\dd x}{x}
\end{equation*}
and
\begin{equation}\label{eq_omega}
\omega_{f,\epsilon}(s)\coloneqq \int_0^1h_{f,\epsilon}(x)x^s\frac{\dd x}{x}
\end{equation}
with
\begin{equation}\label{102}
h_{f,\epsilon}(x)=f(x)-\epsilon x^{-1}f\left(x^{-1}\right)
\end{equation}
where  $\epsilon$ is $\pm 1$. Assumption \eqref{eq_1} ensures that $\varphi_{f,\epsilon}(s)$ is an entire function satisfying the functional equation
\begin{equation*}
\varphi_{f,\epsilon}(s)=\epsilon\varphi_{f,\epsilon}(1-s).
\end{equation*}
Thus, the meromorphic continuation of $\omega_{f,\epsilon}(s)$ is equivalent to the meromorphic continuation of $\Me(f)(s)$,  and the functional equation $\omega_{f,\epsilon}(s)=\epsilon\omega_{f,\epsilon}(1-s)$ is equivalent to the functional equation $\Me(f)(s)=\epsilon\Me(f)(1-s)$. The change of variable $x=e^{-t}$ shows that $\omega_{f,\epsilon}(s)$ is the \emph{Laplace transform} of
\begin{equation}\label{eq_H}
H_{f,\epsilon}(t)=h_{f,\epsilon}\left(e^{-t}\right)
\end{equation}
namely
\begin{equation} \label{104}
\omega_{f,\epsilon}(s)=\La\left(H_{f,\epsilon}\right)(s)=\int_0^{+\infty}H_{f,\epsilon}(t)e^{-st}\dd t.
\end{equation}
The functions $h_{f,\epsilon}(x)$ and $H_{f,\epsilon}(t)$ are called \emph{boundary terms}, having in mind the motivation given in Section \ref{subsec_classical}.
The \textbf{issue} is then the following one. What property of the boundary terms ensures that their Laplace transforms admit a meromorphic continuation to $\C$ and satisfy a functional equation? One possible answer, in which the keyword is mean-periodicity, is given in the two following parts. We shall see there that mean-periodicity of $h_{f,\epsilon}(x)$ or $H_{f,\epsilon}(t)$ is a sufficient condition.\label{sec_mpb} \newline\newline
In this section, $\mathfrak{X}_+$ will be one of the following spaces
\begin{eqnarray*}
\mathcal{C}({\R}) & \text{defined in Section \ref{ssec_C},} \\
\mathcal{C}^\infty({\R}) & \text{defined in Section \ref{ssec_C},} \\
\mathcal{C}_{\rm exp}^\infty({\R}) & \text{defined in \eqref{eq_exp-growth},} \\
\mathbf{S}(\R)^* & \text{whose dual is defined in \eqref{eq_strong_+}}
\end{eqnarray*}
whereas $\mathfrak{X}_\times$ will be one of the following spaces
\begin{eqnarray*}
\mathcal{C}({\R_+^\times}) & \text{defined in Section \ref{ssec_C},} \\
\mathcal{C}^\infty({\R_+^\times}) & \text{defined in Section \ref{ssec_C},} \\
\mathcal{C}_{\rm poly}^\infty({\R_+^\times}) & \text{defined in \eqref{eq_poly-growth},} \\
\mathbf{S}(\R_+^\times)^* & \text{whose dual is defined in \eqref{eq_strong}.}
\end{eqnarray*}
For the definitions of $\mathfrak{X}_+$-mean periodic functions and of $\mathfrak{X}_\times$-mean periodic functions, we refer the reader to Section \ref{sec_mp}.
\subsection{Mean-periodicity and meromorphic continuation of boundary terms}%
The general theory of mean-periodic functions asserts that if $h_{f,\epsilon}$, defined in \eqref{102}, is $\mathfrak{X}_\times$-mean-periodic then $\omega_{f,\epsilon}$, defined in \eqref{eq_omega}, admits a meromorphic continuation to $\C$. More precisely, the \emph{Mellin--Carleman transform} $\MeCa\left(h_{f,\epsilon}\right)$ of $h_{f,\epsilon}$ (see Definition \ref{def_MC}) is the meromorphic continuation of $\omega_{f,\epsilon}$ (see Proposition \ref{proposition_MC} and Proposition \ref{prop_MC2}). Similarly, if $H_{f,\epsilon}$ defined in \eqref{eq_H} is $\mathfrak{X}_+$-mean-periodic then its Laplace transform $\omega_{f,\epsilon}(s)$ (see \eqref{104}) admits a meromorphic continuation to $\C$. Indeed, the \emph{Laplace--Carleman transform} $\LaCa\left(H_{f,\epsilon}\right)$ of $H_{f,\epsilon}$ (see Definition \eqref{def_LC}) is the meromorphic continuation of $\omega_{f,\epsilon}$ (see Proposition \ref{prop_ana_LaCa}).
\subsection{Mean-periodicity and functional equation of boundary terms}%
Let us focus on the eventual functional equation satisfied by $\omega_{f,\epsilon}(s)$. Note that $h_{f,\epsilon}(x)$ satisfies the functional equation
\begin{equation}\label{105}
h_{f,\epsilon}\left(x^{-1}\right)=-\epsilon xh_{f,\epsilon}(x).
\end{equation}
according to \eqref{102}. In other words, the function $\widetilde{h}_{f,\epsilon}\coloneqq \sqrt{x}h_{f,\epsilon}(x)$ satisfies
\begin{equation*}
\widetilde{h}_{f,\epsilon}\left(x^{-1}\right)=-\epsilon\widetilde{h}_{f,\epsilon}\left(x\right).
\end{equation*}
In terms of $H_{f,\epsilon}(t)$, it exactly means that the function $\widetilde{H}_{f,\epsilon}(t)\coloneqq e^{-t/2}H_{f,\epsilon}(t)$ satisfies
\begin{equation*}
\widetilde{H}_{f,\epsilon}(-t)=-\epsilon\widetilde{H}_{f,\epsilon}(t).
\end{equation*}
In general, \eqref{105} does not imply the functional equation $\omega_{f,\epsilon}(s)=\epsilon\omega_{f,\epsilon}(1-s)$, even if $\omega_{f,\epsilon}(s)$ admits a meromorphic continuation to $\C$. For example, if $f(x)=e^{-x}$ and $\epsilon=1$ then $h_f(x)=e^{-x}-x^{-1}e^{-1/x}$ and $\Me(f)(s)=\Gamma(s)$ is a meromorphic function on $\C$, which does not satisfy $\Me(f)(s)=\Me(f)(1-s)$.\newline\newline
Let us assume that $\omega_{f,\epsilon}(s)$ satisfies the functional equation $\omega_{f,\epsilon}(s)=\epsilon\omega_{f,\epsilon}(1-s)$. It can be formally written as
\begin{equation}\label{106}
\int_{0}^{1}h_{f,\epsilon}(x)x^s\frac{\dd x}{x}=\epsilon\int_{0}^{1}h_{f,\epsilon}(x)x^{1-s}\frac{\dd x}{x}.
\end{equation}
The right-hand side is equal to
\[
-\epsilon^2\int_{0}^{1}h_{f,\epsilon}(x^{-1})x^{-s}\frac{\dd x}{x}=-\int_{1}^{+\infty}h_{f,\epsilon}(x)  x^{s}\frac{\dd x}{x}
\]
according to \eqref{105}. Hence we get
\begin{equation}\label{107}
\int_{0}^{1}h_{f,\epsilon}(x)x^s\frac{\dd x}{x}+\int_{1}^{+\infty}h_{f,\epsilon}(x) x^{s}\frac{\dd x}{x}=0.
\end{equation}
Conversely, if we suppose \eqref{107} then we formally obtain \eqref{106} by using \eqref{105}. As a consequence, we guess that, under the meromorphic continuation of $\omega_{f,\epsilon}(s)$, the functional equation of $\omega_{f,\epsilon}(s)$ is equivalent to \eqref{105} and \eqref{107}, 
and \eqref{107} corresponds to  mean-periodicity. \newline\newline
Once again, the general theory of mean-periodic functions asserts that if $h_{f,\epsilon}(x)$, defined in \eqref{102}, is $\mathfrak{X}_\times$-mean-periodic then the Mellin--Carleman transform $\MeCa\left(\widetilde{h}_{f,\epsilon}\right)(s)$ of $\widetilde{h}_{f,\epsilon}(x)$ satisfies the functional equation
\begin{equation*}
\MeCa\left(\widetilde{h}_{f,\epsilon}\right)(s)=\epsilon\MeCa\left(\widetilde{h}_{f,\epsilon}\right)(-s)
\end{equation*}
according to Proposition \ref{lemma_mpC}. This is equivalent to the functional equation
\begin{equation*}
\omega_{f,\epsilon}(s)=\epsilon\omega_{f,\epsilon}(1-s)
\end{equation*}
since $\MeCa\left(\widetilde{h}_{f,\epsilon}\right)(s)=\MeCa\left(h_{f,\epsilon}\right)(s+1/2)$. Similarly, if $H_{f,\epsilon}(t)$ defined in \eqref{eq_H} is $\mathfrak{X}_+$-mean-periodic then the Laplace--Carleman transform $\LaCa\left(\widetilde{H}_{f,\epsilon}\right)(s)$ of $\widetilde{H}_{f,\epsilon}(t)$ satisfies the functional equation
\begin{equation*}
\LaCa\left(\widetilde{H}_{f,\epsilon}\right)(s)=\epsilon\LaCa\left(\widetilde{H}_{f,\epsilon}\right)(-s)
\end{equation*}
according to Proposition \ref{lemma_mpL}. This is equivalent to the functional equation
\begin{equation*}
\omega_{f,\epsilon}(s)=\epsilon\omega_{f,\epsilon}(1-s)
\end{equation*}
since $\LaCa\left(\widetilde{H}_{f,\epsilon}\right)(s)=\LaCa\left(H_{f,\epsilon}\right)(s+1/2)$.

\begin{remark}\label{rem_nmp}
One can give various kinds of examples of smooth functions on the real line of exponential growth, which on one hand are not $\mathfrak{X}_+$-mean-periodic for every functional space $\mathfrak{X}_+$ of functions  given at page \pageref{sec_mpb}, and on the other hand whose Laplace transform extends to a symmetric meromorphic function on the complex plane. A series of  examples due to A.~Borichev is supplied by the following general  
construction. First, recall that if real $b_n>1$ are some of  zeros of a function holomorphic and bounded in the half-place $\Re{(s)}>0$ then $\sum 1/b_n<\infty$, see e.g. \cite[Problem 298, sect.2 Ch.6 Part III]{MR1492447}. Using the map $z\to \exp(-iz)$  we deduce that if $ib_n$, $b_n>0$, are some of zeros of an entire function bounded in the vertical strip $\abs{\Re{(s)}}<\pi/2$ then $\sum \exp{(-b_n)}<\infty$. Now choose  a sequence $(a_n)$ of positive real numbers such that
the set $\{i a_n\}$ is not a subset of all zeros of any entire function
bounded in the  vertical strip $\abs{\Re{(s)}}<\pi/2$. For example, using what has been said previously in this remark, one can take $a_n=\log n$. Choose sufficiently fast decaying non-zero coefficients $c_n$ so that $H(t)=\sum c_n \sin (a_n t)$ belongs to the space $\mathfrak{X}_+$ of smooth functions of exponential growth and its Laplace transform $w(s)=\sum c_n a_n/(s^2+a_n^2)$ is a symmetric meromorphic function on the complex plane. Assume that $H(t)$ is $\mathfrak{X}_+$-mean-periodic. Then $H\ast_+\tau=0$ for some non-zero $\tau\in \mathfrak{X}_+^*$. Convolving $\tau$ with a smooth function we can assume that $\tau$ is a smooth function of over exponential decay. Mean-periodicity of $H(t)$ implies that the meromorphic function $w(s)$ coincides with the Laplace--Carleman transform of $H(t)$, and so the set of poles $\{\pm ia_n\}$ of $w(s)$ is a subset of zeros of the two sided Laplace transform of $\tau$. Note that the two sided Laplace transform of a smooth function with over exponential decay is an entire function $v(s)$ such that for every positive integer $m$ the function $\abs{v(s)}(1+\abs{s})^m$ is bounded in the vertical strip $\abs{\Re{(s)}}<m$. The choice of the sequence $(a_n)$ gives a contradiction. Hence the function $H(t)$ is not $\mathfrak{X}_+$-mean-periodic and its Laplace transform extends to a symmetric meromorphic function.
\end{remark}
\subsection{Statement of the result}%
Let us encapsulate all the previous discussion of this section in the following theorem.
\begin{theorem}\label{thm_rmp}
Let $\mathfrak{X}_\times$ be $\mathcal{C}(\R_+^\times)$ or $\mathcal{C}^\infty(\R_+^\times)$ or $\mathcal{C}_{\rm poly}^\infty({\R_+^\times})$ or $\mathbf{S}(\R_+^\times)^*$  and $\mathfrak{X}_+$ be $\mathcal{C}(\R)$ or $\mathcal{C}^\infty(\R)$ or $\mathcal{C}_{\rm exp}^\infty({\R})$ or $\mathbf{S}(\R)^*$. Let $f(x)$ be a real-valued function on $\R_+^\times$ of rapid decay as $x\to+\infty$ and of polynomial order as $x\to0^+$. Let $\epsilon=\pm 1$.
\begin{itemize}
\item
If $h_{f,\epsilon}(x)=f(x)-\epsilon x^{-1}f\left(x^{-1}\right)$ is a $\mathfrak{X}_\times$-mean-periodic function then the Mellin transform $\Me(f)(s)$ of $f(x)$ admits a meromorphic continuation to $\C$ and satisfies the functional equation
\begin{equation*}
\Me(f)(s)=\epsilon\Me(f)(1-s).
\end{equation*}
More precisely,
\begin{equation*}
\Me(f)(s)=\int_1^{+\infty}f(x)x^s\frac{\dd x}{x}+\epsilon\int_1^{+\infty}f(x)x^{1-s}\frac{\dd x}{x}+\omega_{f,\epsilon}(s)
\end{equation*}
where $\omega_{f,\epsilon}(s)$ coincides on $\Re{(s)}\gg0$ with $\int_0^{1}h_{f,\epsilon}(x)x^s\dd x/x$ , admits a meromorphic continuation to $\C$ given by the Mellin--Carleman transform $\MeCa\left(h_{f,\epsilon}\right)(s)$ of $h_{f,\epsilon}(t)$ and satisfies the functional equation
\begin{equation*}
\omega_{f,\epsilon}(s)=\epsilon\omega_{f,\epsilon}(1-s).
\end{equation*}
\item
If $H_{f,\epsilon}(t)=h_{f,\epsilon}(e^{-t})$ is a $\mathfrak{X}_+$-mean-periodic function then the Mellin transform $\Me(f)(s)$ of $f(x)$ admits a meromorphic continuation to $\C$ and satisfies the functional equation
\begin{equation*}
\Me(f)(s)=\epsilon\Me(f)(1-s).
\end{equation*}
More precisely,
\begin{equation*}
\Me(f)(s)=\int_1^{+\infty}f(x)x^s\frac{\dd x}{x}+\epsilon\int_1^{+\infty}f(x)x^{1-s}\frac{\dd x}{x}+\omega_{f,\epsilon}(s)
\end{equation*}
where $\omega_{f,\epsilon}(s)$ coincides on $\Re{(s)}\gg0$ with the Laplace transform $\La\left(H_{f,\epsilon}\right)(s)$ of $H_{f,\epsilon}(t)$, admits a meromorphic continuation to $\C$ given by the Laplace--Carleman transform $\LaCa\left(H_{f,\epsilon}\right)(s)$ of $H_{f,\epsilon}(t)$ and satisfies the functional equation
\begin{equation*}
\omega_{f,\epsilon}(s)=\epsilon\omega_{f,\epsilon}(1-s).
\end{equation*}
\end{itemize}
\end{theorem}

%
\section{On a class of mean-periodic functions arising from number theory}\label{sec_results}%
In this section, we show that for a certain class of functions, which naturally come from number theory (zeta functions of arithmetic schemes), their meromorphic continuation and their functional equation are essentially equivalent to  mean-periodicity of some associated functions.
%
\subsection{Functions which will supply mean-periodic functions} %
Firstly, we define some suitable set of functions, from which mean-periodic functions will be built.
\begin{definition}\label{def_E}
$\mathcal{F}$ is defined by the set of complex-valued functions $Z(s)$ of the shape
\begin{equation*}
Z(s)=\gamma(s)D(s)
\end{equation*}
where $\gamma(s)$ and $D(s)$ are some meromorphic functions on $\C$ with the following conditions
\begin{itemize}
\item
all the poles of $Z(s)$ belong to the  vertical strip $\abs{\Re{(s)}-1/2}\leq w$ for some $w>0$,
\item
$\gamma(s)$ satisfies the uniform bound
\begin{equation}\label{313}
\forall\sigma\in[a,b],\forall\abs{t}\geq t_0,\quad\abs{\gamma(\sigma+it)}\ll_{a,b,t_0}\abs{t}^{-A}
\end{equation}
for all real numbers $a\leq b$ and every real number $A>0$,
\item
If $\sigma>1/2+w$ then
\begin{equation}\label{eq_AC}
D(\sigma+it)\ll\abs{t}^{A_1}
\end{equation}
for some real number $A_1$,
\item
there exists a real number $A_2$ and a strictly increasing sequence of positive real numbers $\{t_m\}_{m \geq 1}$ satisfying
\begin{equation}\label{314}
\abs{D(\sigma \pm i t_m)}=O\left(t_m^{A_2}\right)
\end{equation}
uniformly for $\sigma\in[1/2-w-\delta,1/2+w+\delta]$ for some $\delta>0$ and for all  integer $m\geq 1$.
\end{itemize}
\end{definition}
If $\lambda$ is a pole of $Z(s)$ in $\mathcal{F}$ of multiplicity $m_\lambda\geq 1$ then the principal part is written as
\begin{equation}\label{principal_part}
Z(s)=\sum_{m=1}^{m_\lambda}\frac{C_m(\lambda)}{(s-\lambda)^m}+O(1)
\end{equation}
when $s\to\lambda$. Note that we can associate to a element $Z\in\mathcal{F}$ its inverse Mellin transform namely
\begin{equation}\label{315}
f_Z(x)\coloneqq \frac{1}{2i\pi}\int_{(c)}Z(s)x^{-s}\dd s
\end{equation}
for $c>1/2+w$. By \eqref{313}, \eqref{314} and the fact that the poles of $Z(s)$ lie in some vertical strip $\abs{\Re{(s)}-1/2}\leq w$ for some $w>0$, we have
\begin{equation} \label{318}
f_{Z}(x)=O\left(x^{-A}\right) \quad  \text{as $x\to+\infty$}
\end{equation}
for all  $A\geq 0$ and
\begin{equation} \label{319}
f_{Z}(x)=O\left(x^{-1/2-w-\epsilon}\right) \quad \text{as $x\to0^+$}.
\end{equation}
As a consequence, $f_Z(x)$ is of rapid decay as $x\to+\infty$ and is at most of polynomial growth as $x\to 0^+$. Let $Z_i(s)=\gamma_i(s)D_i(s)$ for $i=1,2$ be two elements of $\mathcal{F}$ linked by the functional equation
\begin{equation*}
Z_1(s)=\varepsilon Z_2(1-s)
\end{equation*}
where $\epsilon$ is a complex number of absolute value one. Let us define $h_{12}(x)$ by
\begin{equation}\label{316}
h_{12}(x)\coloneqq f_{Z_1}(x)-\varepsilon x^{-1}f_{Z_2}\left(x^{-1}\right)
\end{equation}
and $h_{21}(x)$ by
\begin{equation}\label{316bis}
h_{21}(x)\coloneqq f_{Z_2}(x)-\varepsilon^{-1}x^{-1}f_{Z_1}\left(x^{-1}\right).
\end{equation}
The functions $h_{12}(x)$ and $h_{21}(x)$ are at most of polynomial growth as $x\to0^+$ and $x\to+\infty$ since $f_{Z_1}$ and $f_{Z_2}$ are of rapid decay as $x\to+\infty$ and are at most of polynomial growth as $x\to 0^+$.
The purpose of this section is to establish mean-periodicity of these functions.%
\subsection{Mean-periodicity of these functions}%
\begin{theorem}\label{thm_302}
Let $Z_i(s)=\gamma_i(s)D_i(s)$ for $i=1,2$ be two elements of $\mathcal{F}$ and $\epsilon$ be a complex number of absolute value one. The functions $h_{12}(x)$ and $h_{21}(x)$ defined by \eqref{316} and \eqref{316bis} are continuous functions on ${\R}_+^\times$ satisfying the functional equation
\begin{equation}\label{317}
h_{12}\left(x^{-1}\right)=-\varepsilon xh_{21}(x).
\end{equation}  
In addition, if $Z_1(s)$ and $Z_2(s)$ satisfy the functional equation
\begin{equation} \label{311}
Z_1(s)=\varepsilon Z_2(1-s)
\end{equation}
then
\begin{enumerate}
\item
$h_{12}(x)$ and $h_{21}(x)$ are limits, in the sense of compact uniform convergence, of sums of exponential polynomials\footnote{Remember that exponentials polynomials in $\mathcal{C}(\R_+^\times)$ are given by $x^{-\lambda}P(\log{x})$.} in $\mathcal{C}(\R_+^\times)$,
\item
$h_{12}(x)$ and $h_{21}(x)$ belong to $\mathcal{C}_{\rm poly}^\infty({\R_+^\times})$.
\end{enumerate}
\end{theorem}
\begin{remark}
An equivalent statement of the previous theorem is
\begin{enumerate}
\item
$H_{12}(t)=h_{12}(e^{-t})$ and $H_{21}(t)=h_{21}(e^{-t})$ are two elements of $\mathcal{C}_{\rm exp}^\infty(\R)$ satisfying the functional equation
\begin{equation*}
H_{12}\left(-t\right)=-\varepsilon e^{-t}H_{21}(t),
\end{equation*}
\item
$H_{12}(t)$ and $H_{21}(t)$ are limits, in the sense of compact uniform convergence, of sums of exponential polynomials\footnote{Remember that exponentials polynomials in $\mathcal{C}(\R)$ are given by $P(t)e^{\lambda t}$.} in $\mathcal{C}(\R)$.
\end{enumerate}
\end{remark}
\begin{remark}
$D_i(s)$ ($i=1,2$) have infinitely many poles in practice. Thus, applying this theorem will require some effort to check \eqref{314}.
\end{remark}
\begin{remark}
Of course, the estimate \eqref{314} can be relaxed if we have some stronger estimate for $\gamma(s)$, for instance some exponential decay in vertical strip. The main condition is that one can choose a sequence $\{t_m\}_{m \geq 1}$ such that $\abs{Z(\sigma+it_m)}=O(t_m^{-a})$ uniformly in $\abs{\Re{(s)}-1/2}\leq w+\delta$, $\abs{t}\geq t_0$ and for some $a>1$.
\end{remark}
\begin{proof}[\proofname{} of Theorem~\ref{thm_302}]%
Equation \eqref{317} is essentially trivial. From definition \eqref{315}, $f_{Z_1}(x)$ and $f_{Z_2}(x)$ are continuous functions on ${\R}_+^\times$, which entails that $h_{12}(x)$ and $h_{21}(x)$ are also continuous functions on ${\R}_+^\times$ by \eqref{316}. The functional equation \eqref{317} is an immediate consequence of \eqref{316} and \eqref{316bis}. The second assertion is implied by the first one and the fact that
\[
h_{12}(x),h_{21}(x)=
\begin{cases}
\, O\left(x^{-1/2+w+\delta}\right) & \text{as $x\to+\infty$}, \\
\, O\left(x^{-1/2-w-\delta}\right) & \text{as $x\to0^+$}.
\end{cases}
\]
since $f_{Z_i}(x)=O(x^{-1/2-w-\delta})$ ($i=1,2$) for some $\delta>0$. Let us focus on the first assertion for  $h_{12}$ only. We consider the clockwise oriented rectangle $\mathcal{R}=[1/2-w-\delta,1/2+w+\delta]\times[-T,T]$. We get
\begin{equation*}
\frac{1}{2i\pi}\int_{\mathcal{R}}Z_1(s)x^{-s}\dd s=\sum_{\substack{\text{$\lambda$ pole of $Z_1$} \\
\text{of multiplicity $m_\lambda$} \\
\abs{\Im{(\lambda)}}<T}}\sum_{m=1}^{m_\lambda}C_m(\lambda)\frac{(-1)^{m-1}}{(m-1)!}\log^{m-1}{(x)}x^{-\lambda}
\end{equation*}
where $C_m(\lambda)\neq 0$ ($1\leq m\leq m_\lambda$) are defined in \eqref{principal_part}. We cut the rectangle symmetrically at the points $1/2\pm iT$. By the functional equation \eqref{311}, the integral on the left part of this rectangle equals $-\epsilon_1\epsilon_2x^{-1}$ times the same integral on the right part of this rectangle in which $Z_1$ is replaced by $Z_2$ and $x$ is replaced by $x^{-1}$. Thus,
\begin{multline}\label{401}
h_{12}(x)=\sum_{\substack{\text{$\lambda$ pole of $Z_1$} \\
\text{of multiplicity $m_\lambda$} \\
\abs{\Im{(\lambda)}}<T}}\sum_{m=1}^{m_\lambda}C_m(\lambda)\frac{(-1)^{m-1}}{(m-1)!}\log^{m-1}{(x)}x^{-\lambda} \\
+R_1(Z_1,x,T)+R_2(Z_1,x,T)-\epsilon_1\epsilon_2x^{-1}(R_1(Z_2,x^{-1},T)-R_2(Z_2,x^{-1},T))
\end{multline}
where $R_1(Z_1,x,T)$ (and similarly $R_1(Z_1,x^{-1},T)$) is given by
\begin{multline} \label{402}
R_1(Z_1,x,T)=-\frac{1}{2i\pi}\int_{1/2}^{1/2+w+\delta}Z_1(\sigma+iT)x^{-(\sigma+iT)}\dd\sigma \\
+\frac{1}{2i\pi}\int_{1/2}^{1/2+w+\delta}Z_1(\sigma-iT)x^{-(\sigma-iT)}\dd\sigma
\end{multline}
and $R_2(Z_1,x,T)$ (and similarly $R_2(Z_2,x^{-1},T)$) by
\begin{multline*}
R_2(Z_1,x,T)=\frac{1}{2\pi}\int_{T}^{\infty}Z_1(1/2+w+\delta+it)x^{-1/2-w-\delta-it}\dd t \\
+\frac{1}{2\pi}\int_{-\infty}^{-T}Z_1(1/2+w+\delta+it)x^{-1/2-w-\delta-it}\dd t.
\end{multline*}
Equation \eqref{eq_AC} implies that for $i=1,2$,
\begin{equation*}
R_2(Z_i,x,T)=O\left(T^{-A}\right)
\end{equation*}
uniformly on every  compact set of $\R_+^\times$ and for every  large $A>0$. 
By \eqref{313} and  \eqref{314}, we can take an increasing sequence $\left(t_m\right)_{m \geq 1}$ so that for $i=1,2$
\begin{equation} \label{406}
R_1(Z_i,x,t_m) = O(t_m^{-A})
\end{equation}
uniformly on every  compact set of $\R_+^\times$ and for every  large $A>0$. As a consequence,
\begin{equation}\label{eq_series}
h_{12}(x)=\sum_{\substack{\text{$\lambda$ pole of $Z_1$} \\
\text{of multiplicity $m_\lambda$}}}\sum_{m=1}^{m_\lambda}C_m(\lambda)\frac{(-1)^{m-1}}{(m-1)!}\log^{m-1}{(x)}x^{-\lambda}
\end{equation}
in the sense of the compact uniform convergence.
\end{proof}
A closer inspection of the proof of the previous theorem reveals that, in some particular case, we also establish a kind of summation formula for the poles of the functions belonging to $\mathcal{F}$. 
This formula is described in the following corollary.
\begin{corollary} \label{cor_201}
Let $Z(s)=\gamma(s)D(s)$ be an element of $\mathcal{F}$, 
$\epsilon=\pm 1$ and $\phi$ be a real-valued smooth compactly supported function on $\R_+^\times$. If $Z(s)$ satisfies  the functional equation
\begin{equation*}
Z(s)=\varepsilon Z(1-s)
\end{equation*}
and
\begin{equation*}
D(s)=\sum_{m\geq 1}\frac{d_m}{m^s}
\end{equation*}
for $\Re{(s)}>\sigma_0$ then
\begin{equation*}
\sum_{\substack{\text{$\lambda$ pole of $Z$} \\
\text{of multiplicity $m_\lambda$}}}\sum_{m=1}^{m_\lambda}\frac{C_m(\lambda)}{(m-1)!}\Me(\phi)^{(m-1)}(\lambda) 
=\sum_{m\geq 1}d_m\left[(\phi \ast_\times \kappa)(m)-\varepsilon(\phi^\vee\ast_\times\kappa)(m)\right)
\end{equation*}
where $\phi^\vee(x)=x^{-1}\phi(x^{-1})$, $\kappa(x)$ is the inverse Mellin transform of $\gamma(s)$ namely
\begin{equation*}
\kappa(x)\coloneqq \frac{1}{2i\pi}\int_{(c)}\gamma(s)x^{-s}\dd s
\end{equation*}
for $c>\sigma_0$.
\end{corollary}
\begin{proof}[\proofname{} of Corollary~\ref{cor_201}]%
If $Z_1(s)=Z_2(s)=Z(s)$ then we notice that
\[
f_Z(x)=\sum_{m\geq 1}d_m\frac{1}{2i\pi}\int_{(c)}\gamma(s)(mx)^{-s}\dd s=\sum_{m\geq 1}d_m\kappa(mx),
\]
which implies 
\begin{equation*}
h_{12}(x)=\sum_{m\geq 1}d_m\left[\kappa(mx)-\epsilon x^{-1}\kappa(mx^{-1})\right].
\end{equation*}
Equation \eqref{eq_series} implies that
\begin{equation*}
\sum_{\substack{\text{$\lambda$ pole of $Z$} \\
\text{of multiplicity $m_\lambda$}}}\sum_{m=1}^{m_\lambda}C_m(\lambda)\frac{(-1)^{m-1}}{(m-1)!}\log^{m-1}{(x)}x^{-\lambda} 
=\sum_{m\geq 1}d_m\left[\kappa(mx)-\varepsilon x^{-1}\kappa(mx^{-1})\right].
\end{equation*}
The corollary follows from multiplying by $\phi(x^{-1})$ and integrating over $\R_+^\times$ with respect to the measure $\dd x/x$
\end{proof}
\begin{theorem} \label{thm_303}
Let $Z_1(s)=\gamma_1(s)D_1(s)$ and $Z_2(s)=\gamma_2(s)D_2(s)$ be two elements of $\mathcal{F}$ satisfying the functional equation
\begin{equation*}
Z_1(s)=\epsilon Z_2(1-s)
\end{equation*}
for some complex number $\epsilon$ of absolute value one. Let us suppose that $Z_1(s)$ and $Z_2(s)$ can be written as
\begin{equation} \label{320}
Z_i(s)=\frac{U_i(s)}{V_i(s)}
\end{equation}
for $i=1,2$ where $U_i(s)$ and $V_i(s)$ ($i=1,2$) are some entire functions satisfying the functional equations
\begin{align}
U_1(s) & =  \varepsilon_U U_2(1-s), \label{321} \\
V_1(s) & =  \varepsilon_V V_2(1-s) \label{322}
\end{align}
for some complex numbers $\epsilon_U, \varepsilon_V$ of absolute value one, and satisfying the bounds
\begin{equation}\label{323}
\abs{ U_i(\sigma+it)}, \abs{ V_i(\sigma+it)}=O_{a,b}\left(\abs{ t}^{-(1+\delta)}\right) \quad (i=1,2)
\end{equation}
for some $\delta>0$ in every  vertical strip of finite width $a\leq\sigma\leq b$ and every  $\abs{t}\geq 1$. Under the previous assumptions, the functions $h_{12}(x)$ and $h_{21}(x)$, which are defined in \eqref{316}, satisfy
\begin{equation*}
v_1\ast_\times h_{12} =0, \quad
v_2\ast_\times h_{21} =0.
\end{equation*}
where $v_i\in\mathbf{S}(\R_+^\times)$ is the inverse Mellin transform of $V_i$ for $i=1,2$.
In other words, the functions $h_{12}(x)$ and $h_{21}(x)$ are $\mathcal{C}_{\rm poly}^\infty({\R_+^\times})$-mean-periodic and $\mathbf{S}(\R_+^\times)^*$-mean-periodic.
\end{theorem}
\begin{remark}
An equivalent statement of the previous theorem is that the functions $H_{12}(t)=h_{12}(e^{-t})$ and $H_{21}(t)=h_{21}(e^{-t})$ are $\mathbf{S}(\R)^*$-mean-periodic and $\mathcal{C}_{\rm exp}^\infty(\R)$-mean-periodic.
\end{remark}
\begin{remark}
In general, $\mathcal{C}_{\rm poly}^\infty({\R_+^\times})$-mean-periodicity  
(or $\mathbf{S}(\R_+^\times)^*$-mean-periodicity) of $h_{12}(x)$ and $h_{21}(x)$  
do not imply that $Z_i(s)~(i=1,2)$ are meromorphic functions of order one. 
For example, let $\zeta_{\Gamma}(s)$ be the Selberg zeta function 
associated to a discrete co-compact torsion free subgroup $\Gamma$ of ${\rm SL}_2({\R})$. 
Then $\zeta_{\Gamma}(s)$ is an entire function of order two which has the simple zero at $s=1$,  
order $2g-1$ zero at $s=0$ and has the functional equation
\[
(\Gamma_2(s)\Gamma_2(s+1))^{2g-2} \zeta_{\Gamma}(s) = (\Gamma_2(1-s)\Gamma_2(2-s))^{2g-2} \zeta_{\Gamma}(1-s),
\]
where $g>1$ is the genus of $\Gamma \backslash {\rm SL}_2(\R) \slash {\rm SO}(2)$ 
and $\Gamma_2(s)$ is the double gamma function (\cite{MR0088511} \cite{MR0439755}, \cite{Kuro}). 
Put $\gamma(s)=(\Gamma_2(s)\Gamma_2(s+1))^{2g-2}$, $D(s)=(s(s-1))^{-2}\zeta_{\Gamma}(s)$ 
and $Z_\Gamma(s)=\gamma(s)D(s)$. 
Then we find that $Z_\Gamma(s)$ belongs to $\mathcal{F}$ and its poles are simple poles $s=0,1$ only. 
The function $h_\Gamma(x)$ attached to $Z_{\Gamma}$ is equal to $c_0+c_1x^{-1}$ for some real numbers $c_0,c_1$, 
where $h_\Gamma(x)=f_{Z_\Gamma}(x)-x^{-1}f_{Z_\Gamma}(x^{-1})$ with the inverse Mellin transform $f_{Z_\Gamma}$ of $Z_\Gamma$. 
Hence $h_\Gamma$ is $\mathcal{C}_{\rm poly}^\infty({\R_+^\times})$-mean-periodic. 
Moreover $h_\Gamma$ is $\mathcal{C}({\R_+^\times})$-mean-periodic. 
The Mellin--Carleman transform of $h_\Gamma$ is the rational function $c_1(s-1)^{-1}-c_0s^{-1}$. 
However, as mentioned above, $Z_\Gamma$ is a meromorphic function of order two. 
\end{remark}
\begin{proof}[\proofname{} of Theorem~\ref{thm_303}]%
We only prove the result for $h_{12}$. Let $u_i$ be the inverse Mellin transform of $U_i$ namely
\begin{equation*}
u_i(x)=\frac{1}{2i\pi}\int_{(c)}U_i(s)x^{-s}\dd s
\end{equation*}
and $v_i$ be the inverse Mellin transform of $V_i$ namely
\begin{equation*}
v_i(x)=\frac{1}{2i\pi}\int_{(c)}V_i(s)x^{-s}\dd s
\end{equation*}
for $i=1,2$. These integrals converge for every  real number $c$ according to \eqref{323}. In addition, these functions belong to $\mathbf{S}({\R}_+^\times)$ by shifting the contours to the right or to the left. Let us define $\widetilde{f}(x)\coloneqq x^{-1}f(x^{-1})$. We remark that
\begin{equation}\label{eq_stepA}
u_1=v_1\ast_\times f_{Z_1}=\varepsilon_U\widetilde{v_2}\ast_\times\widetilde{f_{Z_2}}
\end{equation}
since $U_1(s)=V_1(s)Z_1(s)$ and $U_1(s)=\epsilon_UU_2(1-s)=\varepsilon_U V_2(1-s)Z_2(1-s)$ according to the functional equation \eqref{321}. In addition, $v_1=\varepsilon_V\widetilde{v}_2$ by the functional equation \eqref{322}. As a consequence,
\begin{equation}\label{eq_stepB}
\varepsilon_U\widetilde{v}_2=\varepsilon v_1
\end{equation}
since $\varepsilon=\varepsilon_U {\varepsilon_V}^{-1}$. Equations \eqref{eq_stepA} and \eqref{eq_stepB} altogether imply
\[
v_1\ast_\times (f_{Z_1}-\varepsilon\widetilde{f_{Z_2}})=0
\]
which is the desired result since $\mathbf{S}({\R_+^\times}) ~\prec~ \mathcal{C}_{\rm poly}^\infty({\R_+^\times})^\ast$.
\end{proof}
\subsection{On single sign property for these mean-periodic functions}%
\begin{proposition}\label{mono}
Suppose that a function $h(x)$ satisfies the following conditions 
\begin{itemize}
\item 
$h(x)$ is a $\mathcal{C}_{\rm poly}^\infty({\R_+^\times})$-mean-periodic real-valued function 
or a $\mathbf{S}(\R_+^\times)^*$-mean-periodic real-valued function on $\R_+^\times$ satisfying
$h(x^{-1})=-\epsilon x h(x)$, 
\item
there exists $t_0>0$ such that $h(e^{-t})$ is of constant sign on $(t_0,+\infty)$,
\item
the Mellin--Carleman transform $\MeCa\left(h\right)(s)$ of $h(x)$ has no poles in $(1/2+\delta,+\infty)$ for some $0\leq\delta<w$ where $w$ is the positive real number in Definition \ref{def_E}. 
\end{itemize}
Then all the poles of $\;\MeCa\left(h\right)(s)$ belong to the strip $\abs{\Re(s)-1/2}\leq\delta$. In particular, if $\;\MeCa\left(h\right)(s)$ does not have any poles on $(1/2,+\infty)$ namely $\delta=0$ then all the poles of $\;\MeCa(h)(s)$ are on the line $\Re{(s)}=1/2$.
\end{proposition}
Proposition \ref{mono} is a consequence of the following lemma (\cite[Chapter II, Section 5]{MR0005923}) since
\begin{equation*}
\MeCa(h)(s)=\int_0^1h(x)x^{s}\frac{\dd x}{x}=\int_0^{+\infty}h(e^{-t})e^{-st}\dd t
\end{equation*}
on $\Re{(s)}>1/2+w$, and $\MeCa(h)(s)= \varepsilon \MeCa(h)(s)$ 
by the first assumption and Proposition \ref{lemma_mpC}. 

Applications of Proposition \ref{mono} can be seen in Remark~\ref{rem511},
the proof of Proposition~\ref{prop_502} 
and the proof of Proposition ~\ref{prop_503}. 

\begin{lemma}\label{lem_Laplace}
Let $f(x)$ be a real-valued function on $\R_+^\times$. If there exists $x_0>0$ such that $f(x)$ is of constant sign on $(x_0,+\infty)$ and if the abscissa $\sigma_c$ of convergence (not of absolute convergence) of the Laplace transform $\La(f)(s)$ of $f(x)$ is finite then $\La(f)(s)$ has a singularity at $s=\sigma_c$.
\end{lemma}
\section{Zeta functions of arithmetic schemes and mean-periodicity}\label{sec_zeta}%
The main references for the arithmetic and analytic objects briefly introduced in this section are \cite{MR0194396} and \cite{Se}.
\subsection{Background on  zeta functions of schemes}%
Let $S$ be a scheme of dimension $n$. Its (Hasse) \emph{ zeta function}
is the Euler product
\begin{equation*}
\zeta_S(s)=\prod_{x\in S_0}(1-\abs{k(x)}^{-s})^{-1}
\end{equation*}
whose Euler factors correspond to all closed points $x$ of $S$, say $x\in S_0$, with residue field of cardinality $\abs{k(x)}$.
It is known to converge absolutely in $\Re(s)>n$.
If $S$ is a $B$-scheme
then the zeta function $\zeta_S(s)$ is the product of
the zeta functions $\zeta_{S_b}(s)$ where $S_b$ runs through  all fibres
of $S$ over $B$. 
\newline\newline
Let $\K$ be a number field. Let $E$ be an elliptic curve  over $\K$. 
Define the (Hasse--Weil) zeta function $\zeta_E(s)$ of $E$
as the product of factors for each valuation of $\K$, 
the factors are the Hasse zeta function of a minimal Weierstrass equation 
of $E$ with respect to the valuation. 
If $E$ has a global minimal Weierstrass equation 
(for example, this is so if the class number of $\K$ is 1), then 
the Hasse--Weil zeta function $\zeta_E(s)$ equals the zeta function
$\zeta_{{\mathcal E}_0}$ of the model ${\mathcal E}_0$ corresponding to a 
global minimal Weierstrass equation for $E$.
All this easily follows from the description of the special fibre of a minimal 
Weierstrass equation, see e.g. \cite[10.2.1]{Li}. 
The Hasse--Weil zeta function $\zeta_E(s)$ depends on $E$ only. 
\newline
The factor of the  zeta function for each valuation of $\K$
is  the zeta function of appropriate curve (almost always  elliptic curve) over the residue field
of the valuation. It can naturally be written as the quotient whose numerator is the zeta function of the projective line over the residue field. Following Hasse, take the product over all $v$ to get 
\begin{equation}
\zeta_E(s)=\frac{\zeta_\K(s)\zeta_\K(s-1)}{L(E,s)}
\end{equation}
on $\Re{(s)}>2$ where $\zeta_\K(s)$ is the Dedekind zeta function of $\K$, 
and which defines $L(E,s)$, the $L$-function of $E$. 
\newline\newline
Let $\mathcal{E}$ be a regular model of $E$, proper over the ring of integers of $\K$. 
The description of geometry of models
in \cite[Thms 3.7, 4.35 in Ch. 9 and Section 10.2.1 in Ch. 10]{Li}
immediately implies that
\begin{equation}\label{701}
\zeta_\mathcal{E}(s)=n_\mathcal{E}(s)\zeta_E(s)
\end{equation}
where $n_{\mathcal{E}}(s)$ is the product of zeta functions of affine lines
over finite extension $k(b_j)$ of the residue fields $k(b)$:
\begin{equation}\label{702}
n_\mathcal{E}(s)=\prod_{1\leq j\leq J}\left(1-q_j^{1-s}\right)^{-1}
\end{equation}
and  $q_j=\abs{k(b_j)}$ ($1\leq j\leq J$), $J$ is the number of singular fibres of $\mathcal{E}$.
See  \cite[Section 7.3]{Fe3} and also \cite[Section 1]{MR899399}.
Note that $n_\mathcal{E}(s)^{\pm 1}$ are holomorphic functions on $\Re{(s)}>1$.
\newline\newline
In the next sections we look at the $L$-functions of elliptic curves, Hasse-Weil zeta functions of elliptic curves and zeta functions of regular models of elliptic curves, and then at the  zeta functions of arithmetic schemes.
%
%
\subsection{Conjectural analytic properties of $L$-functions of elliptic curves}%
Let $\K$ be a number field. Let $E$ be an elliptic curve over $\K$, denote its conductor by  $\mathfrak{q}_E$. The $L$-function $L(E,s)$ has an absolutely convergent Euler product and Dirichlet series on $\Re{(s)}>3/2$, say
\begin{equation}\label{eq_LE_Dir}
L(E,s)\coloneqq \sum_{n\geq 1}\frac{a_n}{n^{s}}.
\end{equation}
The \textit{completed} "$L$-function" $\Lambda(E,s)$ of $E$ is defined by
\begin{equation*}
\Lambda(E,s)\coloneqq \left(\text{N}_{\K\vert\Q}(\mathfrak{q}_E)\abs{ d_\K}^2\right)^{s/2}L_\infty(E,s)L(E,s)
\end{equation*}
where $d_\K$ is the discriminant of $\K$ and
\begin{equation*}
L_\infty(E,s)\coloneqq \fGamma_\C(s)^{r_1}\fGamma_\C(s)^{2r_2}.
\end{equation*}
Here, $r_1$ is the number of real archimedean places of $\K$, $r_2$ is the number of conjugate pairs of complex archimedean places of $\K$, and $\fGamma_\C(s)\coloneqq (2\pi)^{-s}\fGamma(s)$ as usual (see \cite[Section 3.3]{Se}). The expected analytic properties of $\Lambda(E,s)$ are encapsulated in the following hypothesis.
\begin{hypothesisE} (respectively $\NiceWE(\K)$)
If $E$ is an elliptic curve over the number field $\K$ then the function $\Lambda(E,s)$ is a \emph{completed $L$-function} (respectively \emph{almost completed $L$-function}) in the sense that it satisfies the following \emph{nice} analytic properties:
\begin{itemize}
\item
it can be extended to a holomorphic function (respectively meromorphic function with finitely many poles) of order $1$ on $\C$,
\item
it satisfies a functional equation of the shape
\[%
\Lambda(E,s)=\omega_E\Lambda(E,2-s)
\]
for some sign $\omega_E=\pm 1$.
\end{itemize}
\end{hypothesisE}
\begin{remark}
If $E$ is an elliptic curve over a general number field $\K$ with complex multiplication then its completed $L$-function is nice by the work of Deuring. If  $\K=\Q$ then Hypothesis $\NiceE(\Q)$ is implied by the theorem of Wiles and others that an elliptic curve over the field of rational numbers is \textit{modular}. More generally, extensions of the modularity lifting
property is expected to give meromorphic continuation and functional equation
for $L$-functions of elliptic curves over totally real fields. 
However, this method cannot  handle elliptic curves over arbitrary number fields. 
\end{remark}
\begin{remark}
Assuming Hypothesis $\NiceE(\K)$, the $L$-function $L(E,s)$ of every  elliptic curve $E$ over $\K$ satisfies the \emph{convexity bounds}
\begin{equation}\label{eq_conv_L_E}
L(E,s)\ll_{E,\epsilon}\left[\abs{\Im{(s)}}^{r_1+2r_2}\right]^{\mu_E\left(\Re{(s)}\right)+\epsilon}
\end{equation}
for every $\epsilon>0$ where
\begin{equation*}
\mu_E(\sigma)=\begin{cases}
0 & \text{if $\sigma\geq 3/2$,} \\
-\sigma+3/2 & \text{if $1/2\leq\sigma\leq 3/2$,} \\
2(1-\sigma) & \text{otherwise.}
\end{cases}
\end{equation*}
 (see \cite[Equation (5.21)]{IK}). 
Note that even if Hypothesis $\NiceE(\K)$ is relaxed to Hypothesis $\NiceWE(\K)$ then $L(E,s)$ is still polynomially bounded (see Section 5.5). 
\end{remark}
One of the purposes of this section is to establish a strong link between Hypothesis $\NiceE(\K)$ and mean-periodicity. This will be achieved by investigating analytic properties of (Hasse) \emph{zeta functions} of elliptic curves, in agreement with the philosophy of \cite{Fe2}.
\subsection{Hasse--Weil zeta functions of elliptic curves}%
Let $\K$ be a number field. Remember that the completed Dedekind zeta function of $\K$ is given by
\begin{equation*}
\Lambda_\K(s)\coloneqq \abs{ d_\K}^{s/2}\zeta_{\K,\infty}(s)\zeta_\K(s)
\end{equation*}
where
\begin{equation}\label{eq_zeta_infinity}
\zeta_{\K,\infty}(s)\coloneqq \fGamma_\R(s)^{r_1}\fGamma_\C(s)^{r_2}
\end{equation}
with $\fGamma_\R(s)=\pi^{-s/2}\fGamma(s/2)$ as usual and $\fGamma_\C(s)$ has already been defined in the previous section. $\Lambda_\K(s)$ is a meromorphic function of order $1$ on $\C$ with simple poles at $s=0,1$, which satisfies the functional equation
\begin{equation*}
\Lambda_\K(s)=\Lambda_\K(1-s).
\end{equation*}
Moreover, $\zeta_\K(s)$ satisfies the \emph{convexity bounds}
\begin{equation}\label{eq_conv_Zeta}
\zeta_\K(s)\ll_{\K,\epsilon}\left[\abs{\Im{(s)}}^{r_1+2r_2}\right]^{\mu_\K\left(\Re{(s)}\right)+\epsilon}
\end{equation}
for all $\epsilon>0$ where
\begin{equation*}
\mu_\K(\sigma)=\begin{cases}
0 & \text{if $\sigma\geq 1$,} \\
(-\sigma+1)/2 & \text{if $0\leq\sigma\leq 1$,} \\
1/2-\sigma & \text{otherwise.}
\end{cases}
\end{equation*} 
(see \cite[Equation (5.21)]{IK}). Let $E$ be an elliptic curve over $\K$. The Hasse--Weil zeta function of $E$ may be written in terms of the completed zeta-functions and $L$-function as follows
\begin{equation}\label{eq_Ded2}
\zeta_E(s)=\abs{ d_\K}^{1/2}\text{N}_{\K\vert\Q}(\mathfrak{q}_E)^{s/2}\frac{L_\infty(E,s)}{\zeta_{\K,\infty}(s)\zeta_{\K,\infty}(s-1)}\frac{\Lambda_\K(s)\Lambda_\K(s-1)}{\Lambda(E,s)}.
\end{equation}
The functional equation  $\fGamma(s+1)=s\fGamma(s)$ and the equality $2\fGamma_\C(s)=\fGamma_\R(s)\fGamma_\R(s+1)$, which is implied by Legendre's duplication formula, lead to
\begin{equation} \label{feq_HW} %
\zeta_E(s)=\abs{ d_\K}^{1/2}\text{N}_{\K\vert\Q}(\mathfrak{q}_E)^{s/2}\left(\frac{s-1}{4\pi}\right)^{r_1}\left(\frac{s-1}{2\pi}\right)^{r_2}\frac{\Lambda_\K(s)\Lambda_\K(s-1)}{\Lambda(E,s)}.
\end{equation}
As a consequence, Hypothesis $\NiceE({\K})$ implies the following hypothesis.
\begin{hypothesisHW}
For every elliptic curve $E$ over the number field $\K$, the Hasse--Weil zeta function $\zeta_E(s)$ satisfies the following \emph{nice} analytic properties:
\begin{itemize}
\item
it can be extended to a meromorphic function on $\C$, 
\item
it satisfies a functional equation of the shape
\[%
\left(\text{N}_{\K\vert\Q}(\mathfrak{q}_E)^{-1}\right)^{s/2}\zeta_E(s)=\left(\text{N}_{\K\vert\Q}(\mathfrak{q}_E)^{-1}\right)^{(2-s)/2}(-1)^{r_1+r_2}\omega_E\zeta_E(2-s)
\]
for some sign $\omega_E=\pm 1$.
\end{itemize}
\end{hypothesisHW}
\begin{remark}
Note that the constants in the conjectural functional equations of the Hasse--Weil zeta functions of elliptic curves are much simpler than in the conjectural functional equations of $L$-functions of elliptic curves. In particular, they do not depend on the discriminant of the field. Also note the absence of gamma-factors in the functional equations  of the Hasse--Weil zeta functions of elliptic curves. Even the total conjectural sign in the functional equations does not depend on archimedean data associated to $\K$. The Hasse--Weil zeta functions of elliptic curves are, from several points of view, more basic objects than the $L$-functions of elliptic curves.
\end{remark}
\begin{remark}
Hypothesis $\NiceHW(\K)$ implies the meromorphic continuation of $\Lambda(E,s)$ 
and the conjectural functional equation of $\Lambda(E,s)$ 
for every elliptic curve $E$ over the number field $\K$. 
In particular, Hypothesis $\NiceHW(\K)$ recover the gamma-factor $L_\infty(E,s)$ 
and the norm of the conductor $N_{{\K}\vert{\Q}}(\mathfrak{q}_E)\abs{d_{\K}}^2$ of $L(E,s)$ 
by $\eqref{feq_HW}$ and the one-dimensional study of $\zeta_{\K}(s)$.   
\end{remark}
\begin{remark}
The automorphic property
of the $L$-function does not transfer to any automorphic property of the whole ratio,
the Hasse--Weil zeta function of elliptic curve. 
It is then natural to wonder which replacement of automorphic property should correspond to Hasse--Weil zeta functions of elliptic curves. This work shows  mean-periodicity is one of such replacements.
\end{remark}
Here we state the following hypothesis on mean-periodicity. 
\begin{hypothesisMPHW} For every  elliptic curve $E$ over $\K$, if 
\begin{equation*}
Z_E(s)=\Lambda_\K(s)\left(\text{N}_{\K\vert\Q}(\mathfrak{q}_E)^{-1}\right)^{2s/2}\zeta_E(2s)
\end{equation*}
then the function
\begin{equation*}
h_E(x)\coloneqq f_{Z_E}(x)-(-1)^{r_1+r_2}\omega_Ex^{-1}f_{Z_E}\left(x^{-1}\right)
\end{equation*}
is $\mathcal{C}_{\rm poly}^\infty({\R_+^\times})$-mean-periodic (respectively $\mathbf{S}(\R_+^\times)^*$-mean-periodic), 
or the function $H_E(t)\coloneqq h_E(e^{-t})$ is $\mathcal{C}_{\rm exp}^\infty({\R})$-mean-periodic (respectively $\mathbf{S}(\R)^*$-mean-periodic), 
where $f_{Z_E}$ is the inverse Mellin transform of $Z_E$ defined in \eqref{315}. 
\end{hypothesisMPHW}
Clearly,  mean-periodicity for $h_E(x)$ and mean-periodicity for $H_E(t)$ are equivalent. One link with mean-periodicity is described in the following theorem. 
\begin{theorem}\label{prop_HW1}
Let $\K$ be a number field.  Then 
\begin{itemize}
\item
Hypothesis $\NiceWE(\K)$ implies Hypothesis $\MPHW(\K)$. 
\item  
Hypothesis $\MPHW(\K)$ implies Hypothesis $\NiceHW({\K})$. 
\end{itemize}
\end{theorem}
\begin{remark}
We have already mentioned that Hypothesis $\NiceE(\Q)$ holds. 
As a consequence, the two other hypothesis are also true. 
Let us give some more information on the mean-periodic functions, 
which occur in this particular case (see Corollary \ref{cor_201}). 
If $E$ is an elliptic curve over $\Q$ and
\begin{equation}\label{eq_zeta_Dir}
\zeta_E(2s)\coloneqq \sum_{m\geq 1}\frac{c_m}{m^s}
\end{equation}
then the mean-periodic function $h_E$ satisfies
\begin{multline*}
H_E(t)=h_E(e^{-t})=\sum_{\substack{\text{$\lambda$ pole of $Z_E(s)$} \\
\text{of multiplicity $m_\lambda$}}}\sum_{m=1}^{m_\lambda}C_m(\lambda)\frac{1}{(m-1)!}t^{m-1}e^{\lambda t} \\
=2\sum_{n\geq 1}\left(\sum_{d\mid n}c_d\right)\left[\exp{\left(-\pi n^2e^{-2t}\right)}+\omega_E\exp{\left(t-\pi n^2e^{2t}\right)}\right]
\end{multline*}
where the coefficients $C_m(\lambda)$ ($1\leq m\leq m_\lambda$) are defined in \eqref{principal_part}.
\end{remark}
\begin{proof}[\proofname{} of Theorem~\ref{prop_HW1}]%
The second assertion is a consequence of Theorem \ref{thm_rmp}. 
Let us show the first assertion applying Theorem \ref{thm_302} and Theorem \ref{thm_303}. 
Adopting the same notation as  in Theorem \ref{thm_302}, we choose $Z_1(s)=Z_2(s)=\gamma(s)D(s)=Z_E(s)$ with $\gamma(s)=\Lambda_\K(s)$ and $D(s)=\left(\text{N}_{\K\vert\Q}(\mathfrak{q}_E)^{-1}\right)^{2s/2}\zeta_E(2s)$. 
The functional equation satisfied by $Z_E(s)$ is
\[
Z_E(s)=(-1)^{r_1+r_2}\omega_EZ_E(1-s).
\]
In addition, $Z_E(s)$ belongs to $\mathcal{F}$ since $\Lambda_\K(s)$ has two poles and the poles of $\gamma(s)$ and $D(s)$ are in the vertical strip $\abs{\Re{(s)}-1/2}\leq 1/2\coloneqq w$ and
\begin{itemize}
\item
the estimate \eqref{313} follows from Stirling's formula and classical convexity bounds for Dedekind zeta functions given in \eqref{eq_conv_Zeta},
\item
the estimate \eqref{eq_AC} follows from the Dirichlet series expansion of $\zeta_E(s)$ for $\Re(s)>2$, 
\item
the crucial condition \eqref{314} is an application of Proposition \ref{lem_601} and the convexity bounds for the Dedekind zeta function given in \eqref{eq_conv_Zeta}.
\end{itemize} 
From Hypothesis $\NiceWE(\K)$, the function $P(s)L(E,s)$ is an entire function for some polynomial $P(s)$ satisfying $P(s)=P(1-s)$. 
Adopting the same notation as  in Theorem \ref{thm_303}, we choose $U_1(s)=U_2(s)=U(s)$ and $V_1(s)=V_2(s)=V(s)$ where
\begin{eqnarray*}
U(s) & \coloneqq  & (4\pi)^{-r_1}(2\pi)^{-r_2}\abs{ d_\K}^{1/2}(2s-1)^{r_1+r_2+1}s^2(s-1)^2\Lambda_\K(s)\Lambda_\K(2s)\Lambda_\K(2s-1)P(2s),  \\
V(s) & \coloneqq  & (2s-1)s^2(s-1)^2 \Lambda(E,2s)P(2s).
\end{eqnarray*}
$U(s)$ and $V(s)$ are  entire functions satisfying the functional equations
\begin{eqnarray*}
U(s) & = & (-1)^{r_1+r_2+1}U(1-s), \\
V(s) & = & -\omega_EV(1-s).
\end{eqnarray*}
The estimate \eqref{323} is a consequence of Stirling's formula and convexity bounds for $P(s)L(E,s)$ and the Dedekind zeta function given in \eqref{eq_conv_L_E} and in \eqref{eq_conv_Zeta}.
\end{proof}
We get, arguing along the same lines, the following proposition.
\begin{proposition}\label{prop_HW2}
Let $\K$ be a number field and $E$ be an elliptic curve over $\K$. 
\begin{itemize}
\item
If Hypothesis $\NiceWE(\K)$ holds then the function
\begin{equation*}
h_E^{(2)}(x)\coloneqq f_{Z_E^2}(x)-x^{-1}f_{Z_E^2}\left(x^{-1}\right)
\end{equation*}
with $Z_E(s)\coloneqq \Lambda_\K(s)\left(\text{N}_{\K\vert\Q}(\mathfrak{q}_E)^{-1}\right)^{2s/2}\zeta_E(2s)$ 
is $\mathcal{C}_{\rm poly}^\infty({\R_+^\times})$-mean-periodic (respectively $\mathbf{S}(\R_+^\times)^*$-mean-periodic), 
and the function $H_E^{(2)}(t)\coloneqq h_E^{(2)}(e^{-t})$ is $\mathcal{C}_{\rm exp}^\infty({\R})$-mean-periodic (respectively $\mathbf{S}(\R)^*$-mean-periodic), 
where $f_{Z_E^2}$ is the inverse Mellin transform of $Z_E^2$ defined in \eqref{315}.
\item
If $h_E^{(2)}(x)$ is $\mathcal{C}_{\rm poly}^\infty({\R_+^\times})$-mean-periodic 
or $\mathbf{S}(\R_+^\times)^*$-mean-periodic, 
or $H_E^{(2)}(t)$ is $\mathcal{C}_{\exp}^\infty({\R})$-mean-periodic or $\mathbf{S}(\R)^*$-mean-periodic, 
then $\zeta_E^2(s)$ extends to a meromorphic function on $\C$, which satisfies the functional equation
\begin{equation*}
\left(\text{N}_{\K\vert\Q}(\mathfrak{q}_E)^{-2}\right)^{s/2}\zeta_E^2(s)=\left(\text{N}_{\K\vert\Q}(\mathfrak{q}_E)^{-2}\right)^{(2-s)/2}\zeta_E^2(2-s).
\end{equation*}
\end{itemize}
\end{proposition}
\begin{remark}\label{rem_info_zero}
If $E$ is an elliptic curve over $\Q$ of conductor $q_E$, which only satisfies Hypothesis $\NiceE(\Q)$, then the mean-periodic function $h_E^{(2)}$ satisfies
\begin{multline}\label{eq_explicit2}
h_E^{(2)}(e^{-t})=\sum_{\substack{\text{$\lambda$ pole of $Z_E^2(s)$} \\
\text{of multiplicity $m_\lambda$}}}\sum_{m=1}^{m_\lambda}C_m(\lambda)\frac{1}{(m-1)!}t^{m-1}e^{\lambda t} \\
=4 \sum_{n\geq 1}\left(\sum_{d\mid n}c_d\sigma_0(n/d)\right)\left[K_0(2\pi n e^{-t})-e^{t}K_0(2\pi n e^{t})\right]
\end{multline}
since
\begin{equation*}
\frac{1}{2i\pi}\int_{(c)}\Lambda_\K(s)^2x^{-s}\dd s=4\sum_{n=1}^{\infty}\sigma_0(n)K_0\left(2\pi nx\right)
\end{equation*}
where $K_0$ is the modified Bessel function, the coefficients $C_m(\lambda)$ ($1\leq m\leq m_\lambda$) are defined in \eqref{principal_part}, the coefficients $(c_m)_{m\geq 1}$ are defined in \eqref{eq_zeta_Dir} and $\sigma_0(n)=\sum_{d\mid n}1$ for   integer $n\geq 1$ as usual. In addition, the function
\begin{equation}\label{eq_v}
v(x)\coloneqq \frac{1}{2i\pi}\int_{(c)}(2s-1)^2s^4(s-1)^4\Lambda(E,2s)^2x^{-s}\dd s=\sum_{n\geq 1}a_nW\left(\frac{n^2x}{q_E^{2}}\right)
\end{equation}
belongs to $\mathbf{S}(\R_+^\times)$ and satisfies $v\ast_\times h_E^{(2)}=0$ where the coefficients $(a_n)_{n\geq 1}$ are defined in \eqref{eq_LE_Dir} and
\begin{equation*}
W(x)\coloneqq \frac{1}{2i\pi}\int_{(c)}(2s-1)^2s^4(s-1)^4\fGamma(2s)^2x^{-s}\dd s.
\end{equation*}
\end{remark}
\begin{remark}
Using the series representation in the right-hand side of \eqref{eq_explicit2} and the right-hand side of \eqref{eq_v} 
one can directly check that $h_E^{(2)}(x) \in\mathcal{C}_{\rm poly}^\infty({\R_+^\times})$, 
$v(x)\in {\mathcal{C}}^\infty({\R}_+^\times)$ and $v(x)$ is of rapid decay as $x\to+\infty$. 
If one can prove that $v \ast_\times h_E^{(2)}=0$ and that $v(x)$ is of rapid decay as $x\to 0^+$ 
then it implies the meromorphic continuation of $\zeta_E^2(s)$ and its functional equation, without using the modularity property  of $E$.
\end{remark}
\begin{remark}\label{rem511}
Let $E$ be an elliptic curve over $\Q$ and let $h_E^{(2)}(x)$ be $\mathcal{C}_{\rm poly}^\infty({\R_+^\times})$-mean-periodic. According to \eqref{eq_explicit2}, we can split $h_E^{(2)}(x)$ into
\begin{equation*}
h_E^{(2)}(x)\coloneqq h_{E,0,0}^{(2)}(x)+h_{E,0,1}^{(2)}(x)+h_{E,1}^{(2)}(x)
\end{equation*}
where
\begin{eqnarray*}
h_{E,0,0}^{(2)}(x)  & \coloneqq  & \sum_{m=1}^{4}C_m(0)\frac{(-1)^{m-1}}{(m-1)!}\log^{m-1}{(x)}, \\
h_{E,0,1}^{(2)}(x)  & \coloneqq  & \sum_{m=1}^{4}C_m(1)\frac{(-1)^{m-1}}{(m-1)!}\log^{m-1}{(x)}x^{-1}
\end{eqnarray*}
and
\begin{equation}
h_{E,1}^{(2)}(x)\coloneqq \sum_{\substack{\text{$\lambda$ pole of $Z_E^2(s)$} \\
\text{of multiplicity $m_\lambda$} \\
\lambda\neq 0,1}}\sum_{m=1}^{m_\lambda}C_m(\lambda)\frac{(-1)^{m-1}}{(m-1)!}\log^{m-1}{(x)}x^{-\lambda}.
\end{equation}
All these three functions are $\mathcal{C}_{\rm poly}^\infty({\R_+^\times})$-mean-periodic by the assumption. As a consequence, \begin{equation}
\MeCa\left(h_E^{(2)}\right)(s)=\MeCa\left(h_{E,0,0}^{(2)}\right)(s)+\MeCa\left(h_{E,0,1}^{(2)}\right)(s)+\MeCa\left(h_{E,1}^{(2)}\right)(s), 
\end{equation} 
where $\MeCa\left(h_{E,0,0}^{(2)}\right)(s)$ has only one pole of order four at $s=0$, $\MeCa\left(h_{E,0,1}^{(2)}\right)(s)$ has only one pole of order four at $s=1$ and the poles of $\MeCa\left(h_{E,1}^{(2)}\right)(s)$ are given by the non-trivial zeros of $L(E,2s)$ according to the conjectural linear independence of zeros of $L$-functions (see \cite[Page 13]{MR1954010}). Proposition \ref{mono} entails that if $h_{E,1}^{(2)}(e^{-t})$ is of constant sign\footnote{This function is said to satisfy the \emph{single sign property}.} on $(t_0,+\infty)$ for some real number $t_0$ and if ${\mathsf MC}\left(h_{E,1}^{(2)}\right)(s)$ does not have  poles on the real axis except at $s=1/2$ then all the poles of $\MeCa\left(h_{E,1}^{(2)}\right)(s)$ are on the line $\Re(s)=1/2$. 
In other words, $L(E,s)$ satisfies the Generalized Riemann Hypothesis. 

We would like to provide some evidence\footnote{For numerical computations, see \url{http://www.maths.nott.ac.uk/personal/ibf/comp.html}} for the single sign property of $h_{E,1}^{(2)}(e^{-t})$ to  hold. Set $H_E(t)\coloneqq -e^{-t}h_{E}^{(2)}(e^{-t})$,  then the  function $H_E(t)$ coincides with the function $H(t)$ defined in \cite[Section 8.1]{Fe3}.
Then the single sign property of $H_E^{\prime\prime\prime\prime}(t)$ implies the single sign property of $h_{E,1}^{(2)}(e^{-t})$. On the other hand, the single sign of $H_E^{\prime\prime\prime\prime}(t)$ holds under the Generalized Riemann Hypothesis for $L(E,s)$ if all the non trivial zeros, except $s=1$, of $L(E,s)$ are simple and $E$ is not of analytic rank $0$ (see \cite[Proposition 4]{Su1}).
\end{remark}
%
\begin{remark}\label{rem_NOT}
Finally, let us explain why  the functions $h_E(e^{-t})$ and $h_E^{(2)}(e^{-t})$ cannot be $\mathcal{C}(\R)$-mean-periodic neither $\mathcal{C}^\infty(\R)$-mean-periodic. For instance, let us assume that $H_E(t)=h_E(e^{-t})$ is $\mathcal{C}(\R)$-mean-periodic and that $\Lambda(E,s)$ is nice. We can choose a non-trivial compactly supported measure $\mu$ on $\R$ satisfying $H_E\ast\mu=0$. According to the explicit formula \eqref{eq_explicit2}, the poles of the Laplace--Carleman transform $\LaCa(H_E)(s)$ are exactly the poles of $Z_E(s)$ with multiplicities. Thus, if $\lambda\in\C\setminus\{0,1/2,1\}$ is a pole of $\LaCa(H_E)(s)$ then $\lambda$ is a non-trivial zero of $L(E,2s)$ of multiplicity $M_\lambda\geq 1$ and a non-trivial zero of $\zeta_\K(s)$ of multiplicity $n_\lambda<M_\lambda$. Let $\Lambda_E$ be the multiset (with multiplicities) of poles of $\LaCa(H_E)(s)$ (except $0$,$1/2$ and $1$ as previously) and let $\mathcal{Z}_E$ be the multiset (with multiplicities) of non-trivial zeros of $L(E,2s)$. We have just seen that $\Lambda_E\subset\mathcal{Z}_E$. On one hand, there exists a constant $C_E\neq 0$ such that
\begin{equation*}
N(R;\mathcal{Z}_E)\coloneqq \abs{\left\{\lambda\in\mathcal{Z}_E,\vert\lambda\vert\leq R\right\}}=C_ER\log{R}+O(R)
\end{equation*}
according to \cite[Theorem 5.8]{IK}. On the other hand, the set $\Lambda_E$ is a subset of the multiset (with multiplicities) of the zeros of $\DLa(\mu)(s)$ by definition of the Laplace--Carleman transform. According to \cite[Page 97]{Le}, the function $\DLa(\mu)(s)$ belongs to the Cartwright class $\mathsf{C}$, which is the set of entire functions $\psi$ of exponential type satisfying
\begin{equation}\label{eq_C}
\int_{-\infty}^{+\infty}\frac{\log^{+}{\vert \psi(t)\vert}}{1+t^2}\dd t<\infty.
\end{equation}
Here, the fact that $\mu$ is compactly supported is crucial. It implies that (\cite[Equation (5) Page 127]{Le}),
\begin{equation*}
N\left(R;\DLa(\mu)\right)\coloneqq \abs{\left\{\lambda\in\C, \DLa(\mu)=0,\vert\lambda\vert\leq R\right\}}=C_E^\prime R+o(R)
\end{equation*}
for some $C_E^\prime\neq 0$. As a consequence,
\begin{equation*}
N(R;\Lambda_E)\coloneqq \abs{\left\{\lambda\in\Lambda_E,\vert\lambda\vert\leq R\right\}}=O_E(R)
\end{equation*}
and
\begin{equation*}
N(R;\mathcal{Z}_E \setminus\Lambda_E)\simeq_{R\to+\infty}N(R;\mathcal{Z}_E).
\end{equation*}
Statistically speaking, this means that if $2\lambda$ is a zero of $L(E,s)$ of multiplicity $M_\lambda$ then $\lambda$ is a zero of $\zeta_\K(s)$ of multiplicity greater than $M_\lambda$. 
Of course, such result would not agree with the general admitted expectation that zeros of essentially different $L$-functions are linearly independent (see \cite[Page 13]{MR1954010}).
\end{remark}
%
%
\subsection{Zeta functions of models of elliptic curves}%
Let $E$ be an elliptic curve over $\K$ of conductor $\mathfrak{q}_E$ and $\mathcal{E}$ be a regular model of $E$ over $\K$. 
In the two-dimensional adelic analysis  the Hasse zeta function of $\mathcal{E}$
is studied via its lift to a zeta integral on a certain two-dimensional adelic space,
see \cite{Fe3}. We assume that $\mathcal{E}$ satisfies all the conditions given in \cite[Sections 5.3 and 5.5]{Fe3},
i.e. the reduced part of every fibre  is semistable and $E$ has good or multiplicative reduction in residue characteristic 2 and 3.
If $f_0$ is a well-chosen test function in the Schwartz--Bruhat space on some two-dimensional adelic space then the two-dimensional zeta integral $\zeta_\mathcal{E}(f_0,s)$ defined in
\cite[Section 5]{Fe3}
equals
\begin{equation} \label{703}
\zeta_\mathcal{E}(f_0,s)=\prod_{1\leq i\leq I}\Lambda_{\K_i}(s/2)^{2} c_\mathcal{E}^{1-s}\zeta_\mathcal{E}(s)^2
\end{equation}
where $i$ runs through finitely many indices
and $\K_i$ are finite extensions of $\K$ which include
$\K$ itself,
\begin{equation}\label{704}
c_\mathcal{E}\coloneqq \prod_{\text{$y$ singular}}\mathfrak{k}_y.
\end{equation}
Let us say a few words on the constant $c_\mathcal{E}$. We know that
\begin{equation*}
\text{N}_{\K\vert\Q}(\mathfrak{q}_E)=\prod_{\text{$y=\mathcal{E}_b$ singular}}\abs{ k(b)}^{f_b}
\end{equation*}
according to \cite[Section 4.1]{Se} and that
\begin{equation*}
\mathfrak{k}_y\coloneqq \begin{cases}
\abs{ k(b)}^{f_b+m_b-1} & \text{if $y=\mathcal{E}_b$ is singular},  \\
1 & \text{otherwise}
\end{cases}
\end{equation*}
where $m_b$ is the number of irreducible geometric components of the fibre (\cite[Section 7.3]{Fe3}). As a consequence, the constant $c_\mathcal{E}$ satisfies
\begin{equation}\label{eq_c_E}
c_\mathcal{E}=\text{N}_{\K\vert\Q}(\mathfrak{q}_E)\prod_{1\leq j\leq J}q_j
\end{equation}
since
\begin{equation*}
\prod_{1\leq j\leq J}q_j=\prod_{b\in B_0}\abs{ k(b)}^{m_b-1}
\end{equation*}
according to \cite[Section 7.3]{Fe3}.\newline\newline
In \cite{Fe3} it is conjectured that the two-dimensional zeta integral satisfies the functional equation
\begin{equation*}
\zeta_\mathcal{E}(f_0,s)=\zeta_\mathcal{E}(f_0,2-s)
\end{equation*}
namely
\begin{equation} \label{705}
\left(c_\mathcal{E}^{-2}\right)^{s/2}\zeta_\mathcal{E}(s)^2=\left(c_\mathcal{E}^{-2}\right)^{(2-s)/2}\zeta_\mathcal{E}(2-s)^2
\end{equation}
(note that the completed rescaled zeta functions of $\K_i$ will cancel each out
in the functional equation).
It is easy to check the compatibility with the previously seen functional equations. If $\zeta_E(s)$ satisfies the functional equation
\begin{equation*}
\left(\text{N}_{\K\vert\Q}(\mathfrak{q}_E)^{-1}\right)^{s/2}\zeta_E(s)=(-1)^{r_1+r_2}\omega_E\left(\text{N}_{\K\vert\Q}(\mathfrak{q}_E)^{-1}\right)^{(2-s)/2}\zeta_E(2-s)
\end{equation*}
then
\begin{eqnarray*}
\zeta_\mathcal{E}(s) & = & n_\mathcal{E}(s)\zeta_E(s) \\
& = & (-1)^{r_1+r_2}\omega_E\text{N}_{\K\vert\Q}(\mathfrak{q}_E)^{s-1}\frac{n_\mathcal{E}(s)}{n_\mathcal{E}(2-s)}\zeta_\mathcal{E}(2-s) \\
& = & (-1)^{r_1+r_2+J}\omega_E\left(\text{N}_{\K\vert\Q}(\mathfrak{q}_E)\prod_{1\leq j\leq J}q_j\right)^{s-1}\zeta_\mathcal{E}(2-s) \\
& = & (-1)^{r_1+r_2+J}c_\mathcal{E}^{s-1}\omega_E\zeta_\mathcal{E}(2-s)
\end{eqnarray*}
according to \eqref{eq_c_E}. Thus, Hypothesis $\NiceHW(\K)$ is equivalent to the following hypothesis.
\begin{hypothesisH}
For every  elliptic curve $E$  over the number field $\K$  and regular model
 $\mathcal{E}$ of $E$ 
the  zeta function $\zeta_{\mathcal{E}}(s)$ satisfies the following \emph{nice} analytic properties:
\begin{itemize}
\item
it can be extended to a meromorphic function on $\C$,
\item
it satisfies a functional equation of the shape
\[%
\left(c_\mathcal{E}^{-1}\right)^{s/2}\zeta_\mathcal{E}(s)=(-1)^{r_1+r_2+J}\omega_E\left(c_\mathcal{E}^{-1}\right)^{(2-s)/2}\zeta_\mathcal{E}(2-s)
\]
for some sign $\omega_E=\pm 1$.
\end{itemize}
\end{hypothesisH} 
\begin{hypothesisMPH}
For every elliptic curve $E$ over $\K$, if 
\begin{equation*}
Z_\mathcal{E}(s)=
\left(\prod_{1\leq i\leq I}\Lambda_{\K_i}(s)\right)\left(c_\mathcal{E}^{-1}\right)^{2s/2}\zeta_\mathcal{E}(2s)
\end{equation*}
then the function
\begin{equation*}
h_\mathcal{E}(x)\coloneqq f_{Z_\mathcal{E}}(x)-(-1)^{r_1+r_2+J}\omega_Ex^{-1}f_{Z_\mathcal{E}}\left(x^{-1}\right)
\end{equation*}
is $\mathcal{C}_{\rm poly}^\infty({\R_+^\times})$-mean-periodic (respectively $\mathbf{S}(\R_+^\times)^*$-mean-periodic), 
or the function $H_{\mathcal{E}}(t)\coloneqq h_{\mathcal{E}}(e^{-t})$ is $\mathcal{C}_{\rm exp}^\infty({\R})$-mean-periodic (respectively $\mathbf{S}(\R)^*$-mean-periodic), 
where $f_{Z_\mathcal{E}}$ is the inverse Mellin transform of $Z_\mathcal{E}$ defined in \eqref{315}.
\end{hypothesisMPH}
Another link with mean-periodicity is described in the following theorem. 
\begin{theorem}\label{prop_H1}
Let $\K$ be a number field. 
\begin{itemize}
\item
Hypothesis $\NiceWE(\K)$ implies Hypothesis $\MPH(\K)$.
\item 
If Hypothesis $\MPH(\K)$ or Hypothesis $\MPHW(\K)$ holds 
then Hypothesis $\NiceHW(\K)$ and Hypothesis $\NiceH(\K)$ hold.
\end{itemize}
\end{theorem}
\begin{remark}
We do not describe the explicit formula for $h_\mathcal{E}(e^{-t})$ here but let us say that such formula should contain the contribution of the poles of $n_\mathcal{E}(2s)$.
\end{remark}
\begin{proof}[\proofname{} of Theorem~\ref{prop_H1}]%
We have already seen that Hypothesis $\NiceHW(\K)$ and Hypothesis $\NiceH(\K)$ are equivalent. 
Hence the second assertion is a consequence of Theorem \ref{thm_rmp}. 
Let us show the first assertion applying Theorem \ref{thm_302} and Theorem \ref{thm_303}. 
Adopting the same notations as in Theorem \ref{thm_302}, we choose $Z_1(s)=Z_2(s)=\gamma(s)D(s)=Z_\mathcal{E}(s)$ with $\gamma(s)=\prod_{1\leq i\leq I}\Lambda_{\K_i}(s)$ and $D(s)=\left(c_\mathcal{E}^{-1}\right)^{2s/2}\zeta_\mathcal{E}(2s)$. The functional equation satisfied by $Z_\mathcal{E}(s)$ is
\[
Z_\mathcal{E}(s)=(-1)^{r_1+r_2+J}\omega_EZ_\mathcal{E}(1-s).
\]
In addition, $Z_\mathcal{E}(s)$ belongs to $\mathcal{F}$ since each $\Lambda_{\K_i}(s)$ has two poles at $s=0,1$ and the poles of $\gamma(s)$ and $D(s)$ are in the vertical strip $\abs{\Re{(s)}-1/2}\leq 1/2\coloneqq w$ and
\begin{itemize}
\item
the estimate \eqref{313} follows from Stirling's formula and convexity bounds for Dedekind zeta functions given in \eqref{eq_conv_Zeta},
\item
the estimate \eqref{eq_AC} follows from the Dirichlet series expansion of $\zeta_E(s)$ for $\Re(s)>2$ 
and from the fact that $n_\mathcal{E}(2s)$ is uniformly bounded on $\Re{(s)}>1/2+w$, $w>0$,
\item
the crucial condition \eqref{314} is an application of Proposition \ref{lem_601} and the convexity bounds for the Dedekind zeta function given in \eqref{eq_conv_Zeta}. Note that $n_\mathcal{E}(2s)$ is a finite Euler product, which may have infinitely many poles only on the critical line $\Re{(s)}=1/2$ but its set of poles is a well-spaced set namely
\begin{equation*}
\amalg_{1\leq j\leq J}\frac{2\pi}{\log{q_j}}\Z.
\end{equation*}
\end{itemize}
From Hypothesis $\NiceWE(\K)$, $P(s)L(E,s)$ is an entire function for some polynomial $P(s)$ satisfying $P(s)=P(1-s)$. 
Adopting the same notations as in Theorem \ref{thm_303}, we choose $U_1(s)=U_2(s)=U(s)$ and $V_1(s)=V_2(s)=V(s)$ where
\begin{align*}
U(s) & \coloneqq   (4\pi)^{-r_1}(2\pi)^{-r_2}\abs{ d_\K}^{1/2}(2s-1)^{r_1+r_2+1}s^{1+I}(s-1)^{1+I}\left(\prod_{1\leq i\leq I}\Lambda_{\K_i}(s)\right)\Lambda_\K(2s)\Lambda_\K(2s-1)P(2s),  \\
V(s) & \coloneqq   (2s-1)s^{1+I}(s-1)^{1+I} \left(c_\mathcal{E}^{-1}\right)^{-2s/2}\left(\text{N}_{\K\vert\Q}(\mathfrak{q}_E)^{-1}\right)^{2s/2}n_{\mathcal{E}}(2s)^{-1}\Lambda(E,2s)P(2s).
\end{align*}
$U(s)$ and $V(s)$ are some entire functions satisfying the functional equations
\begin{eqnarray*}
U(s) & = & (-1)^{r_1+r_2+1}U(1-s), \\
V(s) & = & (-1)^{1+J}\omega_EV(1-s).
\end{eqnarray*}
Note that the sign $(-1)^J$, which occurs in the second functional equation, is implied by \eqref{eq_c_E}. The estimate \eqref{323} is an easy consequence of Stirling's formula and convexity bounds for $L(E,s)$ and the Dedekind zeta function given in \eqref{eq_conv_L_E} and in \eqref{eq_conv_Zeta}.
\end{proof}
We get, arguing along the same lines, the following theorem. 
\begin{theorem}\label{theor_H2}
Let $\K$ be a number field and $E$ be an elliptic curve over $\K$.
\begin{itemize}
\item
If Hypothesis $\NiceWE(\K)$ holds then the function
\begin{equation*}
h_\mathcal{E}^{(2)}(x)\coloneqq f_{Z_\mathcal{E}^2}(x)-x^{-1}f_{Z_\mathcal{E}^2}\left(x^{-1}\right),
\end{equation*}
where  $Z_{\mathcal{E}}(s)\coloneqq \left(\prod_{1\leq i\leq I}\Lambda_{\K_i}(s)\right)\left(c_\mathcal{E}^{-1}\right)^{2s/2}\zeta_\mathcal{E}(2s)$, 
is $\mathcal{C}_{\rm poly}^\infty({\R_+^\times})$-mean-periodic (respectively $\mathbf{S}(\R_+^\times)^*$-mean-periodic), 
and the function $H_{\mathcal{E}}^{(2)}(t)\coloneqq h_{\mathcal{E}}^{(2)}(e^{-t})$ is $\mathcal{C}_{\rm exp}^\infty({\R})$-mean-periodic (respectively $\mathbf{S}(\R)^*$-mean-periodic), 
where $f_{Z_\mathcal{E}^2}$  is the inverse Mellin transform of $Z_\mathcal{E}^2$ defined in \eqref{315}.
\item
If $h_\mathcal{E}^{(2)}(x)$ is $\mathcal{C}_{\rm poly}^\infty({\R_+^\times})$-mean-periodic 
or $\mathbf{S}(\R_+^\times)^*$-mean-periodic, 
or $H_{\mathcal{E}}^{(2)}(t)$ is $\mathcal{C}_{\exp}^\infty({\R})$-mean-periodic or $\mathbf{S}(\R)^*$-mean-periodic, 
then $\zeta_\mathcal{E}(s)^2$ extends  to a meromorphic function on $\C$, which satisfies the functional equation
\begin{equation*}
\left(c_\mathcal{E}^{-2}\right)^{s/2}\zeta_\mathcal{E}^2(s)=\left(c_\mathcal{E}^{-2}\right)^{(2-s)/2}\zeta_\mathcal{E}^2(2-s).
\end{equation*}
\end{itemize}
\end{theorem}
\begin{remark}
The first part of the previous theorem justifies the hypothesis on mean-periodicity
of  $H_{\mathcal{E}}^{(2)}(t)$  suggested in 
\cite[Section 47]{Fe2} and \cite[Section 7.3]{Fe3}
\end{remark}
\begin{remark}
See the recent work  \cite{Su2} for a description of the  convolutor of the  $h_\mathcal{E}^{(2)}(x)$
which uses the Soul\'e  extension \cite{So} of the theory of Connes \cite{Co}. 
Assuming  modularity of $E$, 
this work demonstrates some duality between the two dimensional commutative adelic analysis  on $\mathcal{E}$
and the theory of cuspidal automorphic adelic $GL(2)$-representations. 

\end{remark}
%
%
\subsection{Zeta functions of schemes and mean-periodicity}%
\begin{theorem}\label{schemess}
Let ${\mathcal Z}(s)$ be a complex valued function defined in $\Re{(s)}>\sigma_1$.\newline
\noindent
{\rm (I)} Assume that there exists a decomposition
${\mathcal Z}(s) = {\mathcal L}_{1}(s){\mathcal L}_{2}(s)^{-1}$
such that
\begin{itemize}
\item 
${\mathcal L}_i(s)~(i=1,2)$ are some absolutely convergent Dirichlet series in the half plane $\Re(s)>\sigma_1$,
\item
${\mathcal L}_i(s)~(i=1,2)$ have a meromorphic continuation to $\C$,
\item
there exist some ${\mathfrak q}_i>0$, $r_i \geq 1$, $\lambda_{i,j}>0$, 
${\rm Re}(\mu_{i,j}) > - \sigma_1 \lambda_{i,j}$ $(1\leq j \leq r_i)$
and $\abs{\epsilon_i}=1$ such that the function
\begin{equation*}
\widehat{\mathcal L}_i(s)
\coloneqq  \gamma_i(s) {\mathcal L}_i(s)
\coloneqq  {\mathfrak q}_i^{s/2} \prod_{j=1}^{r_i} \Gamma(\lambda_{i,j} s + \mu_{i,j}) 
{\mathcal L}_i(s),
\end{equation*}
satisfies the functional equation $\widehat{\mathcal L}_i(s)=\epsilon_i 
\overline{\widehat{\mathcal L}_i(d+1-\bar{s})}$
for some integer $d\geq 0$,
\item 
there exists a polynomial $P(s)$ such that $P(s)\widehat{\mathcal 
{L}_i}(s)~(i=1,2)$ are entire functions on the complex plane of order one, 
\item 
the logarithmic derivative of ${\mathcal L}_2(s)$ is an absolutely convergent Dirichlet series in the right-half plane $\Re(s)>\sigma_2\geq\sigma_1$.
\end{itemize}
Under the above assumptions, we define
\begin{eqnarray*}
\Lambda_{\mathcal Z}(s)  & \coloneqq & \frac{\widehat{\mathcal L}_1(s)}{\widehat{\mathcal L}_2(s)}
= \frac{\gamma_1(s)}{\gamma_2(s)} \, {\mathcal Z}(s) = \gamma(s) {\mathcal Z}(s), \\
\Lambda_{\widetilde{\mathcal{Z}}}(s)  & \coloneqq & \overline{\Lambda_{\mathcal Z}(\overline{s})}
\end{eqnarray*}
and the inverse Mellin transforms
\begin{eqnarray*}
f_{{\mathcal Z},m}(s)  & \coloneqq  & \frac{1}{2 \pi i} \int_{(c)} \Lambda_{\Q}(s)^m \Lambda_{\mathcal Z}((d+1)s) \, x^{-s} ds,  \\
f_{\widetilde{\mathcal{Z}},m}(s)  & \coloneqq  & \frac{1}{2 \pi i} \int_{(c)} \Lambda_{\Q}(s)^m \Lambda_{\widetilde{\mathcal{Z}}}((d+1)s) \, x^{-s} ds
\end{eqnarray*}
where  $c > 1/2 +w$. Then there exists an integer $m_{\mathcal Z} \in {\Z}$ such that the 
function
\begin{equation}\label{eqeqeq}
h_{{\mathcal Z}, m}(x)\coloneqq f_{{\mathcal Z},m}(x) - \epsilon x^{-1} f_{\widetilde{\mathcal Z},m} (x^{-1}),
\quad  (\epsilon = \epsilon_1  \epsilon_2^{-1})
\end{equation}
is ${\mathcal C}_{\rm poly}^\infty({\R_+^\times})$-mean-periodic and 
${\mathbf S}({\R}_+^\times)^*$-mean-periodic, and the function 
$H_{{\mathcal Z}, m}(t)\coloneqq h_{{\mathcal Z}, m}(e^{-t})$ is 
${\mathcal C}_{\exp}^\infty({\R})$-mean-periodic 
and $\mathbf{S}({\R})^*$-mean-periodic for every integer $m \geq m_{\mathcal Z}$.\newline\newline
\noindent
{\rm (II)} Conversely, suppose that there exists a meromorphic function 
$\gamma(s)$ on ${\C}$ and an integer $m$
such that
\begin{equation*}
\Lambda_{\Q}(s)^m \gamma((d+1)s) \ll_{a,b,t_0} \abs{t}^{-1-\delta} \quad 
(s=\sigma+it, ~a \leq\sigma\leq b,~\abs{t}\geq t_0)
\end{equation*}
for some $\delta>0$ for all $a \leq b$, and that the function 
$h_{{\mathcal Z},m}(x)$ defined in \eqref{eqeqeq} is $\mathcal{C}_{\rm poly}^\infty({\R_+^\times})$-mean-periodic
or $\mathbf{S}(\R_+^\times)^*$-mean-periodic, or the function  
$H_{{\mathcal Z},m}(t)\coloneqq h_{{\mathcal Z},m}(e^{-t})$ is 
$\mathcal{C}_{\exp}^\infty({\R})$-mean-periodic or ${\mathbf S}({\R})^*$-mean-periodic. Then the function  ${\mathcal Z}(s)$ extends meromorphically to ${\C}$
and satisfies the functional equation
\begin{equation} \label{feq_zeta}
\gamma(s) {\mathcal Z}(s) = \epsilon \overline{\gamma(d+1-\bar{s})} \,  \overline{{\mathcal Z}(d+1-\bar{s})}.
\end{equation}
\end{theorem}
\begin{remark}
If $\mathcal{Z}(s)$ is real-valued on the real line, namely the Dirichlet coefficients of $\mathcal{L}_i(s)$ ($i=1,2$) are real, then $\overline{\mathcal{Z}(\bar{s})}=\mathcal{Z}(s)$ for any complex number $s$ and $\mathcal{Z}(s)$ is self-dual with $\epsilon=\pm 1$. This is the case when studying the  zeta functions of arithmetic schemes.
\end{remark}
\begin{remark}
Theorem \ref{schemess} can be applied to the study of the  zeta functions
of arithmetic schemes. Let $S$ be a scheme of dimension $d+1$ proper flat over ${\rm Spec} \, \Z$ with smooth  generic fibre. 
Denote by  $ {\mathcal Z}_S(s) $ the rescaled  zeta function $\zeta_S(s)$ defined in Section 5.1 such that its functional equation is with respect to $s\to 1-s$. 
The function ${\mathcal Z}_S(s)$ can be canonically written as the quotient of two meromorphic functions with finitely many poles ${\mathcal L}_1(s){\mathcal L}_2(s)^{-1}$, and the numerator and denominator are factorized into the product of certain $L$-factors. 
In particular,  all the assumptions of the  theorem above follow from  the much stronger standard conjectures on the $L$-factors of the zeta functions, see e.g.  Serre (\cite{Se}). 
The last condition about the absolute convergence of the logarithmic derivative of the denominator follows from the fact that the denominator of the Hasse zeta function is the product of Euler factors $(1-\alpha_{1}(p)p^{-s})^{-1} \dots (1-\alpha_{d+1}(p)p^{-s})^{-1}$, 
such that $\abs{\alpha_i(p)} < p^{a}$ for some $a \geq 0$,
for almost all $p$ and finitely many factors  $(1-g_p(p^{-s}) )^{-1}$ with polynomials $g_p$ of degree not exceeding $d+1$.\newline\newline
Thus, we have a correspondence between zeta functions of $S$ which admit 
meromorphic continuation and functional equation of the type expected in 
number theory and mean-periodic functions in certain functional spaces: for each such 
${\mathcal Z}_S(s)$ and every $m\geq m_{{\mathcal Z}_S}$ we get the  function $h_{{\mathcal Z}_S,m}$. \newline\newline
\noindent More generally, the previous theorem can be applied to the class of functions
closed with respect to product and quotient
and generated by the rescaled zeta functions (and $L$-functions) of arithmetic schemes.
\end{remark}
\begin{remark}
Note that the function  $h_{{\mathcal Z}_S,m}$ preserves the information
about poles of the zeta function and essentially about zeros of
the denominator ${\mathcal L}_2(s)$, but not the information about zeros of the zeta function. We can also apply the previous theorem to the function 
${\mathcal Z}_S(s)^{-1}$ then the corresponding 
$h_{{\mathcal Z}_S^{-1}, m}(x)$ will preserve information
about zeros of $\zeta_S(s)$. \newline
\noindent In dimension two the numerator of  the zeta function of a regular model of
a curve over a number field is the product of one-dimensional zeta 
functions, whose meromorphic continuation and functional equation is the one 
dimensional theory. The two-dimensional object is actually the denominator of the zeta 
function and conjectural mean-periodicity of the associated function $h$
 implies  the conjectural meromorphic continuation and functional equation 
of the  denominator of the zeta function. 
One can imagine a more general recursive procedure, applied to zeta 
functions
of arithmetic schemes of dimension $d+1$ assuming the knowledge of meromorphic 
continuation
and functional equation in smaller dimensions. 
\end{remark}
\begin{proof}[\proofname{} of Theorem~\ref{schemess}]
We show only (I), since (II) is a consequence of the general theory of mean-periodicity as above.\newline
\noindent We prove that there exists $m_{\mathcal Z} \in {\Z}$ such that 
for every  $m \geq m_{\mathcal Z}$ 
the function $\Lambda_{\Q}(s)^m \Lambda_{\mathcal Z}((d+1)s)$ and $\Lambda_{\Q}(s)^m \Lambda_{\widetilde{\mathcal Z}}((d+1)s)$ 
are in the class $\mathcal{F}$ (see Definition \ref{def_E}) and then we obtain (I) by applying Theorem \ref{thm_302} and Theorem \ref{thm_303} 
to $\Lambda_{\Q}(s)^m \Lambda_{\mathcal Z}((d+1)s)$ and $\Lambda_{\Q}(s)^m \Lambda_{\widetilde{\mathcal Z}}((d+1)s)$ 
with entire functions 
\begin{eqnarray*}
U_1(s) & = & \Lambda_{\Q}(s)^m P(s)\widehat{{\mathcal L}}_1(s), \\
U_2(s) & = & \Lambda_{\Q}(s)^m \overline{P(\overline{s})}\overline{\widehat{{\mathcal L}}_1(\overline{s})}, \\
V_1(s) & = & P(s)\widehat{{\mathcal L}}_2(s), \\
V_2(s) & = & \overline{P(\overline{s})}\overline{\widehat{{\mathcal L}}_2(\overline{s})}.
\end{eqnarray*}
\noindent By the assumption of the theorem the function  $\widehat{{\mathcal L}}_1(s)$ has finitely many poles. Using the Dirichlet series representation of ${\mathcal L}_2(s)$ we get 
${\mathcal L}_2(\sigma+it)= a_{n_1} n_1^{-\sigma}(1+o(1)))$
as $\Re(s)=\sigma \to \infty$,  for some integer $n_1$ with a non-zero 
Dirichlet coefficient $a_{n_1}$.
Therefore ${\mathcal L}_2(s) \not=0$ in some right-half plane $\Re(s) > c_2 \geq \sigma_1$.
Since $\gamma_2(s)$ does not vanish for $\Re(s)>\sigma_1$,
the function $\widehat{{\mathcal L}}_2(s)$ has no zeros in the right-half 
plane $\Re(s) > c_2 \geq \sigma_1$.
By the functional equation,  $\widehat{{\mathcal L}}_2(s) \not=0$ in the 
left-half plane $\Re(s) < d + 1 -c_2$.
Hence all zeros of $\widehat{{\mathcal L}}_2(s)$ are in the vertical strip 
$d+1-c_2 \leq \Re(s) \leq c_2$.
Thus all poles of $\Lambda_{\Q}(s)^m \Lambda_{\mathcal Z}((d+1)s)$ are in some vertical strip, say $\abs{s-1/2} \leq w$.\newline
\noindent We choose
\begin{eqnarray*}
\gamma(s) & = & \Lambda_{\Q}(s)^m \gamma_1(s) \gamma_2(s)^{-1}, \\
D(s) & = & {\mathcal L}_1(s) {\mathcal L}_2(s)^{-1}
\end{eqnarray*}
as a decomposition of  $\Lambda_{\Q}(s)^m \Lambda_{\mathcal Z}((d+1)s)$ 
in definition \ref{def_E}.\newline
\noindent Using Stirling's formula we have \eqref{313} if $m$ is sufficiently 
large.
Using the Dirichlet series of ${\mathcal L}_i(s)~(i=1,2)$ we have 
\eqref{eq_AC},
if necessary, by replacing the above $w$ by a larger real number. \newline
\noindent For \eqref{314}, we first  prove that
\begin{itemize}
\item ${\mathcal L}_1(s)$ is polynomially bounded in vertical strips $a \leq 
\Re(s) \leq b$ for all  $a \leq b$,
\item there exists a real number $A$ and a strictly increasing sequence of positive real numbers $\{t_m\}$ 
satisfying ${\mathcal L}_2(\sigma \pm it_m)^{-1} \ll t_m^A$ uniformly for $\abs{(d+1)(\sigma -1/2)} \leq 1/2+w+\delta$.
\end{itemize}
The first assertion is obtained by a standard convexity argument. 
The function ${\mathcal L}_1(s)$ is bounded 
in the half-plane $\Re(s) > \sigma_1+\varepsilon$ by the absolute
convergence of the Dirichlet series. 
Hence we have a polynomial bound for ${\mathcal L}_1(s)$
in the half-plane $\Re(s) < d+1-\sigma_1-\varepsilon$ 
by the functional equation and Stirling's formula.
Then the polynomial bound in the remaining strip follows by the 
Phragmen-Lindel\"of principle. From the assumptions for ${\mathcal L}_2(s)$ we have
\begin{itemize}
\item the number of zeros $\rho = \beta + i\gamma$ of ${\mathcal L}_2(s)$ 
such that $\abs{\gamma - T} \leq 1$, say $m(T)$,
satisfies $m(T) \ll \log T$ with an implied constant depends only on 
$\gamma_2(s)$,
\item there exists $c_2^\prime \geq c_2$ such that
\begin{equation*}
\frac{{\mathcal L}_2^\prime(s)}{{\mathcal L}_2(s)}
= \sum_{\abs{t-\gamma}<1} \frac{1}{s-\rho} +O(\log{\abs{t}})
\end{equation*}
for all $s=\sigma+it$ with $d+1-c_2^\prime \leq \sigma \leq c_2^\prime$, $\abs{t} 
\geq t_0$,
where the sum runs over all zeros $\rho=\beta+i\gamma$
such that $d+1-c_2 \leq \beta \leq c_2$ and $\abs{\gamma -t}<1$
\end{itemize}
as an application of Proposition \ref{ap_prop_1}. 
The above two claims allow us to prove the above assertion for ${\mathcal L}_2(s)^{-1}$ as an application of Proposition \ref{lem_601},
if necessary, by replacing $w$ by a larger real number. 
Combining the polynomial bounds for ${\mathcal L}_1(s)$ and ${\mathcal L}_2(s)^{-1}$ we obtain \eqref{314}.\newline
\noindent Hence we find that $\Lambda_{\Q}(s)^m \Lambda_{\mathcal Z}((d+1)s)$ is in $\mathcal F$ if $m$ is sufficiently large. 
Finally, we show \eqref{323} for $U_1(s)$
and $V_1(s)$ for instance, having in mind that the same argument gives \eqref{323} for $U_2(s)$ and $V_2(s)$. Because $\zeta(s)$, $P(s)$ and ${\mathcal L}_1(s)$ are polynomially bounded in every vertical strip, 
the decomposition $U_1(s)=\fGamma_\mathbb{R}^m(s)\gamma_1(s)\cdot\zeta^mP(s){\mathcal L}_1(s)$ and 
Stirling's formula give \eqref{323} for $U_1(s)$. 
Similarly, ${\mathcal L}_2(s)$ is polynomially bounded in every vertical strip. 
By Stirling's formula and $V_1(s)=\gamma_2(s)P(s){\mathcal L}_2(s)$, 
we have \eqref{323} for $V_2(s)$. 
\end{proof}
%
%
\section{Other examples of mean-periodic functions}\label{sec_others}%
%
\subsection{Dedekind zeta functions and mean-periodicity}\label{sec_5_1}%
In this part we apply the general results in Section 4 to the Dedekind zeta functions.
Let $\K$ be a number field and
\begin{equation*}
Z_{\K}(s)=\frac{\Lambda_{\K}(2s)\Lambda_{\K}(2s-1)}{\Lambda_{\K}(s)},
\end{equation*}
which satisfies the functional equation $Z_{\K}(s)=Z_{\K}(1-s)$.
\begin{proposition} \label{prop_502}
Let $\K$ be a number field. The function
\begin{equation*}
h_{\K}(x)\coloneqq f_{Z_{\K}}(x)-x^{-1}f_{Z_{\K}}(x^{-1})
\end{equation*}
is $\mathcal{C}_{\rm poly}^\infty({\R_+^\times})$-mean-periodic and $\mathbf{S}(\R_+^\times)^*$-mean-periodic,
where $f_{Z_{\K}}$ is the inverse Mellin transform of $Z_{\K}$ defined in \eqref{315}. Moreover, the single sign property for $h_{\K}$ and the non-vanishing of $\Lambda_{\K}$ on the real line imply that all the poles of $Z_{\K}(s)$ lie on the critical line $\Re{(s)}=1/2$.
\end{proposition}
\begin{remark}
Equivalently, the function $H_\K(t)\coloneqq h_{\K}(e^{-t})$ is $\mathcal{C}_{\rm exp}^\infty({\R})$-mean-periodic and $\mathbf{S}(\R)^*$-mean-periodic.
\end{remark}
\begin{proof}[\proofname{} of Proposition~\ref{prop_502}]
We can decompose $Z_{\K}(s)$ in $Z_{\K}(s)=\gamma(s)D(s)$ where
\begin{eqnarray*}
\gamma(s) & \coloneqq  & \abs{d_{\K}}^{\frac{3s-1}{2}}
\frac{\zeta_{{\K},\infty}(2s)\zeta_{{\K},\infty}(2s-1)}{\zeta_{{\K},\infty}(s)}, \\
D(s) & \coloneqq  & \frac{\zeta_{\K}(2s)\zeta_{\K}(2s-1)}{\zeta_{\K}(s)}.
\end{eqnarray*}
The function $Z_\K(s)$ belongs to $\mathcal{F}$ since all its poles are in the critical strip $\abs{\Re{(s)}-1/2}\leq 1/2$ ($w=1/2$) and
\begin{itemize}
\item
the estimate \eqref{313} follows from Stirling's formula namely
\begin{equation*}
\forall\sigma\in[a,b], \forall \abs{t}\geq t_0,\quad\gamma(\sigma+it)\ll_{a,b,t_0}
\left(\abs{t}^{\frac{3}{2}\sigma-1}e^{-\frac{3\pi}{4}\abs{t}}\right)^{r_1}\left(\abs{t}^{3\sigma-\frac{3}{2}}e^{-\frac{3\pi}{2}\abs{t}}\right)^{r_2}
\end{equation*}
for all  real numbers $a\leq b$,
\item
the estimate \eqref{eq_AC} follows from classical convexity bounds for Dedekind zeta functions given in \eqref{eq_conv_Zeta},
\item
the crucial condition \eqref{314} is an application of Proposition \ref{lem_602} and the convexity bounds for the Dedekind zeta function given in \eqref{eq_conv_Zeta}.
\end{itemize}
Adopting the same notations as in Theorem \ref{thm_303}, we choose $U_1(s)=U_2(s)=U(s)$ and $V_1(s)=V_2(s)=V(s)$ where
\begin{eqnarray*}
U(s) & \coloneqq  & s(s-1)(2s-1)^2\Lambda_{\K}(2s)\Lambda_{\K}(2s-1),  \\
V(s) & \coloneqq  & s(s-1)(2s-1)^2\Lambda_{\K}(s).
\end{eqnarray*}
$U(s)$ and $V(s)$ are some entire functions satisfying the functional equations
\begin{eqnarray*}
U(s) & = & U(1-s), \\
V(s) & = & V(1-s).
\end{eqnarray*}
The estimate \eqref{323} is a consequence of Stirling's formula and convexity bounds for Dedekind zeta functions given in \eqref{eq_conv_Zeta}. Mean-periodicity of $h_\K$ follows from Theorem \ref{thm_302} and Theorem \ref{thm_303}. The final assertion of the proposition is a consequence of Proposition \ref{mono}, since $Z_{\K}(s)$ is holomorphic at $s=0,1$.
\end{proof}

\begin{proposition} \label{prop_ss} 
Assume that all the zeros of $\Lambda_{\K}(s)$ lie on the line $\Re(s)=1/2$ and that all the non-real zeros are simple. If
\begin{equation} \label{est}
\sum_{\substack{\Lambda_{\K}(1/2+i\gamma)=0 \\
0<\gamma<T}}\abs{\zeta_{\K}(1/2+i\gamma)}^{-1}\ll_A e^{AT}
\end{equation}
for any positive real numbers $A>0$ then the function $h_{\K}(x)$ defined in Proposition \ref{prop_502} has a single sign on $(0,x_0)$ for some $x_0>0$.
\end{proposition}
\noindent
\begin{remark} In the case of the Riemann zeta function, it is conjectured that
\begin{equation*}
\sum_{0<\gamma<T} \abs{\zeta(1/2+i\gamma)}^{-2k} \ll_k T(\log{}T)^{(k-1)^2}
\end{equation*}
for any real $k\in\R$ by Gonek~\cite{MR1014202} and Hejhal~\cite{MR993326}, under the Riemann hypothesis and assuming that all the zeros of $\zeta(s)$ are simple. Assumption \eqref{est} is quite weaker than this conjectural estimate. 
\end{remark}
\begin{proof}[\proofname{} of Proposition~\ref{prop_ss}] 
Suppose that $\Lambda_{\K}(1/2)\not=0$ for simplicity. The case $\Lambda_{\K}(1/2)=0$ is proved by a similar argument.  $Z_{\K}(s)$ has a double pole at $s=1/2$, and all the other poles are simple. Applying Theorem \ref{thm_302} to $Z_{\K}(s)$ we have
\begin{equation*}
x^{1/2} h_{\K}(x) =
c_1 \log x + c_0 + \lim_{T \to \infty} \sum_{ 0<\abs{\gamma}\leq T} c_\gamma x^{-i\gamma},
\end{equation*}
by \eqref{eq_series}, where $\{\gamma\}$ is the set of all imaginary parts of the zeros of $\Lambda_{\K}(s)$. Here $c_0$ and $c_1$ are given by 
\begin{equation*}
c_1 \log x + c_0 = \frac{C_{\K}^2}{4\Lambda_{\K}(1/2)^2}\left( \Lambda_{\K}(1/2) \log x + \Lambda_{\K}^\prime(1/2) \right),
\end{equation*}
where
\begin{equation*}
\Lambda_{\K}(s) = \frac{C_{\K}}{s-1} + A_{\K} + O(s-1) =  -\frac{C_{\K}}{s} + A_{\K} + O(s). 
\end{equation*}
Hence, in particular, $c_1$ is a non-zero real number. As for $c_\gamma$, we have
\begin{equation*}
c_\gamma = \frac{\Lambda_{\K}(2i\gamma)\Lambda_{\K}(1+2i\gamma)}{\zeta_{\K,\infty}(1/2+i\gamma)\zeta_{\K}^\prime(1/2+i\gamma)}=
\frac{\zeta_{\K,\infty}(2i\gamma)\zeta_{\K,\infty}(1+2i\gamma)}{\zeta_{\K,\infty}(1/2+i\gamma)}
\frac{\zeta_{\K}(2i\gamma)\zeta_{\K}(1+2i\gamma)}{\zeta_{\K}^\prime(1/2+i\gamma)}.
\end{equation*}
Using Stirling's formula, assumption \eqref{est} and the convex bound \eqref{eq_conv_Zeta} of $\zeta_{\K}(s)$, we have
\begin{equation*}
\sum_{\abs{\gamma}>T}\abs{c_\gamma}\ll_N   T^{-N}
\end{equation*}
for any positive real number $N>0$. Hence we have
\begin{equation*}
\left\vert\lim_{T \to \infty} \sum_{ 0<\abs{\gamma}\leq T} c_\gamma x^{-i\gamma}\right\vert\leq
\lim_{T \to \infty} \sum_{ 0<\abs{\gamma}\leq T}\abs{c_\gamma}=O(1)
\end{equation*}
uniformly for every $x \in (0,\infty)$. Thus
\begin{equation*}
x^{1/2} h_{\K}(x) = c_1 \log x +O(1) 
\end{equation*}
with $c_1\neq 0$ and for all $x \in (0,\infty)$. This implies $h_{\K}(x)$ has a single sign near $x=0$. 
\end{proof}
%
\subsection{Cuspidal automorphic forms and mean-periodicity}\label{sec_cusp}%
The following proposition provides some evidence for the fact that mean-periodic functions also appear in the context of automorphic cuspidal representations. The analytic background is available in \cite{MR2331346}. Let us say that if $\pi$ is an automorphic cuspidal irreducible representation of $GL_{m}(\A_\Q)$ with unitary central character then its completed $L$-function $\Lambda(\pi,s)$ satisfies the functional equation
\begin{equation*}
\Lambda(\pi,s)=\epsilon_{\pi}\Lambda(\widetilde{\pi},s)
\end{equation*}
where $\epsilon_{\pi}$ is the sign of the functional equation and $\widetilde{\pi}$ is the contragredient representation of $\pi$.
\begin{proposition}\label{prop_mp_auto}
Let $\varpi_1,\cdots,\varpi_n$ in $\left\{\pm1\right\}$ and $\pi_1,\cdots,\pi_n$ some automorphic cuspidal irreducible representations of $GL_{m_1}(\A_\Q),\cdots,GL_{m_n}(\A_\Q)$ with unitary central characters. There exists an integer $m_0\geq 0$ such that for every  integer $m\geq m_0$, the function
\begin{equation*}
h_m(x)\coloneqq f_{Z_m}(x)-\left(\prod_{i=1}^n\epsilon_{\pi_i}^{\varpi_i}\right)x^{-1}f_{\widetilde{Z}_m}\left(x^{-1}\right)
\end{equation*}
is $\mathcal{C}_{\rm poly}^\infty({\R_+^\times})$-mean-periodic and $\mathbf{S}(\R_+^\times)^*$-mean-periodic, where
\begin{eqnarray*}
Z_m(s) & \coloneqq  & \Lambda_\Q(s)^m\prod_{i=1}^n\Lambda(\pi_i,s)^{\varpi_i}, \\
\widetilde{Z}_m(s) & \coloneqq  & \Lambda_\Q(s)^m\prod_{i=1}^n\Lambda(\widetilde{\pi}_i,s)^{\varpi_i}, \\
\end{eqnarray*}
and $f_{Z_m}$ (respectively $f_{\widetilde{Z}_m}$) is the inverse Mellin transform of $Z_m$ (respectively $\widetilde{Z}_m$) defined in \eqref{315}.
\end{proposition}
\begin{remark}
Equivalently, the function $H_m(t)=h_m(e^{-t})$ is $\mathbf{S}(\R)^*$-mean-periodic and $\mathcal{C}_{\rm exp}^\infty({\R})$-mean-periodic.
\end{remark}
\begin{remark}
We would like to focus on the fact that, unlike for zeta-functions of schemes, the objects are not necessarily self-dual but the general results on mean-periodicity proved in Section \ref{sec_results} are still applicable in this context.
\end{remark}
\begin{remark}
The proof of the previous proposition is omitted since it is an immediate application of Theorems \ref{thm_302} and \ref{thm_303}. Proving that $Z_m(s)$ belongs to $\mathcal{F}$ requires the \emph{convexity bounds} for general $L$-functions of $GL_n$ given in \cite[Section 1.3]{MR2331346} and, 
of course, the use of Proposition \ref{lem_602}.
\end{remark}
%
\subsection{Eisenstein series and mean-periodicity}\label{sec_Eis}%
In this section, we construct several continuous families of mean-periodic functions, the main tool being Eisenstein series. For simplicity, let us restrict ourselves to $\K=\Q$. Let $\mathfrak{h}$ be the upper-half plane. The non-holomorphic Eisenstein series attached to the full modular group
$\gGamma={\rm PSL}(2,{\Z})$ is defined by
\begin{equation*}
\widehat{E}(\tau,s)= \Lambda_{\Q}(2s)E(\tau,s)=\Lambda_{\Q}(2s)\sum_{\left( \smallmatrix \ast & \ast \\ c & d \endsmallmatrix\right)\in\gGamma_\infty\backslash\gGamma}\frac{y^s}{\abs{c\tau+d}^{2s}}
\end{equation*}
for $\tau=x+iy \in \mathfrak{h}$ and $\Re{(s)}>1$, where $\gGamma_\infty=\left\{\left(\smallmatrix \ast & \ast \\ 0 & \ast \endsmallmatrix \right)\right\} \cap \Gamma$. For a fixed $\tau \in \mathfrak{h}$,
$\widehat{E}(\tau,s)$ has a meromorphic continuation to $\C$ with simple poles at $s=0,1$ and satisfies the functional equation
$\widehat{E}(\tau,s)= \widehat{E}(\tau, 1-s)$. Fix a $\tau \in \mathfrak{h}$, then 
\begin{eqnarray*}
Z_{\Q}(\tau,s) & \coloneqq  & \frac{\widehat{E}(\tau,s)}{\Lambda_{\Q}(s)},\\
Z_{\Q}^{\vee}(\tau,s) & \coloneqq  & \frac{\Lambda_{\Q}(2s)\Lambda_{\Q}(2s-1)}{\widehat{E}(\tau,s)}
\end{eqnarray*}
satisfy the functional equations $Z_{\Q}(\tau,s)=Z_{\Q}(\tau,1-s)$ and $Z_{\Q}^{\vee}(\tau,s)=Z_{\Q}^{\vee}(\tau,1-s)$.
\begin{proposition} \label{prop_503}
The functions
\begin{eqnarray*}
h_{\Q}(\tau,x) & \coloneqq  & f_{Z_{\Q}(\tau,.)}(x)-x^{-1}f_{Z_{\Q}(\tau,.)}\left(x^{-1}\right), \\
h_{\Q}^{\vee}(\tau,x) & \coloneqq  & f_{Z_{\Q}^{\vee}(\tau,.)}(x)-x^{-1}f_{Z_{\Q}^{\vee}(\tau,.)}\left(x^{-1}\right)
\end{eqnarray*}
are $\mathcal{C}_{\rm poly}^\infty({\R_+^\times})$-mean-periodic and $\mathbf{S}(\R_+^\times)^*$-mean-periodic, where $f_{Z_{\Q}(\tau,.)}$ (respectively $f_{Z_{\Q}^\vee(\tau,.)}$) is the inverse Mellin transform of $Z_{\Q}(\tau,s)$ (respectively $Z_{\Q}^\vee(\tau,s)$) defined in \eqref{315}. Moreover, if $x\mapsto h_{\Q}(\tau,x)$ does not identically vanish then the single sign property for $h_{\Q}(\tau,x)$ implies that all the poles of $Z_{\Q}(\tau,s)$ lie on the critical line $\Re{(s)}=1/2$.
\end{proposition}
\begin{proof}[\proofname{} of Proposition~\ref{prop_503}]%
Let us only prove that both $Z_{\Q}(\tau,s)$ and $Z_{\Q}^{\vee}(\tau,s)$ belong to $\mathcal{F}$ since the results are an application of Theorem \ref{thm_302}, Theorem \ref{thm_303} and Proposition \ref{mono}. Note that $\Z_{\Q}(\tau,s)$ is regular at $s=0,1$ and $\Lambda_{\Q}(s)\not=0$ on the real line. Let us focus on $Z_{\Q}^{\vee}(\tau,s)$ since the case of $Z_{\Q}(\tau,s)$ is very similar to the case of $Z_{\Q}(s)$. We can decompose  in $Z_{\Q}^{\vee}(\tau,s)=\gamma(\tau,s)D(\tau,s)$ where
\begin{eqnarray*}
\gamma(\tau,s) & \coloneqq  & \Gamma_{\R}(2s-1), \\
D(\tau,s) & \coloneqq  & \frac{\zeta_{\Q}(2s-1)}{E(\tau,s)}.
\end{eqnarray*}
By Stirling's formula, we have
\begin{equation*}
\gamma_{\Q}^{\vee}(\tau,s)\ll_{a,b,t_0}\abs{t}^{\sigma-1} e^{-\frac{\pi}{2}\abs{t}}
\end{equation*}
for all  real numbers $a \leq b$, all  $\sigma$ in $[a,b]$ and all  $\abs{t}\geq t_0$.
For a fixed $\tau \in \mathfrak{h}$, ${\Z}+{\Z}\tau$ is a lattice in $\C$.
Thus, the image of $(m,n)\mapsto\abs{m\tau+n}^2$ is discrete in ${\R}_+^\times \cup \{0\}$. 
We arrange the distinct values of its image as
\begin{equation*}
0=c_\tau(0) < c_\tau(1) < c_\tau(2) < \dots \to \infty,
\end{equation*}
and define
$N_\tau(k)=\abs{\left\{(m,n)\in\Z\times\Z, \abs{m\tau+n}^2 = c_\tau(k)\right\}}$.
Then,
\begin{equation*}
E(\tau,s)=y^{s} \sum_{k \geq 1} \frac{N_\tau(k)}{c_\tau(k)^{s}}
= N_\tau(1) (y \, c_\tau(1)^{-2})^s (1+o(1))
\end{equation*}
as $\Re{(s)}\to\infty$ since $N_\tau(k) \ll_\tau c_\tau(k)$ uniformly for all  fixed $\tau \in \mathfrak{h}$, and $\sum_{c_\tau(k) \leq T} c_\tau(k) \ll_\tau T^2$. 
Thus, $E(\tau,s)\not=0$ for $\Re(s) \gg_\tau 0$, and $E(\tau,s)^{-1}$ is uniformly bounded in every vertical strip contained in some right-half plane.
In other words, there exists $\sigma_\tau \geq 1$ such that
\begin{equation*}
D_{\Q}^{\vee}(\tau,\sigma+it) \ll \abs{t}^{A_1}
\end{equation*}
uniformly in every vertical strip contained in the right-half plane $\Re(s) > \sigma_\tau$ for some real number $A_1$. In addition, we have
\begin{equation*}
\frac{E^\prime(\tau,s)}{E(\tau,s)} = y^s \, E(\tau,s)^{-1}
\sum_{k \geq 1} \frac{N_\tau(k)\left( \log y -\log c_\tau(k) \right) }{c_\tau(k)^{s}}
\end{equation*}
for $\Re(s) \gg_\tau 0$, where $E^\prime(\tau,s) = \frac{d}{ds}E(\tau,s)$.
Hence $E^\prime(\tau,s)/E(\tau,s)$ is bounded in every vertical strip contained in the right-half plane $\Re(s) \geq \sigma_\tau^\prime \geq \sigma_\tau$.
Now we obtain
\begin{equation*}
\frac{E^\prime(\tau,s)}{E(\tau,s)}
= \sum_{\substack{E(\tau,\rho)=E(\tau,\beta+i\gamma)=0 \\
1-\sigma_\tau < \beta < \sigma_\tau \\
\abs{t-\gamma}<1}}\frac{1}{s-\rho}+O_\tau(\log t)
\end{equation*}
uniformly for $-\sigma_\tau^\prime\leq \sigma \leq 1 + \sigma_\tau^\prime$ and $t \geq 2$,
and
\begin{equation*}
\sum_{\substack{E(\tau,\rho)=E(\tau,\beta+i\gamma)=0 \\
1-\sigma_\tau < \beta < \sigma_\tau \\
\abs{t-\gamma}<1}}1 = O_\tau(\log t)
\end{equation*}
with an implied constant depending only on $\tau$
following the same lines as used  in the proof of Proposition \ref{lem_601},
since $s(s-1)\widehat{E}(\tau,s)$ is an entire function of order one
~\footnote{$\widehat{E}(\tau,s)$ is not a $L$-function in the sense of \cite{IK},
but the proof of \cite[Proposition 5.7]{IK} only requires that
$(s(s-1))^r L(f,s)$ is an entire function of order one and $L^\prime/L(s)$ is bounded
with respect to the conductor in every vertical strip contained in some right-half plane.}. 
These two facts entail the existence of a sequence of positive real numbers$\{t_m\}_{m \geq 1}$ and of some real number $A$ such that
$E(\tau,\sigma+it_m)^{-1} = O_\tau(t_m^A)$
uniformly for $-\sigma_\tau^\prime \leq \sigma \leq 1 + \sigma_\tau^\prime$.
Hence there exists a sequence of positive real numbers $\{t_m\}_{m \geq 1}$ and a real number $A_2$ such that
\begin{equation*}
\abs{D_{\Q}^{\vee}(\tau,\sigma+it_m)}\ll_\tau t_m^{A_2}
\end{equation*}
uniformly for $\sigma \in [-\sigma_\tau^\prime,1+\sigma_\tau^\prime]$ and for every  integer $m \geq 1$.
The above estimates imply $Z_{\Q}^{\vee}(\tau,s)$ is an element of $\mathcal{F}$.
\end{proof}
\begin{remark}
If $\tau \in \mathfrak{h}$ is a generic point then ${\mathsf{MC}}(h_{\Q}(\tau,x))$ is expected to have infinitely many poles whereas if $\tau \in \mathfrak{h}$ is a special point then ${\mathsf{MC}}(h_{\Q}(\tau,x))$ does not have  poles since
\begin{equation*}
E(\tau,s)= a b^s \zeta_{\Q}(s)L\left(s,\left(\frac{D}{\cdot}\right)\right)
\end{equation*}
for some positive real numbers $a$, $b$ at certain CM-point $\tau$ 
(see \cite{MR633666}). Thus, $h_{\Q}(\tau,x)$ identically vanishes in the second case according to \eqref{eq_series}.
\end{remark}
\begin{remark}
It is known that $\widehat{E}(\tau,s)$ has infinitely many zeros outside the critical line 
(see \cite{MR0030997}, \cite{DaHe}, \cite{DaHe2} and \cite{MR0215798}). Hence it is expected that the behaviour of the families $\{h_{Q}(\tau,x)\}_{\tau \in \mathfrak{h}}$ and $\{h_{Q}^{\vee}(\tau,x)\}_{\tau \in \mathfrak{h}}$ are quite different. The comparison of these two families is an interesting problem, from a number theoretical point of view. In particular, we would like to have information on the single sign property for both families.
\end{remark}
%
%
%
%
\appendix%
\section{A useful analytic estimate for $L$-functions}\label{ap_analytic}%
Proposition \ref{lem_601}  below is used several times in this paper. 
Proposition \ref{lem_601} is deduced from Proposition \ref{ap_prop_1} 
which holds for the ``$L$-functions'' defined in \cite[Section 5.1]{IK} 
and is a slight extension of \cite[Proposition 5.7]{IK}. 
In particular,  Proposition \ref{lem_601} holds for the $L$-functions in \cite[Section 5.1]{IK}. 
To prove Proposition \ref{ap_prop_1} and Proposition \ref{lem_601} we use the Hadamard product of  order one entire functions, 
Stirling's formula and the boundedness of the logarithmic derivative of a function in a vertical strip contained in a right-half plane.
\begin{proposition} \label{ap_prop_1}
Let ${\mathcal L}(s)$ be a complex valued function defined in the right 
half plane ${\rm Re}(s)>\sigma_1$. 
Assume that 
\begin{itemize}
\item ${\mathcal L}(s)$ is expressed as an absolutely convergent 
Dirichlet series in the right-half plane ${\rm Re}(s)> \sigma_1$,
\item ${\mathcal L}(s)$ has a meromorphic continuation to $\C$, 
\item there exist some ${\mathfrak q}>0$, $r \geq 1$, $\lambda_{j}>0$, 
${\rm Re}(\mu_{j}) > - \sigma_1 \lambda_{j}$ $(1\leq j \leq r)$
and $\abs{\epsilon}=1$ such that the function
\begin{equation*}
\widehat{\mathcal L}(s)
\coloneqq  \gamma(s) {\mathcal L}(s)
\coloneqq  {\mathfrak q}^{s/2} \prod_{j=1}^{r} \Gamma(\lambda_{j} s + \mu_{j}) 
{\mathcal L}(s),
\end{equation*}
satisfies the functional equation $\widehat{\mathcal L}(s)=\epsilon 
\overline{\widehat{\mathcal L}(d+1-\bar{s})}$
for some integer $d\geq 0$ 
(the condition ${\rm Re}(\mu_{j}) > - \sigma_1 \lambda_{j}$ tells us that $\gamma(s)$ has no zeros in $\C$ and no poles for $\Re(s) \ge \sigma_1$). 
\item there exists a polynomial $P(s)$ such that $P(s)\widehat{\mathcal L}(s)$ is an entire function on the complex plane of order one, 
\item the logarithmic derivative of ${\mathcal L}(s)$ is expressed as an absolutely convergent 
Dirichlet series
\begin{equation*}
-\frac{{\mathcal L}^\prime(s)}{{\mathcal L}(s)} = \sum_{n=1}^\infty \frac{\Lambda_{\mathcal L}(n)}{n^s}
\end{equation*}
in the right-half plane $\Re(s) > \sigma_2 \geq \sigma_1$. 
\end{itemize}
Then 
\begin{enumerate}
\item the number of zero $\rho=\beta+i\gamma$ such that $|\gamma - T| \le 1$, 
say $m(T)$, satisfies 
\begin{equation}
m(T) = O(\log |T|)
\end{equation} 
with an implied constant depending on ${\mathcal L}(s)$ only, 
\item there exists $c \ge {\rm max}\{\sigma_1,(d+1)/2\}$ such that all zeros of $\widehat{\mathcal L}(s)$ are in the strip $d+1-c \le \Re(s) \le c$,  
\item there exists $t_0>0$ such that for any $s=\sigma+it$ in the strip $d+1-c-1 \le \sigma \le c+1$, $|t| \ge t_0 $ we have
\begin{equation}
\frac{{\mathcal L}^\prime(s)}{{\mathcal L}(s)} = \sum_{|s - \rho| < 1} \frac{1}{s-\rho} +O(\log |t|)
\end{equation}
where the sum runs over all zeros $\rho=\beta+i\gamma$ of ${\mathcal L}(s)$ such that $d+1-c \le \beta \le c$ and $|s-\rho| <1$. 
\end{enumerate}
\end{proposition}

\begin{proof}[\proofname{} of Proposition~\ref{ap_prop_1}]%
We argue  similarly  to the proof of Proposition 5.7 in \cite{IK} which is essentially the proof
in \cite[Section 9.6]{MR882550}. 
First we prove (2). Using the Dirichlet series of ${\mathcal L}(s)$ we obtain 
${\mathcal L}(\sigma+it)= a_{n_1} n_1^{-\sigma}(1+o(1))$
as $\Re(s)=\sigma \to +\infty$ for some integer $n_1$ with a non-zero Dirichlet coefficient $a_{n_1}$.
Therefore there exists $c\geq {\rm max}\{\sigma_1,(d+1)/2\}$ such that ${\mathcal L}(s) \not=0$ in the right-half plane $\Re(s) > c$.
Since $\gamma(s)$ does not vanish for $\Re(s)>\sigma_1$,
the function $\widehat{{\mathcal L}}(s)$ has no zeros in the right-half plane $\Re(s) > c \geq \sigma_1$.
By the functional equation,  $\widehat{{\mathcal L}}(s) \not=0$ in the left-half plane $\Re(s) < d + 1 -c$.
Hence all zeros of $\widehat{\mathcal L}(s)$ are in the vertical strip $d+1-c \leq \Re(s) \leq c$. 

By the assumptions there exist  constants $a$, $b$ and a nonnegative integer $m$ such that 
\begin{equation*}
P(s)\widehat{\mathcal L}(s) = (s(s-d-1))^m e^{a + bs} \prod_{\rho\not=0,d+1}\Bigl(1- \frac{s}{\rho} \Bigr)e^{s/\rho},
\end{equation*}
where $\rho$ ranges over all zeros of $P(s)\widehat{\mathcal L}(s)$ different from $0$, $d+1$. 
This is a consequence of the Hadamard factorization theorem for  entire functions of order one. 
Taking the logarithmic derivative, 
\begin{equation}\label{ap_01}
-\frac{{\mathcal L}^\prime(s)}{{\mathcal L}(s)}=
\frac{1}{2} \log {\mathfrak q} + \frac{\gamma^\prime(s)}{\gamma(s)}-b+\frac{P^\prime(s)}{P(s)}
- \frac{m}{s} - \frac{m}{s-d-1}
- \sum_{\rho} \left(\frac{1}{s-\rho}+\frac{1}{\rho} \right).
\end{equation}
We take $c^\prime \ge {\rm max}\{{\sigma_2},c\}$ such that the polynomial $P(s)$ has no zero in the right-half plane $\Re(s) \ge c^\prime$. 
Let $T \ge 2$ and $s_0=c^\prime+2+iT$. Then
\begin{equation} \label{ap_02}
\left| \frac{{\mathcal L}^\prime(s_0)}{{\mathcal L}(s_0)} \right| 
\le \sum_{n=1}^\infty \frac{|\Lambda_{\mathcal L}(n)|}{n^{c^\prime+2}}=O(1).
\end{equation}
By Stirling's formula we have 
\begin{equation*}
\frac{1}{2} \log {\mathfrak q} + \frac{\gamma^\prime(s_0)}{\gamma(s_0)} = O(\log |T|). 
\end{equation*}
For every  zero $\rho=\beta+i\gamma$ we have
\begin{equation*}
\frac{2}{(2c-d+1)^2+(T-\gamma)^2} \le \Re \left( \frac{1}{s_0-\rho} \right) \le \frac{2c-d+1}{4+(T-\gamma)^2}. 
\end{equation*}
Hence we can take the real part in \eqref{ap_01} for $s_0=c^\prime+2+iT$ and rearrange the resulting absolutely convergent series to derive that
\begin{equation} \label{ap_03}
\sum_{\rho} \frac{1}{1+(T-\gamma)^2} = O(\log |T|). 
\end{equation}
This implies (1). Here we used the fact that $\sum_{\rho\not=0,d+1}(\frac{1}{\rho}+\frac{1}{\overline{\rho}})$ converges absolutely, 
which is a consequence of the order one condition and $d+1 -c \le \Re(\rho) \le c$. 

To prove (3), we write $s=\sigma+it$. 
We suppose that $d+1-c-1 \le \Re(s) \le c^\prime+1$ 
and take $t_0$ such that $\gamma(s)$ and $P(s)$ has no pole for $|t| \ge t_0$. 
We have 
\begin{equation}
-\frac{{\mathcal L}^\prime(s)}{{\mathcal L}(s)}
= -\frac{{\mathcal L}^\prime(s)}{{\mathcal L}(s)}+\frac{{\mathcal L}^\prime(c^\prime+2+it)}{{\mathcal L}(c^\prime+2+it)} +O(\log|t|).
\end{equation} 
By \eqref{ap_01}, \eqref{ap_02} and Stirling's formula, we obtain
\begin{equation}
-\frac{{\mathcal L}^\prime(s)}{{\mathcal L}(s)}=
\frac{\gamma^\prime(s)}{\gamma(s)}+\frac{P^\prime(s)}{P(s)}
- \frac{m}{s} - \frac{m}{s-d-1}
- \sum_{\rho} \left(\frac{1}{s-\rho}-\frac{1}{c^\prime+2+it-\rho} \right) + O(\log|t|).
\end{equation}
In the series, keep the zeros with $|s-\rho|<1$ and estimate the remainder by $O(\log|t|)$ using
\begin{equation*}
\left| \frac{1}{s-\rho} -  \frac{1}{c^\prime+2+it-\rho} \right| \le \frac{2c-d+1}{1+(t-\gamma)^2}
\end{equation*} 
and \eqref{ap_03}. Moreover we have
\begin{equation*}
\frac{\gamma^\prime(s)}{\gamma(s)} =O(\log|t|) 
\end{equation*}
uniformly for $d+1-c-1 \le \Re(s) \le c^\prime+1$ and $|t| \ge t_0$
by Stirling's formula. Thus (3) follows. 
\end{proof}
Using Proposition \ref{ap_prop_1} we have the following proposition which is used several times in this paper.
\begin{proposition}\label{lem_601}
Let ${\mathcal L}(s)$ be a function satisfying the conditions in Proposition \ref{ap_prop_1}. 
Let $H>1$ and $T \geq {\rm max}\{2,t_0\}$ be some real numbers, 
and let $a<b$ be real numbers such that $a \le d+1-c-1$ and $b \ge c+1$,
where $t_0$ and $c$ are real numbers appeared in Proposition \ref{ap_prop_1}. 
Then there exists a real number $A$ and a subset $\mathcal{E}_T$ of $(T,T+1)$ such that
\begin{equation*}
\forall t\in\mathcal{E}_T,\forall\sigma\in(a,b),\quad {\mathcal L}(f,\sigma\pm it)^{-1}=O\left(t^A\right)
\end{equation*}
and
\begin{equation*}
\mu\left[(T,T+1)\setminus\mathcal{E}_T\right]\leq\frac{1}{H}
\end{equation*}
where $\mu$ stands for the Lebesgue measure on $\R$.
\end{proposition}
\begin{remark}
In the case of $\zeta(s)$, this proposition is nothing else than \cite[Theorem 9.7]{MR882550}.
\end{remark}
\begin{proof}[\proofname{} of Proposition~\ref{lem_601}]%
The following arguments are essentially as in \cite[Section 9.7]{MR882550}. 
We  use the notation $s=\sigma+it$,   $t>0$. 
It sufficies to prove for $a=d+1-c-1$ and $b=c+1$, 
since ${\mathcal L}^\prime(s)/{\mathcal L}(s)$ has absolutely convergent Dirichlet series for $\Re(s) > c  \ge \sigma_2$ 
and has a functional equation. 
We have
\begin{equation}\label{603}
\frac{{\mathcal L}^\prime(s)}{{\mathcal L}(s)}=\sum_{\substack{{\mathcal L}(\rho)={\mathcal L}(\beta+i\gamma)=0 \\
d+1-c<\beta< c \\
\vert t-\gamma\vert<1}}\frac{1}{s-\rho}+O(\log{t})
\end{equation}
uniformly for $d+1-c-1 \leq\sigma\leq c+1$ and $t\geq t_0$ by Proposition \ref{ap_prop_1} . 
The difference between the sum in \eqref{603} and the sum in Proposition \ref{ap_prop_1} (3) does not exceed $O(\log{t})$ 
according to Proposition \ref{ap_prop_1} (1). 
Assuming that ${\mathcal L}(s)$ does not vanish on $\Im{(s)}=t$ and integrating \eqref{603} from $s$ to $c+1+it$, we get
\begin{equation}\label{604}
\log{{\mathcal L}}(s)=\sum_{\substack{{\mathcal L}(\rho)={\mathcal L}(\beta+i\gamma)=0 \\
d+1-c<\beta< c \\
\vert t-\gamma\vert<1}}\log{(s-\rho)}+O(\log t),
\end{equation}
uniformly for $d+1-c-1 \leq \sigma \leq c+1$, $t \geq t_0$, where $\log{{\mathcal L}}(s)$ has its usual meaning ($-\pi<\Im{(\log{{\mathcal L}}(s))}\leq\pi$). 
The fact that the number of $\rho=\beta+i\gamma$ satisfying $\vert t-\gamma\vert<1$ is $O(\log t)$ (Proposition \ref{ap_prop_1} (1)) is used here. Taking real parts in \eqref{604}, we have
\begin{equation*}
\log{\vert {\mathcal L}(s)\vert}=\sum_{\substack{{\mathcal L}(\rho)={\mathcal L}(\beta+i\gamma)=0 \\
d+1-c<\beta<c \\
\vert t-\gamma\vert<1}}\log{\vert s-\rho\vert}+O(\log{t})\geq\sum_{\substack{{\mathcal L}(\rho)={\mathcal L}(\beta+i\gamma)=0 \\
d+1-c<\beta<c \\
\vert t-\gamma\vert<1}}\log{\vert t-\gamma\vert}+O(\log t).
\end{equation*}
One would like to integrate with respect to $t$ from $T\geq t_0$ to $T+1$ taking care of the fact that there may be zeros of ${\mathcal L}(s)$ of height between $(T,T+1)$ where the previous inequality does not hold. Let $y_1<\cdots<y_M$ be a finite sequence of real numbers satisfying
\begin{equation*}
\forall j\in\{1,\cdots,M\}, \exists \sigma\in[-1,2], \quad {\mathcal L}(\sigma+iy_j)=0.
\end{equation*}
Setting $y_0\coloneqq T$ and $Y_{M+1}\coloneqq T+1$ and assuming that ${\mathcal L}(s)$ does not vanish on $\Im{(s)}=T$ and on $\Im{(s)}=T+1$, we have
\begin{eqnarray*}
\sum_{j=0}^M\int_{y_j}^{y_{j+1}}\sum_{\substack{{\mathcal L}(\rho)={\mathcal L}(\beta+i\gamma)=0 \\
d+1-c<\beta<c \\
\vert t-\gamma\vert<1}}\log{\vert t-\gamma\vert}\dd t & = & \sum_{j=0}^M\sum_{\substack{{\mathcal L}(\rho)={\mathcal L}(\beta+i\gamma)=0 \\
d+1-c<\beta<c \\
y_j-1<\gamma<y_{j+1}+1}}\int_{\max{(\gamma-1,y_j)}}^{\min{(\gamma+1,y_{j+1})}}\log{\vert t-\gamma\vert}\dd t\\
& \geq & \sum_{j=0}^M\sum_{\substack{{\mathcal L}(\rho)={\mathcal L}(\beta+i\gamma)=0 \\
d+1-c<\beta<c \\
y_j-1<\gamma<y_{j+1}+1}}\int_{\gamma-1}^{\gamma+1} \log{\vert t-\gamma\vert}\dd t \\
& = & \sum_{j=0}^M\sum_{\substack{{\mathcal L}(\rho)={\mathcal L}(\beta+i\gamma)=0 \\
d+1-c<\beta<c \\
y_j-1<\gamma<y_{j+1}+1}}(-2) \\
& > & -A\log{T}
\end{eqnarray*}
The last inequality is a consequence of an analogue of \cite[Theorem 5.8]{IK}. 
Thus,
\begin{equation*}
\sum_{\substack{{\mathcal L}(\rho)={\mathcal L}(\beta+i\gamma)=0 \\
d+1-c<\beta<c \\
\vert t-\gamma\vert<1}}\log{\vert t-\gamma\vert}>-AH\log T
\end{equation*}
for all  $t$ in $(T,T+1)$, except for a set of Lebesgue measure $1/H$.
\end{proof}
If $L(f,s)$ is a $L$-function defined in \cite[Section 5.1]{IK}, 
Proposition \ref{lem_601} can be stated in the following form.
\begin{proposition} \label{lem_602}
Let $H>1$ and $T\geq 2$ be some real numbers. 
If $L(f,s)$ is a $L$-function in the sense of \cite{IK} then there exists a real number $A$ and a subset $\mathcal{E}_T$ of $(T,T+1)$ such that
\begin{equation*}
\forall t\in\mathcal{E}_T,\forall\sigma\in\left(-\frac{1}{2},\frac{5}{2}\right),\quad L(f,\sigma\pm it)^{-1}=O\left(t^A\right)
\end{equation*}
and
\begin{equation*}
\mu\left[(T,T+1)\setminus\mathcal{E}_T\right]\leq\frac{1}{H}
\end{equation*}
where $\mu$ stands for the Lebesgue measure on $\R$.
\end{proposition}
%
%
%
%
%

\bibliographystyle{amsplain}


\def\cprime{$'$} \def\cprime{$'$}
\providecommand{\bysame}{\leavevmode\hbox to3em{\hrulefill}\thinspace}
\providecommand{\MMR}[2]{\relax\ifhmode\unskip\space\fi }
\providecommand{\MMRhref}[2]{%
  \href{http://www.ams.org/mathscinet-getitem?mr=#1}{#2}
}
\providecommand{\href}[2]{#2}

\medskip 

{\it 

Masatoshi Suzuki \quad University of Tokyo, Japan

Guillaume Ricotta \quad University of Bordeaux, France

Ivan Fesenko \quad University of Nottingham, England

} 

\end{document}